\documentclass{amsart}
\usepackage[hyperref]{waynemath}
\usepackage{waynealgtop}
\usepackage[all, cmtip]{xy}
\usepackage{caption}

\makeatletter
	\let\c@equation\c@theorem
\makeatother
\numberwithin{equation}{section}

\makeatletter
	\renewcommand\subsection{\@startsection{subsection}{2}%
		\z@{.5\linespacing\@plus.7\linespacing}{.5\linespacing}%
		{\normalfont\bfseries}} 
	\newcommand{\myitem}[1]{%
		\item[#1]\protected@edef\@currentlabel{#1}%
		}
\makeatother

\newcommand{\makebibliography}{
	\newpage
	\bibliographystyle{plain}
	\bibliography{MaySS.bib}
}

\theoremstyle{definition}
\newtheorem{Algorithm}[theorem]{Algorithm}

\newcommand{\Ann}{\mathrm{Ann}}
\newcommand{\Groebner}{Gr\"obner}
\newcommand{\RN}[1]{\textup{\uppercase\expandafter{\romannumeral#1}}}

\title{On the May Spectral Sequence at the Prime 2}
\author{Weinan Lin}
\date{}

\begin{document}
\begin{abstract}
	We make a conjecture about the whole $E_2$ page of the May spectral sequence in terms of generators and relations and we prove it in a subalgebra which covers a large range of dimensions. 
	We show that the $E_2$ page plays a universal role in the study of Massey products in commutative DGAs.
	We conjecture that the $E_2$ page is nilpotent free and also prove it in this subalgebra. 
	We compute all the $d_2$ differentials of the generators in the conjecture and construct maps of spectral sequences which allow us to explore Adams vanishing line theorem to compute differentials in the May spectral sequence.
\end{abstract}
\maketitle
\setcounter{tocdepth}{2}
\tableofcontents

\section{Introduction}
In algebraic topology, one basic task is to compute the stable homotopy groups of objects of interest. 
Nearly all methods if not all for this task involve spectral sequences starting from some Ext groups. 
In the work of Hill, Hopkins and Ravenel \cite{HHR} on the Kervaire invariant one problem, one key ingredient of the proof is the detection theorem, which for degree reasons can be induced by an algebraic detection theorem relating the $E_2$ pages of three spectral sequences, all being $\Ext$ groups of some sort.

Despite the importance of the $\Ext$ groups, it is generally hard to compute them in all degrees, at least in the case of $\Ext_{\sA}(\bF_p, \bF_p)$ for the Adams spectral sequence. 
In modern days, the most extensive calculation of $\Ext_{\sA}(\bF_p, \bF_p)$ is done by computers based on minimal resolutions.
Bruner \cite{Bruner97}, Nassau \cite{Nassau} and Wang \cite{Wang} have implemented the  algorithms and calculated $\Ext_{\sA}(\bF_p, \bF_p)$ beyond dimension 200.
The computation of the stable homotopy groups of sphere relies on these machine outputs for the Adams $E_2$ page.
The interested reader can refer to Isaksen, Wang and Xu \cite{IWX20} for the most up to date calculation.

The spectral sequence of May \cite{May64} offers a generally useful method for calculating the $\Ext$ over a Hopf algebra (or Hopf algebroid).
The classical one for the Steenrod algebra
$$\Ext_{E^0\sA}^{***}(\bF_p, \bF_p)\Longrightarrow \Ext_{\sA}^{**}(\bF_p, \bF_p)$$
was the first effective way to compute $\Ext_{\sA}(\bF_p, \bF_p)$.
It is still probably the best way to compute $\Ext_{\sA}(\bF_p, \bF_p)$ by hand.
The author is currently working to realize it as a successful way to compute $\Ext_{\sA}(\bF_p, \bF_p)$ in very high dimensions, by implementing this method as a computer program.
This paper serves as the foundation for this project. Most methods we use are quite general and the project is intended to be a generic one for computations of spectral sequences by machines.

We are concerned with the Steenrod algebra $\sA$ over prime 2. 
Compared with the cohomology of the Steenrod algebra, the $E_2$ page of the May spectral sequence can be computed in a much larger range. 
The data provided by this paper can cover the $E_2$ page of May up to stem 285. 
It is already computed up to total degree $t=1024$. 
However, the data is too big to be put into a paper. 
The more important thing is that we have Conjecture \ref{conj:2cc43131} and Conjecture \ref{conj:5f7d758} which describe the whole $E_2$ page of May by generators and relations. 
The conjectures are shown to be true in the range of dimensions that have been computed (see Theorem \ref{thm:a26989b6}).

One highlight of this paper is that many of our computations of $E_2$ are based on the algorithms of \Groebner{} bases.
The (reduced) \Groebner{} basis gives a unique presentation of an algebra if we fix a choice of monomial ordering.
Hence we can easily check the conjectures against the computation in a range of dimensions if we present all in terms of \Groebner{} bases.

We make Conjecture \ref{conj:2a224128} claiming that the $E_2$ page of May is nilpotent free. 
This is in sharp contrast to the fact that all of the elements in positive dimensions
in the stable homotopy groups of spheres are nilpotent. 
It is also proved in the computed range by showing that there is an $\bF_2$-basis $B$ of a subalgebra such that $x\in B$ implies $x^2\in B$.

But there is weaker version of this conjecture which can be proved in all of the dimensions and the differentials in the May spectral sequence will gradually destroy this property:
\begin{theorem}[Theorem \ref{thm:512ce8e9}]
	For any $x$ in the $E_2$ page of the May spectral sequence and any $n\ge 0$, if $h_nx\neq 0$, then $h_n^ix\neq 0$ for all $i$.
\end{theorem}

The $E_2$ page of May plays a certain universal role in the study of Massey products. 
We know that
$$\langle h_0,h_1,\dots,h_n\rangle=0, ~n=1,2,\dots$$
is strictly defined and the $E_2$ page of May can be identified as the homology of the commutative DGA generated by the formal defining system (see Theorem \ref{thm:e0783a1e}).
A theorem by Gugenheim and May \cite{GM74} indicates that $E_2$ can be generated by the $h_i$ under matric Massey product.
In Theorem \ref{thm:d55b92ee}, we will show how this can be done by explicit formulas. The author learned the explicit formulas for $h_i(S)$ from May and wrote down the general proof.

For further computations of the May spectral sequence beyond $E_2$, we prove Theorem \ref{thm:9f9e1790} which will determine all $d_2$ differentials if our conjectures describing $E_2$ are true. 
In section \ref{sec:a82fc999} we construct some maps of spectral sequences which systematically extend the techniques used by May and Tangora computing differentials in the May spectral sequence based on the Adams vanishing theorem.

The author came across Conjecture \ref{conj:2cc43131} from May's thesis \cite{May64}, tried very hard to prove it, failed, and then wrote this paper.
May tried hard to prove it too while writing his thesis back in 1960s.
Steenrod told him to finish his thesis to complete the PhD and the conjecture remains unsolved ever since. May did not have any conjecture defining what the relations ought to be and our Conjecture \ref{conj:5f7d758} complements his conjecture.
The formulation of Conjecture \ref{conj:5f7d758} was given after the author's interaction with computers. The machines have given several counter examples of some earlier versions of the conjecture.
The author believes that one may need to assume both conjecture and prove them inductively at the same time.


\subsection{Acknowledgement}
The author would like to thank his advisor Peter May for many helpful discussions on this topic. 
The author was very glad to come across this problem and to get the support from his advisor to choose this problem as his thesis topic. 
May also read many drafts of this paper and offered tremendous help on writing.

\section{The \mybm{$E_2$}{E2} Page of the May Spectral Sequence}\label{sec:2f8ab9ed}
The main goal of this section is to state the conjectures which fully describes the $E_2$ page of the May spectral sequence in terms of generators and relations. 
We will show the universal role that $E_2$ plays in the study of Massey products.
\subsection{The associated graded algebra of the Steenrod algebra}
Recall May's results in his thesis \cite{May64} that we can filter the Steenrod algebra as follows.

Let $I(\sA)\subset \sA$ be the augmentation ideal. 
Let
$$\Phi_n: I(\sA)\otimes\cdots\otimes I(\sA)\longrightarrow I(\sA)$$
be the $n$-fold multiplication.

Define a decreasing filtration
$$F_0\sA=\sA;~ F_n\sA=\Image \Phi_n, n>0.$$
Then the associated graded Hopf algebra $E^0\sA$ of $\sA$ is defined by
$$E^0_{n,p}\sA=(F_p\sA/F_{p+1}\sA)_{n}.$$

A theorem due to Milnor and Moore \cite{MM} states that any primitively generated Hopf algebra over a field of characteristic $p$ is isomorphic to the universal enveloping algebra of its restricted Lie algebra of primitive elements. 
The associated graded algebra $E^0\sA$ satisfies the conclusion as follows.

\begin{theorem}[May]\label{thm:c65df7f}
	The associated graded algebra $E^0\sA$ can be represented by the associative algebra generated by $P^i_j$, $i\ge 0, j>0$ with relations
	$$(P^i_j)^2=0;~ [P^i_j, P^k_l]=\delta_{i,k+\ell}P^k_{j+\ell},~ i\ge k.$$
	Here $P^i_j\in E^0\sA$ corresponds to the projection of the dual of $\xi_j^{2^i}$ in the dual Steenrod algebra $$\sA_*=\bF_2[\xi_1,\xi_2,\dots]$$ with monomial basis. The grading of $E^0\sA$ is given by
	$$|P^i_j|_{n,p}=(2^i(2^j-1), j).$$
\end{theorem}

We can also filter the the cobar complex of $\sA$ based on this filtration. 
The resulting spectral sequence is the May spectral sequence.

\begin{theorem}[May]
	There exists a spectral sequence $(E_r, d_r)$ converging to the cohomology of the Steenrod algebra, and having its $E_2$ term $H^*(E^0\sA)$. 
	Each $E_r$ is a tri-graded algebra and each $d_r$ is a homomorphism
	$$d_r: E_r^{s,t,u}\to E_r^{s+1,t,u-r}$$
	which is a derivation with respect to the algebraic structure.
\end{theorem}

\subsection{The cohomology of \mybm{$E^0\sA$}{E0A} and its relationship to the Massey product}
For any Hopf algebra $A$, May \cite{May64} found a reasonably small complex with which to calculate $H^*(E^0A)$. 
As an application, for the Steenrod algebra $\sA$ we get the following.

\begin{theorem}[May]
	The cohomology of the associated graded algebra $E^0\sA$ is isomorphic to the homology  of the differential graded algebra
	$$X=\bF_2[R^i_j: i\ge 0, j>0]$$
	$$|R^i_j|_{s,t,u}=(1, 2^i(2^j-1), j-1)$$
	with differentials $d: X_{s,t,u}\to X_{s+1,t,u-1}$ given by
	\begin{equation}\label{eq:3eab9b7f}
	d R^i_j=\sum_{k=1}^{j-1}R^{i+k}_{j-k}R^i_k.
	\end{equation}
\end{theorem}

\begin{remark}
	May proved this theorem by showing that $E^0\sA\otimes X^*$ is an $E^0\sA$-free resolution of $\bF_2$ which is much smaller than the bar construction. 
	In 1970 after May's thesis, Priddy \cite{Priddy70}\cite{Priddy70-2} conceptualized this method into Koszul resolutions which apply to a more general kind of algebras called Koszul algebras. 
	The complex $X$ can be interpreted as the co-Koszul complex of $E^0\sA$ in terms of Priddy's setting.
\end{remark}

\begin{definition}\label{def:8ac2f389}
	We reindex the generators of $X$ by
	$$R_{ij}=\begin{cases}
	R^i_{j-i}, & \text{ if } 0\le i<j,\\
	0,  & \text{ if } 0\le j\le i.
	\end{cases}$$
\end{definition}

With a little rewriting, (\ref{eq:3eab9b7f}) now becomes
\begin{equation}\label{eq:99758f58}
	d R_{ij}=\sum_{k=i+1}^{j-1} R_{ik}R_{kj}.
\end{equation}
If we regard $R$ as the strictly upper triangular matrix $(R_{ij})$, then $d R=R^2$.

\begin{remark}
	The symbol $R^i_j$ is written as $R_{ij}$ by Tangora \cite{Tangora} but as $R_{i, i+j}$ in this paper.
\end{remark}

The homology $HX$ has some universal properties in the study of Massey products. 
Recall from \cite{May69} that if $A$ is a commutative differential graded algebra and $a_1,\dots,a_{n}\in HA$, then the $n$-fold Massey product $\Massey{a_1,\dots,a_n}$ is defined if and only if we can find $a_{ij}\in A$ for $0\le i<j\le n$ and $(i,j)\neq (0,n)$ such that $a_i\in HA$ is represented by $a_{i-1,i}\in A$ and
$$da_{ij}=\sum_{k=i+1}^{j-1}a_{ik}a_{kj}, \hspace{0.5cm} 0\le i<j\le n \text{ and } (i,j)\neq (0,n).$$
In this case $\Massey{a_1,\dots,a_n}$ is defined by the homology class of
$$a_{0,1}a_{1,n}+\cdots+a_{0,n-1}a_{n-1,n}.$$
Note that the formulas resemble (\ref{eq:99758f58}), which directly implies the following two theorems.

\begin{definition}
We define $X_n=\bF_2[R_{ij}: 0\le i<j\le n]$. It is a sub-differential graded algebra of $X$.
\end{definition}

\begin{theorem}\label{thm:e0783a1e}
    If $A$ is a commutative differential algebra, then the decompositions of zero in $HA$ as an $n$-ary Massey product (together with a defining system)
    $$0\in \langle a_1,\dots,a_n\rangle,\hspace{1cm} a_i\in HA$$
    are in one-to-one correspondence with maps of differential algebras:
    $$f: X_n\to A$$
    where $f$ induces the algebraic map
    $$f_*: HX_n\to HA$$
    with $f_*(h_{i-1})=a_i$, $1\le i\le n$, where $h_{i-1}$ is the homology class of $R_{i-1,i}$.
\end{theorem}

\begin{theorem}
    A nontrivial element $a\in HA$ and a defining system for the Massey product
    $$a\in\langle a_1,\dots,a_n\rangle$$
    corresponds to the obstruction to obtaining the dashed map
    $$\xymatrix{
        X_n\ar@{-->}[rd] & X_{n-1}\ar[d]^{f_0}\ar@{_(->}[l]_{i_0}\\
        X_{n-1}\ar[r]_{f_1}\ar@{^(->}[u]^{i_1} & A
    }$$
    where $f_0$ corresponds to the sub-defining system for
    $0\in \langle a_1,\dots,a_{n-1}\rangle$ and $f_1$ for $0\in \langle a_2,\dots,a_n\rangle$. 
	The embeddings $i_0$ and $i_1$ are given by $i_0(R_{ij})=R_{ij}$ and $i_1(R_{ij})=R_{i+1,j+1}$.
\end{theorem}

\subsection{The indecomposables of \mybm{$H^*(E^0\sA)$}{H*(E0A)}} 

\begin{definition}
	For two strictly increasing sequences of distinct numbers $S=\{s_1,\dots,s_n\}$, $T=\{t_1,\dots,t_n\}$, we define
	$$R_{S,T}=\det(R_{s_it_j})=\sum_{\sigma\in \Sigma_n} R_{s_1t_{\sigma(1)}}\cdots R_{s_nt_{\sigma(n)}}.$$
	Note that the value of $R_{S, T}$ does not depend on the ordering of numbers in $S$ or $T$. 
	However we prefer to put them in order, and in the rest of the paper, we assume all sequences $S$ and $T$ are ordered.
\end{definition}

\begin{definition}
	For two sequences $S$ and $T$, we write $S<T$ if $\max(S)<\min(T)$ and $S\le T$ if $\max(S)\le\min(T)$.
\end{definition}

\begin{proposition}\label{prop:af5d3cd9}
	The determinants $R_{S, T}$ have the following properties
	\begin{enumerate}
		\item $R_{S, T}$ is nonzero if and only if $s_i<t_i$ for $1\le i\le n$.\label{item:83431c8f}
		\item If $T_1\le S_2$ or $T_2\le S_1$, then
		$$R_{S_1\cup S_2,T_1\cup T_2}=R_{S_1,T_1}R_{S_2,T_2}$$\label{item:19a3a4b}
		\item $d R_{S, T}=\sum\limits_{k\in \bZ_{\ge 0}\backslash (S\cup T)}R_{S\cup\{k\},T\cup\{k\}}$.\vspace{2pt} Note that the summand of the summation is zero when $k<\min(S\cup T)$ or $k>\max(S\cup T)$ because of (1).\label{item:374301ae}\vspace{4pt}
		\item For any fixed subset $I$ of $S$, 
		$$R_{S,T}=\sum\limits_{|J|=|I|}R_{I,J}R_{S-I,T-J},$$
		where $|J|$ means the cardinality of the set $J$.
		Similarly, for any fixed subset $J$ of $T$, 
		$$R_{S,T}=\sum\limits_{|I|=|J|}R_{I,J}R_{S-I,T-J}.$$ \label{item:cea34d63}
	\end{enumerate}
\end{proposition}

\begin{proof}
	We keep using the fact that $R_{S, T}$ is the determinant of $(R_{s_it_j})$.
	\begin{enumerate}
		\item If $s_i\ge t_i$ for some $i$, then $R_{s_jt_k}=0$ if $j\ge i\ge k$ which yields zero determinant. 
		Thus $s_i<t_i$ for all $i$.
		\item If $T_1\le S_2$ or $T_2\le S_1$ we have either an upper or lower triangular block matrix associated to $R_{S_1\cup S_2,T_1\cup T_2}$ with determinants of the diagonal blocks being $R_{S_1,T_1}$ and $R_{S_2,T_2}$.
		\item By the definition of $R_{S, T}$ and property (1), we have
		\begin{align*}
		d R_{S, T} &=\sum_{\sigma\in \Sigma_n} d(R_{s_1t_{\sigma(1)}}\cdots R_{s_nt_{\sigma(n)}})\\
		&= \sum_{\sigma\in \Sigma_n}\sum_i\sum_k R_{s_1t_{\sigma(1)}}\cdots \hat R_{s_it_{\sigma(i)}}\cdots R_{s_nt_{\sigma(n)}}\cdot R_{s_ik}R_{k\sigma(i)}\\
		&= \sum\limits_{k\notin S\cup T} R_{S\cup\{k\},T\cup\{k\}}.
		\end{align*}
		Here $\hat R_{s_it_{\sigma(i)}}$ means that we skip the factor in the product.
		\item This is the expansion of the determinant of $(R_{s_it_j})$ by the rows corresponding to $I$.
	\end{enumerate}
\end{proof}

\begin{definition}
	Assume we have two sequences $S=\{s_1,\dots,s_n\}$ and $T=\{t_1,\dots,t_n\}$ such that $s_k<t_k$ for $1\le k\le n$ and 
	$$S\cup T=\{i,i+1,\dots,i+2n-1\}$$
	for some integer $i$. 
	Then $d R_{S,T}=0$ by (\ref{item:374301ae}) of the above proposition. 
	Let $\sH^\prime$ be the set of homology classes of all such $R_{S, T}$. 
	Let $\sH$ be the set of homology classes of all such $R_{S, T}$ with one extra condition that $s_k<t_{k-1}$ for $2\le k\le n$. 
	For convenience we use $h_{S, T}$ or $h_i(S^\prime)$ to denote the homology class of $R_{S, T}$, where $i=s_1$ and $S^\prime=\{s_2-s_1,\dots,s_n-s_1\}$. 
	The simplest examples are $h_{i,i+1}=h_i=[R_{i,i+1}]$.
\end{definition}

\begin{remark}
	By Proposition \ref{prop:af5d3cd9}.(\ref{item:19a3a4b}) we can see that every element in $\sH^\prime$ can be decomposed as a product of elements in $\sH$.
\end{remark}

\begin{theorem}[May]
	All elements in $\sH$ are indecomposable in $HX$.
\end{theorem}

\begin{proof}
	Here we provide the proof from \cite[II.5]{May64} adapted to our notation.

	First we consider $X$ as the dual of the divided power algebra
	$$\Gamma=\Gamma[\gamma_n(P_{ij}):0\le i<j, ~n\ge 0]$$
	whose additive basis is given by
	$$\gamma_{r_1}(P_{i_1j_1})\cdots \gamma_{r_k}(P_{i_kj_k}), \hspace{0.5cm}k\ge 0, ~r_\ell\ge 0, ~P_{i_1j_1}<\cdots<P_{i_kj_k}$$
    where $P_{i_aj_a}<P_{i_bj_b}$ if and only if $i_a<i_b$ or $i_a=i_b$ and $j_a<j_b$. 
	The dual basis for $X$ is given by
    $$R_{i_1j_1}^{r_1}\cdots R_{i_kj_k}^{r_k}=\left(\gamma_{r_1}(P_{i_1j_1})\cdots \gamma_{r_k}(P_{i_kj_k})\right)^*.$$

    The multiplication in $X$ corresponds to the comultiplication
    $$\psi:\Gamma\to \Gamma\otimes \Gamma$$
    given by
    \begin{align*}
        & \psi(\gamma_{r_1}(P_{i_1j_1})\cdots \gamma_{r_k}(P_{i_kj_k}))\\
        =& \sum_{\substack{r_\ell^\prime+r_\ell^\pprime=r_\ell\\1\le\ell\le k}} \left(\gamma_{r_1^\prime}(P_{i_1j_1})\cdots \gamma_{r_k^\prime}(P_{i_kj_k})\right)\otimes \left(\gamma_{r_1^\pprime}(P_{i_1j_1})\cdots \gamma_{r_k^\pprime}(P_{i_kj_k})\right).
    \end{align*}

    The differential $d:X\to X$ corresponds to
    $$\delta:\Gamma\to \Gamma$$
    given by
    \begin{align*}
        & \delta(\gamma_{r_1}(P_{i_1j_1})\cdots \gamma_{r_k}(P_{i_kj_k}))\\
        =& \sum_{\substack{s<t\\r_s,r_t\text{ odd}}} \gamma_1([P_{i_sj_s},P_{i_tj_t}])\gamma_{r_1}(P_{i_1j_1})\cdots \gamma_{r_s-1}(P_{i_sj_s})\cdots\gamma_{r_t-1}(P_{i_tj_t})\cdots \gamma_{r_k}(P_{i_kj_k}).
    \end{align*}
    where
    $$[P_{i_sj_s},P_{i_tj_t}]=\begin{cases}
        P_{i_sj_t} & \text{ if }j_s=i_t,\\
        0 & \text{ otherwise}.
    \end{cases}$$

    Next we consider $h_{S, T}\in \sH$ where
    $$S=\{s_1,\dots,s_n\},~T=\{t_1,\dots,t_n\},$$
    and
    $$S\cup T=\{i,i+1,\dots,i+2n-1\}.$$
    The definition of $\sH$ requires that $s_k<t_{k-1}$ for $2\le k\le n$.

    For convenience we write $P_{ij}=\gamma_1(P_{ij})$. 
	Consider $\alpha\in \Gamma$ given by
    $$\alpha=P_{s_1t_n}P_{s_2t_1}\cdots P_{s_nt_{n-1}}.$$
    The element $\alpha$ is a cycle in $\Gamma$ by the definition of $\delta$ since $\{s_k,t_k,1\le k\le n\}$ are all distinct integers. 
	Note that $h_{S,T}$ is represented by $R_{S,T}$ and the pairing $\langle R_{S,T},\alpha\rangle=1$ is nontrivial. 
	Hence it suffices to prove that the homology class of $\alpha$ is primitive in $H\Gamma$.

    In fact, we have
    $$\psi(\alpha)=\sum \alpha^\prime\otimes\alpha^\pprime$$
    where the summation is indexed on all $\alpha^\prime$ that are products of a subset of 
    $$\{P_{s_1t_n},P_{s_2t_1},\cdots,P_{s_nt_{n-1}}\}$$
    and $\alpha^\pprime$ are the products of the complementary set. 
	Consider a summand $\alpha^\prime\otimes\alpha^\pprime$ such that $\alpha^\prime\neq 1$ and $\alpha^\pprime\neq 1$.
    Note that both $\alpha^\prime\neq 1$ and $\alpha^\pprime\neq 1$ are cycles in $\Gamma$ and either $\alpha^\prime$ or $\alpha^\pprime$ contains the factor $P_{s_1t_n}$. 
	Without loss of generality, we assume that $\alpha^\prime$ contains $P_{s_1t_n}$ and it suffices to show that $\alpha^\prime$ is a boundary.

    Write
    $$\alpha^\prime=P_{s_1t_1}P_{i_1j_1}\cdots P_{i_\ell,j_\ell}$$
    where $\{P_{i_1j_1}\cdots P_{i_\ell,j_\ell}\}$ is a proper set of
    $$\{P_{s_2t_1},\cdots,P_{s_nt_{n-1}}\}.$$
    Then there exists $u$ such that $s_1=i<u<i+2n-1=t_n$ and
    $$u\notin\{i_1,j_1,\dots,i_\ell,j_\ell\}.$$
    By the definition of $\delta$ we have
    $$\delta(P_{s_1u}P_{ut_1}P_{i_1j_1}\cdots P_{i_\ell,j_\ell})=P_{s_1t_1}P_{i_1j_1}\cdots P_{i_\ell,j_\ell}=\alpha^\prime.$$
    Therefore $\alpha^\prime$ is a boundary in $\Gamma$.
\end{proof}

Beside elements of $\sH$, we can also see that the homology classes of $R_{ij}^2$ for $j-i\ge 2$ are also indecomposables of $HX$. 
Let $b_{S,T}$ denote the homology class of $R_{S, T}^2$. 
Especially, $b_{ij}=[R_{ij}^2]$ and $b_{i,i+1}=h_i^2$.

The following conjecture suggests that it is possible that these are all the indecomposables we need in $HX$.

\begin{conjecture}[May, {\cite[Conjecture II.5.7]{May64}}]\label{conj:2cc43131}
	The elements of $\sH$ and $b_{ij}$ ($j-i\ge 2$) form a basis of indecomposables of $HX$.
\end{conjecture}

\subsection{The relations in \mybm{$H^*(E^0\sA)$}{H*(E0A)}}
In addition to Conjecture \ref{conj:2cc43131}, we will state a conjecture to describe all the relations in $H^*(E^0\sA)\iso HX$.

\begin{definition}
	For $0\le m<n$, we define
	$$\sH_{mn}=\{h_{S,T}\in \sH: \min(S)=m,~\max(T)=n\}$$
	and
	$$\sH^\prime_{mn}=\{h_{S,T}\in \sH^\prime: \min(S)=m,~\max(T)=n\}.$$
	Note that $\sH_{mn}\subset \sH^\prime_{mn}$ and $\sH_{mn}$, $\sH^\prime_{mn}$ are empty if $n-m$ is even.
\end{definition}

\begin{definition}
	For a sequence $S=\{s_1,\dots,s_n\}$, we define $|S|$ to be the length $n$ of $S$. 
	For $a<b$, let $N_{a, b}$ be the sequence $\{a, a+1,\dots,b\}$.
\end{definition}

\begin{conjecture}\label{conj:5f7d758}
	The algebra $HX$ is generated by $h_{S, T}\in \sH$ and $b_{ij}$ ($j-i\ge 2$) with the following relations.

	\begin{enumerate}
		\myitem{(1)} For all $0\le i<j$, \label{rel:8d29c79f}
		$$\sum_k b_{ik}b_{kj}=0.$$
		\myitem{(2)} Assume $h_{S_1,T_1}\in \sH^\prime_{a_1,b_1}$, $h_{S_2,T_2}\in \sH^\prime_{a_2,b_2}$, $a_1<a_2<b_1<b_2$ and $b_1-a_2$ is even. 
		Then \label{rel:2762c538}
		$$h_{S_1, T_1}h_{S_2, T_2}=0.$$
		\myitem{(3A)} Assume that $S\subset N=\{a,a+1,\dots,a+2k-1\}$ and $|S|=k+1$. 
		Let $T$ be the complement of $S$ in $N$. 
		Then\label{rel:b0a9481d}
		$$\sum_{s\in S} b_{sj}h_{S-\{s\}, T+\{s\}}=0$$
		for any $j\le a+2k$.
		\myitem{(3B)}\label{rel:c4698d4e} Assume that $T\subset N=\{a,a+1,\dots,a+2k-1\}$ and $|T|=k+1$. 
		Let $S$ be the complement of $T$ in $N$.
		Then 
		$$\sum_{t\in T} b_{it}h_{S+\{t\}, T-\{t\}}=0$$ for any $i\ge a-1$.
		\myitem{(3C)}\label{rel:1c2da7ed} Assume that $$S_1\subset N_1=\{a_1,a_1+1,\dots,a_1+2k_1-1\}, ~|S_1|=k+1,$$ $$T_2\subset N_2=\{a_2,a_2+1,\dots,a_2+2k_2-1\}, ~|T_2|=k+1$$ and $\max(N_1)<\min(N_2)$. 
		Let $T_1=N_1\backslash S_1$ and $S_2=N_2\backslash T_2$.
		Then 
		$$\sum_{s\in S_1, t\in T_2} b_{st}h_{S_1-\{s\}, T_1+\{s\}}h_{S_2+\{t\}, T_2-\{t\}}=0.$$
		\myitem{(4A)}\label{rel:a078c1d9} Assume $h_{S_1,T_1}\in \sH^\prime_{a_1,b_1}$, $h_{S_2,T_2}\in \sH^\prime_{a_2,b_2}$, $a_1\le a_2<b_1\le b_2$ and $b_1-a_2$ is odd. 
		Then
		\begin{equation*}
		h_{S_1,T_1}h_{S_2,T_2}=\sum_{\substack{I\subset T_1^\pprime\cap S_2\\2|I|=|T_1^\pprime|-|S_1^\pprime|}} h_{S_1^\pprime+I,T_1^\pprime-I}h_{S_1^\prime+S_2-I, T_1^\prime+T_2+I}
		\end{equation*}
		Where $S_1^\prime=S_1\backslash N_{a_2,b_1}$, $S_1^\pprime=S_1\cap N_{a_2,b_1}$, $T_1^\prime=T_1\backslash N_{a_2,b_1}$, $T_1^\pprime=T_1\cap N_{a_2,b_1}$.
		\myitem{(4B)}\label{rel:51b194a6} Assume $h_{S_1,T_1}\in \sH^\prime_{a_1,b_1}$, $h_{S_2,T_2}\in \sH^\prime_{a_2,b_2}$, $a_1\le a_2<b_1\le b_2$ and $b_1-a_2$ is odd. 
		Then
		\begin{equation*}
		h_{S_1,T_1}h_{S_2,T_2} = \sum_{\substack{I\subset T_1\cap S_2^\pprime\\2|I|=|S_2^\pprime|-|T_2^\pprime|}} h_{S_2^\pprime-I,T_2^\pprime+I}h_{S_1+S_2^\prime+I, T_1+T_2^\prime-I}
		\end{equation*}
		Where $S_2^\prime=S_2\backslash N_{a_2,b_1}$, $S_2^\pprime=S_2\cap N_{a_2,b_1}$, $T_2^\prime=T_2\backslash N_{a_2,b_1}$, $T_2^\pprime=T_2\cap N_{a_2,b_1}$.
		\myitem{(5)}\label{rel:e7460c84} Assume $h_{S_1,T_1}\in \sH^\prime_{a_1,b_1}$, $h_{S_2,T_2}\in \sH^\prime_{a_2,b_2}$, $a_1\le a_2<b_1\le b_2$ and $b_1-a_2$ is odd. 
		Then
		\begin{equation*}
		h_{S_1,T_1}h_{S_2,T_2}=\sum_{\substack{I\subset S_1^\prime\\J\subset T_2^\prime}} h_{S_1^\prime-I,T_1^\prime+I}h_{S_2^\prime+J,T_1^\prime-J}b_{S_1^\pprime+I, T_2^\pprime+J}
		\end{equation*}
		Where $S_i^\prime=S_i\backslash N_{a_2,b_1}$, $S_i^\pprime=S_i\cap N_{a_2,b_1}$, $T_i^\prime=T_i\backslash N_{a_2,b_1}$, $T_i^\pprime=T_i\cap N_{a_2,b_1}$, $i=1, 2$.
		\myitem{(6)}\label{rel:80bb77fd}
		Assume $h_{S_i, T_i}\in \sH_{a, b}$, $i=1,\dots,n$, and 
		$$\sum_i x_ih_{S_i-\{a\}, T_i-\{b\}}=0$$
		where $x_i$ is a product of elements in
		$$\bigcup_{\substack{a<a^\prime<b^\prime<b\\a^\prime-a\text{ is odd}}} \sH_{a^\prime, b^\prime}$$
		Then 
		$$\sum_i x_ih_{S_i, T_i}=0$$
	\end{enumerate}
\end{conjecture}

In order to prove Conjecture \ref{conj:5f7d758}, we have to prove that all the relations in the conjecture hold and they imply all the other relations.
We are not there yet although we have a great deal of evidence for the conjecture.
In the rest of the section we will describe the results we already have, including evidence for Conjecture \ref{conj:2a224128}.

\begin{theorem}\label{thm:bc7c165}
	The relations (1), (2), (3A), (3B), (4A) and (4B) in Conjecture \ref{conj:5f7d758} hold in $HX$. 
	The relations (3C), (5) and (6) hold in a large range of dimensions.
\end{theorem}

The proof of this theorem contains a lot of technical calculations. 
Readers who are not interested in the details of this proof may skip to Theorem \ref{thm:a26989b6}.

The following proposition for all $n$ shows that the statement (3A) is symmetric to (3B) and (4A) is symmetric to (4B). 
Hence we only have to prove one for each pair.

\begin{proposition}\label{prop:8ac1afe}
	The reflection map
	$$X_n\to X_n$$
	$$R_{ij}\mapsto R_{n-j, n-i}$$
	is an isomorphism between differential algebras. 
	Therefore $HX_n$ is isomorphic to itself via this reflection map.
\end{proposition}

The proof is straightforward. 
Before we prove Theorem \ref{thm:bc7c165} we need the following lemma.
\begin{lemma}\label{lem:13393410}
	Assume that $S_1,T_1,S_2,T_2$ are four sequences such that $|S_1|=|T_1|-1$, $|S_2|=|T_2|+1$, $S_1\cap T_1=\emptyset=S_2\cap T_2$ and 
	$$(S_1\cup T_1)\backslash (S_2\cup T_2)<(S_1\cup T_1)\cap(S_2\cup T_2)<(S_2\cup T_2)\backslash (S_1\cup T_1).$$ 
	Then
	\begin{equation*}
	\sum_{\substack{s\in S_1\cap S_2\\i\in T_1\cap S_2}} R_{S_1,T_1-\{i\}}R_{S_2-\{s\},T_2}R_{s,i} =
	\sum_{\substack{i\in T_1\cap S_2\\t\in T_1\cap T_2}}R_{S_1,T_1-\{t\}}R_{S_2-\{i\},T_2}R_{i,t}.
	\end{equation*}
\end{lemma}
\begin{proof}
	By Proposition \ref{prop:af5d3cd9}.(\ref{item:cea34d63}), these are both equal to $$\sum_{\substack{s\in S_1\cap S_2\\t\in T_1\cap T_2}}R_{S_1,T_1-\{t\}}R_{S_2-\{s\},T_2}R_{s,t}.$$
\end{proof}

\newcommand{\mysum}{\sum_{\substack{S_1\ni s<a_2\\i\in T_1\cap S_2}}}
We now prove Theorem \ref{thm:bc7c165} by realizing the relations as boundaries via explicit constructions.
\begin{proof}[Proof of Theorem \ref{thm:bc7c165}]
	\ref{rel:8d29c79f}. 
	The relation follows from $$d(R_{ij}d R_{ij})=(d R_{ij})^2=\sum_k R_{ik}^2R_{kj}^2.$$

	\ref{rel:2762c538}. 
	Let $$y=\mysum R_{S_1-\{s\},T_1-\{i\}}R_{S_2-\{i\}+\{s\},T_2}.$$
	It suffices to show that $d y=R_{S_1,T_1}R_{S_2,T_2}$. 
	In fact,
	\begin{align*}
	d y =& \mysum \Big( R_{S_1,T_1-\{i\}+\{s\}}R_{S_2-\{i\}+\{s\},T_2} + R_{S_1-\{s\}+\{i\},T_1}R_{S_2-\{i\}+\{s\},T_2} + \\
	&R_{S_1-\{s\},T_1-\{i\}}R_{S_2+\{s\},T_2+\{i\}} + \sum_{j<a_2}R_{S_1-\{s\},T_1-\{i\}}R_{S_2-\{i\}+\{j\},T_2}R_{s,j}\Big).
	\end{align*}
	
	We apply \ref{prop:af5d3cd9}.(\ref{item:cea34d63}) and get
	\begin{equation*}
	\mysum\sum_{j<a_2}R_{S_1-\{s\},T_1-\{i\}}R_{S_2-\{i\}+\{j\},T_2}R_{s,j} = 
	\mysum R_{S_1,T_1-\{i\}+\{s\}}R_{S_2-\{i\}+\{s\},T_2}.
	\end{equation*}
	
	Therefore
	\begin{align*}
	d y =& \mysum\left(R_{S_1-\{s\},T_1-\{i\}}R_{S_2+\{s\},T_2+\{i\}} + R_{S_1-\{s\}+\{i\},T_1}R_{S_2-\{i\}+\{s\},T_2}\right)\\
	=& \mysum R_{S_1-\{s\},T_1-\{i\}}R_{S_2,T_2}R_{s,i} + 
	\sum_{\substack{S_1\ni s<a_2\\i\in T_1\cap S_2\\s_2\in S_2\cap S_1}} R_{S_1-\{s\},T_1-\{i\}}R_{S_2-\{s_2\}+\{s\},T_2}R_{s_2,i} \\
	+& \sum_{\substack{S_1\ni s<a_2\\i\in T_1\cap S_2\\t_1\in T_1\cap T_2}}R_{S_1-\{s\},T_1-\{t_1\}}R_{S_2-\{i\}+\{s\},T_2}R_{i,t_1}.
	\end{align*}
	
	By Lemma \ref{lem:13393410} we have
	\begin{align*}
	& \sum_{\substack{S_1\ni s<a_2\\i\in T_1\cap S_2\\s_2\in S_2\cap S_1}} R_{S_1-\{s\},T_1-\{i\}}R_{S_2-\{s_2\}+\{s\},T_2}R_{s_2,i} +
	\sum_{\substack{S_1\ni s<a_2\\i\in T_1\cap S_2\\t_1\in T_1\cap T_2}}R_{S_1-\{s\},T_1-\{t_1\}}R_{S_2-\{i\}+\{s\},T_2}R_{i,t_1}\\
	=& \sum_{\substack{S_1\ni s<a_2\\t_1\in T_1\cap T_2}}R_{S_1-\{s\},T_1-\{t_1\}}R_{S_2,T_2}R_{s,t_1}
	\end{align*}
	
	Therefore
	\begin{align*}
	d y =& \mysum R_{S_1-\{s\},T_1-\{i\}}R_{S_2,T_2}R_{s,i} + \sum_{\substack{S_1\ni s<a_2\\t_1\in T_1\cap T_2}}R_{S_1-\{s\},T_1-\{t_1\}}R_{S_2,T_2}R_{s,t_1}\\
	=& \sum_{\substack{S_1\ni s<a_2\\T_1\ni t\ge a_2}} R_{S_1-\{s\},T_1-\{t\}}R_{S_2,T_2}R_{s,t}\\
	=& ~R_{S_1,T_1}R_{S_2,T_2}.
	\end{align*}
	The last equality holds because for every monomial $\alpha$ in $R_{S_1,T_1}$ there is an odd number of factors $R_{s,t}$ in $\alpha$ such that $S_1\ni s<a_2, T_1\ni t\ge a_2$.\vspace{5pt}
	
	\ref{rel:b0a9481d}. 
	Let $$y=\sum_{\{s_1<s_2\}\subset S} R_{s_1j}R_{s_2j}R_{S-\{s_1,s_2\},T}$$
	\begin{align*}
	dy =& \sum_{\{s_1<s_2\}\subset S}\sum_i \left(R_{s_1i}R_{ij}R_{s_2j} + R_{s_2i}  R_{ij}R_{s_1j}\right)R_{S-\{s_1,s_2\},T}\\
	+& \sum_{\{s_1<s_2\}\subset S} R_{s_1j}R_{s_2j}\left(R_{S-\{s_1\}, T+\{s_2\}} + R_{S-\{s_2\}, T+\{s_1\}}\right)\\
	=& \sum_{\substack{s_1,s_2\in S\\s_1\neq s_2}}\sum_i R_{s_2i}R_{ij}R_{s_1j}R_{S-\{s_1,s_2\},T} + \sum_{\substack{s_1,s_2\in S\\s_1\neq s_2}} R_{s_1j}R_{s_2j}R_{S-\{s_1\}, T+\{s_2\}}\\
	=& \RN{1}+\RN{2}
	\end{align*}
	where
	$$\RN{1}=\sum_{\substack{s_1,s_2\in S\\s_1\neq s_2}}\sum_i R_{s_2i}R_{ij}R_{s_1j}R_{S-\{s_1,s_2\},T}$$
	and
	\begin{align*}
	\RN{2}=& \sum_{\substack{s_1,s_2\in S\\s_1\neq s_2}} R_{s_1j}R_{s_2j}R_{S-\{s_1\}, T+\{s_2\}}\\
	=& \sum_{\substack{s_1,s_2,i\in S\\s_1\neq s_2\\s_1\neq i}} R_{s_1j}R_{s_2j}R_{is_2}R_{S-\{s_1,i\}, T}\\
	=& \sum_{\substack{s_1,i,s_2\in S\\s_1\neq i\\s_1\neq s_2}} R_{s_1j}R_{ij}R_{s_2i}R_{S-\{s_1,s_2\}, T}\\
	\end{align*}
	
	The only difference between summations $\RN{1}$ and $\RN{2}$ is that $i$ can be equal to $s_1$ or $i\in T$ in summation $\RN{1}$. 
	Therefore
	\begin{align*}
	dy =& \sum_{\substack{s_1,s_2\in S\\s_1\neq s_2\\i\in T\cup\{s_1\}}} R_{s_2i}R_{ij}R_{s_1j}R_{S-\{s_1,s_2\},T}\\
	=& \sum_{\substack{s_1\in S\\i\in T\cup\{s_1\}}} R_{ij}R_{s_1j}R_{S-\{s_1\},T+\{i\}}\\
	=& \sum_{\substack{s_1\in S}} R_{s_1j}^2R_{S-\{s_1\},T+\{s_1\}}\\
	\end{align*}
	where the right-hand side represents our relation. 
	Hence our relation holds.\vspace{6pt}

	\ref{rel:c4698d4e}. 
	This follows from \ref{rel:b0a9481d} because of the symmetry given by Proposition \ref{prop:8ac1afe}.\vspace{5pt}

	\ref{rel:a078c1d9}.	Let
	$$y = \sum_{\substack{I,J\subset T_1^\pprime\cap S_2\\j_0=\max (J\backslash I)>I\backslash J}} R_{S_1^\pprime+I,T_1^\pprime-J}R_{S_1^\prime+S_2-I-\{j_0\}, T_1^\prime+T_2+J^\prime}$$
	Then
	$$dy=\RN{1}+\RN{2}+\RN{3}+\RN{4}$$
	where
	\begin{align*}
		\RN{1} =& \sum_{\substack{I,J\subset T_1^\pprime\cap S_2\\j_0=\max (J\backslash I)>I\backslash J}} R_{S_1^\pprime+I,T_1^\pprime-J}R_{S_1^\prime+S_2-I, T_1^\prime+T_2+J}\\
		\RN{2} =& \sum_{\substack{I,J\subset T_1^\pprime\cap S_2\\j_0=\max (J\backslash I)>I\backslash J}} R_{S_1^\pprime+I+\{j_0\},T_1^\pprime-J+\{j_0\}}R_{S_1^\prime+S_2-I-\{j_0\}, T_1^\prime+T_2+J^\prime}\\
		\RN{3} =& \sum_{\substack{I,J\subset T_1^\pprime\cap S_2\\j_0=\max (J\backslash I)>I\backslash J\\j\in J^\prime\backslash I}} R_{S_1^\pprime+I+\{j\},T_1^\pprime-J+\{j\}}R_{S_1^\prime+S_2-I-\{j_0\}, T_1^\prime+T_2+J^\prime}\\
	\end{align*}
	\begin{align*}
		\RN{4} =& \sum_{\substack{I,J\subset T_1^\pprime\cap S_2\\j_0=\max (J\backslash I)>I\backslash J\\i\in I\backslash J}} R_{S_1^\pprime+I,T_1^\pprime-J}R_{S_1^\prime+S_2-I-\{j_0\} + \{i\}, T_1^\prime+T_2+J^\prime+\{i\}}.\\
	\end{align*}
	In summation $\RN{3}$, change index by $I_1=I+\{j\}$, $J_1=J-\{j\}$. 
	We have
	\begin{align*}
	\RN{3} =& \sum_{\substack{I_1,J_1\subset T_1^\pprime\cap S_2\\j_0=\max (J_1\backslash I_1)>I_1\backslash J_1\\j\in I_1\backslash J_1}} R_{S_1^\pprime+I_1,T_1^\pprime-J_1}R_{S_1^\prime+S_2-I_1+\{j\}-\{j_0\}, T_1^\prime+T_2+J_1^\prime+\{j\}}\\
	=& \sum_{\substack{I,J\subset T_1^\pprime\cap S_2\\j_0=\max (J\backslash I)>I\backslash J\\i\in I\backslash J}} R_{S_1^\pprime+I,T_1^\pprime-J}R_{S_1^\prime+S_2-I-\{j_0\} + \{i\}, T_1^\prime+T_2+J^\prime+\{i\}}\\
	=& \RN{4}
	\end{align*}
	In summation $\RN{2}$, change index by $I_1=I+\{j_0\}$, $J_1=J-\{j_0\}$. 
	We have
	\begin{align*}
	\RN{2} =& \sum_{\substack{I_1,J_1\subset T_1^\pprime\cap S_2\\j_0=\max(I_1\backslash J_1)>J_1\backslash I_1}} R_{S_1^\pprime+I_1,T_1^\pprime-J_1}R_{S_1^\prime+S_2-I_1, T_1^\prime+T_2+J_1}
	\end{align*}
	Therefore
	\begin{align*}
	dy = \RN{1} + \RN{2} =& \sum_{\substack{I,J\subset T_1^\pprime\cap S_2\\I\neq J}} R_{S_1^\pprime+I,T_1^\pprime-J}R_{S_1^\prime+S_2-I, T_1^\prime+T_2+J}
	\end{align*}
	Note that if we instead require $I=J$ in the above summation, we get the right hand side of Relation \ref{rel:a078c1d9}. 
	Hence in order to prove Relation \ref{rel:a078c1d9} it suffices to show that
	$$\sum_{\substack{I,J\subset T_1^\pprime\cap S_2}} R_{S_1^\pprime+I,T_1^\pprime-J}R_{S_1^\prime+S_2-I, T_1^\prime+T_2+J}=R_{S_1,T_1}R_{S_2,T_2}.$$
	In fact, if we denote the summation on the left hand side above by $\RN{5}$, then
	$$\RN{5}=\sum_{\substack{I,J\subset T_1^\pprime\cap S_2\\M\subset S_1^\prime\\L\subset S_1^\prime-M+S_2-I}} R_{S_1^\pprime+I,T_1^\pprime-J}R_{S_1^\prime-M+S_2-I-L, J}R_{M, T_1^\prime}R_{L, T_2}$$
	Fix $I$, $M$ and $L$. 
	If
	\begin{equation}\label{eq:e4826ad1}
		(S_1^\pprime+I)\cap(S_1^\prime-M+S_2-I-L)=\emptyset
	\end{equation}
	which is equivalent to
	$$(S_1^\pprime+I)\cap((S_2\backslash L)-I)=\emptyset
	\text{ \hspace{1pt}and to\hspace{1pt} } S_1^\pprime\cap S_2=S_1\cap S_2\subset L,$$
	then by \ref{prop:af5d3cd9}.\ref{item:cea34d63} we have
	\begin{align*}
	& \sum_{J\subset T_1^\pprime\cap S_2} R_{S_1^\pprime+I,T_1^\pprime-J}R_{S_1^\prime-M+S_2-L-I, J}\\
	=& R_{(S_1^\pprime+I)+(S_1^\prime-M+S_2-L-I),T_1^\pprime}\\
	=& R_{S_1-M+S_2-L,T_1^\pprime}.
	\end{align*}
	Otherwise if (\ref{eq:e4826ad1}) does not hold, then
	$$\sum_{J\subset T_1^\pprime\cap S_2} R_{S_1^\pprime+I,T_1^\pprime-J}R_{S_1^\prime-M+S_2-L-I, J}=0.$$
	Therefore
	\begin{align*}
	\RN{5} =& \sum_{\substack{I\subset T_1^\pprime\cap S_2\\M\subset S_1^\prime\\S_1\cap S_2\subset L\subset S_1^\prime-M+S_2-I}} R_{S_1-M+S_2-L,T_1^\pprime}R_{M, T_1^\prime}R_{L, T_2}\\
	=& \sum_{\substack{S_1^\pprime\cap S_2\subset L\subset S_1^\prime+S_2\\M\subset S_1^\prime\backslash L\\I\subset (T_1^\pprime\cap S_2)\backslash L}} R_{S_1-M+S_2-L,T_1^\pprime}R_{M, T_1^\prime}R_{L, T_2}\\
	=& \sum_{\substack{S_1^\pprime \cap S_2\subset L\subset S_1^\prime+S_2\\M\subset S_1^\prime\backslash L}} 2^{(T_1^\pprime\cap S_2)\backslash L}R_{S_1-M+S_2-L,T_1^\pprime}R_{M, T_1^\prime}R_{L, T_2}\\
	\end{align*}
	The summand is nontrivial only if 
	$$(T_1^\pprime\cap S_2)\backslash L=\emptyset \text{ \hspace{5pt}and\hspace{5pt} } S_1-M+S_2-L<\max(T_1^\pprime)=b_1,$$
	which is equivalent to
	$$(T_1^\pprime\cap S_2)\subset L \text{ \hspace{5pt}and\hspace{5pt} } N_{b_1+1, b_2}\cap S_2\subset L.$$
	Note that in the summation we also require $(S_1^\pprime\cap S_2)\subset L$. 
	Hence 
	$$(S_1^\pprime+T_1^\pprime+N_{b_1+1,b_2})\cap S_2=S_2\subset L$$
	which implies $S_2=L$. 
	Therefore 
	\begin{align*}
	\RN{5} =& \sum_{M\subset S_1^\prime} R_{S_1-M,T_1^\pprime}R_{M, T_1^\prime}R_{S_2, T_2}\\
	=& R_{S_1,T_1}R_{S_2, T_2}
	\end{align*}

	\ref{rel:51b194a6}. 
	This follows from \ref{rel:a078c1d9} because of the symmetry given by Proposition \ref{prop:8ac1afe}.
\end{proof}

\begin{theorem}\label{thm:a26989b6}
	Conjecture \ref{conj:5f7d758} holds in $HX_7$.
\end{theorem}

We will prove this theorem by computing $HX_7$ completely in Section \ref{sec:11e59387}.

\subsection{Nilpotent freeness}
Conjecture \ref{conj:5f7d758} and computations in a large range of dimensions let the computer admit the following conjecture.

\begin{conjecture}\label{conj:2a224128}
	$HX$ is nilpotent free. 
\end{conjecture}

This means that if $x\in HX$ is nonzero, then no power of $x$ is zero. We shall explore a way to study this conjecture using \Groebner{} bases in [Theorem 4.16] below.

\begin{theorem}\label{thm:7392b792}
	Conjecture \ref{conj:2a224128} holds in $HX_7$.
\end{theorem}

This will be proved together with Theorem \ref{thm:a26989b6} by computing $HX_7$ in Section \ref{sec:11e59387}.
It is strong evidence for the two conjectures since the subalgebra $HX_7\subset HX$ together with $h_7$ generates a subalgebra isomorphic to $HX$ in dimensions $t-s\le 285$.

The following theorem is a weaker form of Conjecture \ref{conj:2a224128}, but we are able to prove it in $HX$ in all dimensions.
	
\begin{theorem}\label{thm:512ce8e9}
	For any $x\in HX$ and $n\ge 0$, if $h_nx\neq 0$, then $h_n^ix\neq 0$ for all $i$.
\end{theorem}

The rest of this section is dedicated to proving this theorem. The method used in the proof is based on general results and may be useful in other context.

\begin{proposition}\label{prop:dd4d11fa}
	Let $A$ be an abelian group with a homomorphism $d:A\to A$ such that $d^2=0$. 
	We write $Ah^i$ for a copy of $A$ labeled by $h^i$. 
	Consider 
	$$A[h]=\bigoplus_{i=0}^{\infty} Ah^i$$
	and define a differential $D: A[h]\to A[h]$ by
	$$D(ah^i)=d(a)h^{i+1}$$
	for $a\in A$. 
	Then
	\begin{equation}\label{eq:827f1292}
		H(A[h], D)\iso \Ker d\oplus\bigoplus_{i=1}^\infty H(A,d)h^i.
	\end{equation}
	Define a homomorphism $m_h: A[h]\to A[h]$ by 
	$$m_h(ah^i)=ah^{i+1}.$$
	Then the induced map $Hm_h$ on the homology $H(A[h], D)$ when restricted to $\Ker d$ in (\ref{eq:827f1292}) is given by
	\begin{equation}\label{eq:7d298732}
		Hm_h: \ker d\to \Ker d/\Image d=H(A,d)\iso H(A,d)h.
	\end{equation}
	\end{proposition}
	
	\begin{proof}
		We have the following commutative diagram between the differentials $D$ and $d$.  
		$$\xymatrix{
			\cdots & \cdots & \cdots\\
			Ah^2 \ar[u]^{m_h} \ar[ru]^D \ar[r]^d & Ah^2 \ar[u]^{m_h} \ar[r]^d \ar[ru]^D & Ah^2\ar[u]^{m_h}\\
			Ah \ar[r]^d \ar[u]^{m_h} \ar[ru]^D & Ah \ar[r]^d \ar[ru]^D \ar[u]^{m_h} & Ah\ar[u]^{m_h}\\
			A \ar[r]^d \ar[u]^{m_h} \ar[ru]^D & A \ar[r]^d \ar[ru]^D \ar[u]^{m_h} & A\ar[u]^{m_h}
		}$$
		It is straightforward to see that on the bottom row, the homology with respect to $D$ is the same as $\Ker d$. 
		On the rows above, since the vertical arrows $m_h$ are isomorphisms, the homology of $D$ is the same as the homology of $d$. 
		Hence we get (\ref{eq:827f1292}) and (\ref{eq:7d298732}) follows easily.
	\end{proof}
	
	\begin{corollary}\label{cor:8cc765ad}
		The homology $H(A[h],D)$ satisfies the property that
		if $Hm_hx\neq 0$ for some $x\in H  (A[h],D)$, then $Hm_h^ix\neq 0$ for all $i\ge 1$.
	\end{corollary}
	
	\begin{proof}
		By (\ref{eq:827f1292}) the image of $Hm_h$ on $H(A[h],D)$ is $\bigoplus_{i=1}^\infty H(A,d)h^i$ where the map $Hm_h$ is injective.
	\end{proof}
	
	\begin{remark}
		If $A$ is a ring, then $A[h]$ is isomorphic to the polynomial ring with coefficients in $A$ and the map $m_h$ is the same as the multiplication by $h=1\cdot h\in Ah$.
	\end{remark}
	
	\begin{proposition}\label{prop:670687b9}
		Let $V$ be a vector space over a field $F$ equipped with a basis $B=\{v_1,v_2,\dots\}$, a weight function $w:B\to \bZ_{\ge 0}$, two fixed integers $0\le m<n$ and a linear map $d: V\to V$ such that
		\begin{enumerate}
			\item $d^2=0$;
			\item for any $v\in B$, $dv=\delta_mv+\delta_nv$ where $\delta_mv$ is a linear combination of vectors in $B$ of weight $w(v)+m$ and $\delta_nv$ is a linear combination of vectors in $B$ of weight $w(v)+n$.
		\end{enumerate}
	
		Consider the vector space $V[h]$ and the linear map 
		$$D=\delta_m+m_h\delta_n: V[h]\to V[h].$$
		Then $D^2=0$ and $H(V[h], D)$ satisfies the property that if $Hm_hx\neq 0$ for some $x\in H(V[h],D)$, then $Hm_h^ix\neq 0$ for all $i$.
	\end{proposition}
	\begin{proof}
		We have
		$$d^2=\delta_m^2+(\delta_m\delta_n+\delta_n\delta_m)+\delta_n^2.$$
		Hence $d^2=0$ implies that $\delta_m^2=0$, $\delta_m\delta_n+\delta_n\delta_m=0$ and $\delta_n^2=0$. 
		Therefore
		$$D^2=\delta_m^2+m_h(\delta_m\delta_n+\delta_n\delta_m)+m_h^2\delta_n^2=0.$$
	
		We choose
		$$B[h]=\{vh^i\in V[h]:v\in B, i\ge 0\}$$
		as a basis for $V[h]$ and extend the weight function $w$ to
		$$w:B[h]\to \bZ_{\ge 0}$$
		by $w(vh^i)=w(v)$ for $v\in B$. 
		We filter $V$ and $V[h]$ by
		$$F_pV=\Span\left\{v\in B:w(v)\ge p\right\},$$
		$$F_pV[h]=\Span\left\{v\in B[h]: w(v)\ge p\right\},$$
		and consider the resulting spectral sequences respectively:
		$$(E_r,d_r)\Rightarrow H(V, d),$$
		$$(\bar E_r,\bar d_r)\Rightarrow H(V[h], D).$$
		Here we write $m_{h,r}$ for the map on $\bar E_r$ induced by $m_h$. 
		In both spectral sequences, nontrivial differentials must be of length $m$ or $n$ by the definitions of $d$ and $D$. 
		Hence $H(E_m, d_m)\iso H(V, \delta_m)$ and $H(\bar E_m, \bar d_m)\iso H(V[h], \delta_m)$. 
		We have
		$$\bar E_n\iso H(V[h], \delta_m)\iso H(V, \delta_m)[h]\iso E_n[h].$$
		Since $D=\delta_m+m_h\delta_n$, we have
		$$\bar d_n=m_{h,n}d_n.$$
		Therefore we can apply Proposition \ref{prop:dd4d11fa} to the case $(A, h)=(E_n, d_n)$ and use Corollary \ref{cor:8cc765ad} to get that for $x\in \bar E_\infty=H(\bar E_n,d_n)$, if $m_{h,\infty}(x)\neq 0$, then $m_{h,\infty}^i(x)\neq 0$ for all $i$.
	
		Next we show that there is no extension problem for computing the kernel of
		$$Hm_h:H(A[h], D)\to H(A[h], D),$$
		in other words, if $y\in H(A[h], D)$ projects to $x\in \bar E_\infty$, then $m_{h,\infty}(x)=0$ implies that $Hm_h(y)=0$.
		In fact, by (\ref{eq:7d298732}) from Proposition \ref{prop:dd4d11fa}, we have
		$$\Ker m_{h,\infty}=\Image d_n,$$
		which means that if $m_{h,\infty}(x)=0$ then $x$ is represented by some $v\in V=Vh^0\subset V[h]$ such that $v=\delta_n(v^\prime)$ for some $v^\prime\in V$ with $\delta_m(v^\prime)=0$.
		Since $d=\delta_m+\delta_n$, we have $v=d(v^\prime)$ and $D(v)=m_hd_n^2(v^\prime)=0$. 
		Hence $v$ is a $D$-cycle in $V[h]$ and it represents $y_0\in H(A[h], D)$ such that $y-y_0\in F_{p+1}H(A[h],D)$. 
		Therefore
		$$m_h(v)=m_hd_n(v^\prime)=D(v^\prime)\in V[h]$$
		and it implies that $Hm_h(y_0)=0$ in $H(V[h], D)$. 
		We now have $Hm_h(y-y_0)=0$ and we can repeat the discussion on $y-y_0$. 
		By induction each summand of $y$ homogeneous in weight is annihilated by $m_h$. 
		Therefore $m_h(y)=0$.
	
		Combining the above, if $y\in H(A[h], D)$ projects to $x\in \bar E_\infty$ and $Hm_h(y)\neq 0$, then $m_{h,\infty}(x)\neq 0$. 
		Therefore $m_{h,\infty}^i(x)\neq 0$ for all $i$, which implies that $Hm_h^i(y)\neq 0$ for all $i$.
	\end{proof}

	We are now ready to prove Theorem \ref{thm:512ce8e9}.
	\begin{proof}[Proof of Theorem \ref{thm:512ce8e9}]
		First we assign each generator $R_{ij}$ of $X$ a weight $w$ given by
		$$
		w(R_{ij})=\begin{cases}
			0 & ~i=n\text{ or } j=n+1,\\
			2 & ~i=n+1\text{ or } j=n,\\
			1 & \text{ otherwise}.
		\end{cases}
		$$
		Then we choose the set of all the monomials in $R_{ij}$ as an $\bF_2$-basis of $X$. 
		We assign each monomial a weight $w$ given by
		$$w(\prod R_{i_kj_k}^{e_k})=\sum e_kw(R_{i_kj_k}).$$
	
		Consider the differential
		$$dR_{ij}=\sum_{i<k<j} R_{ik}R_{kj}.$$
		The right-hand side is homogeneous in weight and $d$ increases the weight by one with the following two exceptions
		$$dR_{i,n+1}=R_{i,n}R_{n,n+1}+\sum_{i<k<n} R_{i,k}R_{k,n+1},$$
		$$dR_{n,j}=R_{n,n+1}R_{n+1,j}+\sum_{n+1<k<j} R_{n,k}R_{k,j}.$$
	
		We define a subalgebra of $X$
		$$Y=\bF_2[R_{ij}: 0\le i<j \text{ and } (i, j)\neq (n, n+1)],$$
		with differential given by
		$$d_YR_{ij}=\sum_{i<k<j} R_{ik}R_{kj} \text{ if }i\neq n \text{ or } j\neq n+1$$
		$$d_YR_{i,n+1}=R_{i,n}+\sum_{i<k<n} R_{i,k}R_{k,n+1},$$
		$$d_YR_{n,j}=R_{n+1,j}+\sum_{n+1<k<j} R_{n,k}R_{k,j}.$$
		It is not hard to check that $d_Y^2=0$. 
		Note that
		$X=Y[R_{n,n+1}]$. We can choose the set of monomials of $Y$ as its $\bF_2$-basis and it is a sub-basis of the basis for $X$. 
		We restrict the weight function $w$ for $X$ to this sub-basis and obtain the weight function for $Y$. 
		Now we define
		\begin{align*}
			& \delta_{Y,1}R_{ij} = d_YR_{ij} \hspace{0.5cm}\text{ if } i\neq n \text{ and } j\neq n+1\\
			& \delta_{Y,1}R_{i,n+1} = \sum_{i<k<n} R_{i,k}R_{k,n+1},\\
			& \delta_{Y, 2}R_{i,n+1} = R_{i,n},\\
			& \delta_{Y,1}R_{n,j} = \sum_{n+1<k<j} R_{n,k}R_{k,j},\\
			& \delta_{Y, 2}R_{n,j} = R_{n+1,j}.
		\end{align*}
		We see that the right-hand sides are all homogeneous in weight and $\delta_{Y,1}$ increases the weight by $1$ while $\delta_{Y,2}$ increases the weight by $2$.
		We have
		$$d_X=\delta_{Y, 1}+R_{n,n+1}\delta_{Y, 2}$$
		where $d_X=d$ is the differential of $X$. 
		Note that $h_n\in HX$ is represented by $R_{n,n+1}$ and $w(R_{n,n+1})=0$. 
		We apply Proposition \ref{prop:670687b9} to the case $V=Y$ and $V[h]=X=Y[R_{n,n+1}]$ and get our conclusion immediately.
	\end{proof}

\subsection{Massey products in \mybm{$H^*(E^0\sA)$}{H*(E0A)}}
A theorem due to Gugenheim and May \cite{GM74} states that for a connected algebra $A$, the cohomology $H^*(A)$ is generated under matric Massey products by $H^1(A)$. 
As a concrete example, we will show how to obtain the indecomposables $h_{S, T}\in \sH$ from $h_i$ under matric Massey products.

\begin{theorem}\label{thm:d55b92ee}
	For $h_{S, T}\in \sH$ where
	$$S\cup T=\{k,k+1,\dots,k+2n-1\}$$
	we have
	\begin{equation*}
	h_{S, T}\in \Massey{h_k, h_{k+1}, \dots, h_{k+2n-2}, h_{S-\{k\}, T-\{k+2n-1\}}}.
	\end{equation*}
\end{theorem}
\begin{proof}
	The proof is the same for $k>0$ as the case for $k=0$ and hence we assume that $k=0$. 
	By the definition of matric Massey products, we must find a defining system $(A_{ij})$ with $0\le i<j\le 2n$ and $(i, j)\neq (0, 2n)$ such that
	\begin{equation}\label{eq:8100d1d0}
		A_{i,{i+1}}=\begin{cases}
		R_{S-\{s_1\}, T-\{t_n\}} & \text{ if } 0\le i<2n-1\\
		R_{i, i+1} & \text{ if } i=2n-1,
		\end{cases}
	\end{equation}
	\begin{equation}\label{eq:6cbe8f7}
		d A_{ij}=\sum_{i<k<j} A_{ik}A_{kj}
	\end{equation}
	and
	\begin{equation}\label{eq:3}
		\widetilde A_{0, 2n}=\sum_{0<k<2n}A_{0, k}A_{k, 2n}=R_{S, T}.
	\end{equation}
	In fact, for $0\le i< j\le 2n-1$, if we let $A_{ij}=R_{ij}$, then (\ref{eq:8100d1d0}) and (\ref{eq:6cbe8f7}) are automatically true by (\ref{eq:99758f58}).
	
	We adopt the convention that $R_{S-\{0\}, T-\{i\}}=0$ if $i\notin T$.
	We let $A_{i, 2n}=R_{S-\{0\}, T-\{i\}}$ ($i\neq 0$). 
	Now for (\ref{eq:6cbe8f7}) we only have to show
	$$d A_{i, 2n}=\sum_{i<k<2n} A_{0, k}A_{k, 2n}$$
	i.e.
	$$d R_{S-\{0\}, T-\{i\}}=\sum_{i<k<2n} R_{ik}R_{S-\{0\}, T-\{k\}}.$$
	If $i\in T$, by (\ref{item:374301ae}) and (\ref{item:cea34d63}) of Proposition \ref{prop:af5d3cd9}.
	$$d R_{S-\{0\}, T-\{i\}}=R_{S-\{0\}+\{i\}, T}=\sum_{i<k<2n} R_{ik}R_{S-\{0\}, T-\{k\}}.$$
	If $i\notin T$, the right-hand side is zero because $i\in S$ and hence
	$$\sum_{i<k<2n} R_{ik}R_{S-\{0\}, T-\{i\}}=R_{S-\{0\}+\{i\}, T}=0$$
	since $R_{S-\{0\}+\{i\}, T}$ is the determinant of a matrix with repeating rows.
	
	Finally, to show (\ref{eq:3}) we have
	$$\sum_{0<k<2n}A_{0k}A_{k, 2n}=\sum_{0<k<2n} R_{0k}R_{S-\{0\}, T-\{k\}} = R_{S, T}.$$
\end{proof}

Note that in Theorem \ref{thm:d55b92ee} the $s$ degree of $h_{S-\{k\}, T-\{k+2n-1\}}$ is one less than the $s$ degree of $h_{S, T}$. 
The element $h_{S-\{k\}, T-\{k+2n-1\}}$ is either an element of $\sH$ or a product of elements in $\sH$. 
Hence by induction on $s$ all indecomposables $h_{S, T}\in \sH$ can be obtained inductively from $h_i$ under matric Massey products.

\begin{remark}
	Although the indecomposables $b_{ij}=[R_{ij}^2]$ are represented by simpler cycles, the decompositions of $b_{ij}$ by matric Massey products are more complicated. 
	The author has followed the proofs in the work of Gugenheim and May \cite[Chapter 5]{GM74} and produced a computer program to write elements in $HX$ by ``canonically defined matric Massey products'' as defined in \cite[Theorem 5.6]{GM74}. 
	It means that we can generate a sequence of matrices $W_1, W_2, \dots$ such that we can write everything in $HX$ in terms of
	$$\langle W_1, \dots, W_n, V_{n+1}\rangle$$
	with indeterminacies where $V_{n+1}$ is some column matrix (not unique even if the sequence $W_1,W_2,\dots$ is fixed). 
	One can simplify the canonical form if $V_{n+1}$ contains zero entries. 
	Here we list some decompositions of $b_{ij}$ via this method.
	$$b_{02}\in \langle h_0, h_1, h_0, h_1\rangle\subset \left\langle h_0, h_1, \begin{pmatrix}
		h_0 & h_2
	\end{pmatrix}, 
	\begin{pmatrix}
		h_1 \\ 0
	\end{pmatrix}\right\rangle$$
	$$b_{03}\in \left\langle h_0, h_1, \begin{pmatrix}
		h_0 & h_2
	\end{pmatrix}, 
	\begin{pmatrix}
		h_2 & 0\\
		h_0 & h_3
	\end{pmatrix}, 
	\begin{pmatrix}
		h_1 \\ 0
	\end{pmatrix}, 
	h_2\right\rangle$$
	$$b_{04}\in \left\langle h_0, h_1, \begin{pmatrix}
		h_0 & h_2
	\end{pmatrix}, 
	\begin{pmatrix}
		h_2 & 0\\
		h_0 & h_3
	\end{pmatrix}, 
	\begin{pmatrix}
		h_1 & h_3\\
		0 & h_0
	\end{pmatrix}, 
	\begin{pmatrix}
		h_3 \\ h_1
	\end{pmatrix}, 
	h_2,h_3\right\rangle.$$
	Here $W_1=h_0$, $W_2=h_1$, $W_3=\begin{pmatrix}
		h_0 & h_2
	\end{pmatrix}$, $\begin{pmatrix}
		h_2 & 0\\
		h_0 & h_3
	\end{pmatrix}$ is a submatrix of $W_4$, $\begin{pmatrix}
		h_1 & h_3\\
		0 & h_0
	\end{pmatrix}$ is a submatrix of $W_5$, $\dots$.
\end{remark}

\section{The May Spectral Sequence}\label{sec:b0f0886a}
The main goal of this section is to compute the differentials on $H^*(E^0\sA)$ in the May spectral sequence.

In this section we use the method of Ravenel \cite{Ravenel} to obtain the May spectral sequence. 
The reason behind this is that the associated graded algebra $E^0_R\sA$ of the Steenrod algebra by the filtration suggested by Ravenel is $E_K^0E^0\sA$, which is Priddy's associated homogeneous Koszul algebra of May's associated graded algebra of $\sA$. 
When we interact with the cobar complex this filtration is more efficient computationally.

\subsection{The cobar complex}
Recall that if $I$ is the augmentation ideal of the dual Steenrod algebra $\sA_*$, then the cobar complex $C(\sA_*)$ is the tensor algebra $T^*(I)$ with $d: I^{\otimes n}\to I^{\otimes (n+1)}$ given by
\begin{equation}\label{eq:ea84c17e}
d(\alpha_1\otimes\cdots\otimes\alpha_n)=\sum_i\sum \alpha_1\otimes\cdots\otimes\alpha_{i-1}\otimes \alpha_{i}^\prime\otimes \alpha_i^{\prime\prime}\otimes \alpha_{i+1}\otimes\cdots\otimes \alpha_n
\end{equation}
where 
$$\psi(\alpha_i)=\alpha_i\otimes 1+1\otimes \alpha_i+\sum \alpha_i^\prime\otimes \alpha_i^{\prime\prime}$$
in $\sA_*$. 
Then $H^*(\sA)=\Ext_{\sA}(\bF_2, \bF_2)=HC(\sA_*)$.

\begin{definition}\label{def:826d0f73}
	The weight function $w$ on $\sA_*$ is given by setting $w(\xi_j^{2^i})=2j-1$, i.e.
	$$w(\xi_1^{r_1}\cdots\xi_n^{r_n})=\sum_k (2k-1)a_{k,i}$$
	where $r_k=\sum_i a_{k,i}2^i$ is the 2-adic expansion. 
	
	We also define $w$ on $C(\sA_*)$ by
	$$w(\alpha_1\otimes\cdots\otimes\alpha_n)=w(\alpha_1)+\cdots+w(\alpha_n).$$
\end{definition}

\begin{definition}\label{def:d05cc3a5}
	The filtrations $F_p(\sA_*)$ and $F_p(C(\sA_*))$ are given by elements in $\sA_*$ and $C(\sA_*)$ with weight $\le p$ respectively. 
	Note that we are using an increasing filtration indexed positively. 
	The associated graded algebra by this filtration is denoted with $E^0_R\sA_*$.
\end{definition}

It follows that the associated graded algebra $E^0_R\sA_*$ is an exterior algebra generated by the projections of $\tilde R_{ij}=\xi^{2^i}_{j-i}$ ($0\le i<j$), which are primitive. 
Therefore we have the following.
\begin{proposition}\label{prop:619ea206}
	The $E_1$ page of the spectral sequence determined by the filtration $F_p(C(\sA_*))$ is isomorphic to $X=\bF_2[R_{ij}: 0\le i<j]$ with $d_1(R_{ij})=\sum_k R_{ik}R_{kj}$. 
	Here $R_{ij}$ corresponds to the primitive generator $\tilde R_{ij}=\xi^i_{j-i}$ in the associated graded algebra.
\end{proposition}

\begin{remark}\label{rem:1}
	$d_0(x)=\sum x_1\otimes x_2$ in $E_0^{p,q}=(F_pC(\sA_*)/F_{p-1}C(\sA_*))_{s=p+q}$ if $x$ is a monomial in $\sA_*$ where the summation is taken over all ordered monomial pairs $(x_1, x_2)$ such that $x=x_1x_2$ in the augmentation ideal of $E^0_R\sA_*$. 
	In particular, $d_0(\tilde R_{ij}\tilde R_{kl})=\tilde R_{ij}\otimes \tilde R_{kl}+\tilde R_{kl}\otimes \tilde R_{ij}$.
\end{remark}

Since  $w(\xi^{2^i}_j)=2j-1$ is odd and the $s$ degree of all differentials in the spectral sequence is $1$, all nontrivial differentials $d_r$ in the spectral sequence must have odd index $r$. 
The following is the comparison between the spectral sequence obtained by the method of Ravenel and the May spectral sequence.

\begin{center}
\captionsetup{type=table}
\captionof{table}{}
\begin{tabular}{|c|c|}\hline\label{tab:1}
	Ravenel & May\\\hline
	$E_1=X$ & $E_1=C(E^0\sA_*)$\\\hline
	$(E_{2r-1}, d_{2r-1})$, $r\ge 2$ & $(E_{r}, d_{r})$, $r\ge 2$\\\hline
	$E_2=E_3=H^*(E^0\sA_*)$ & $E_2=H^*(E^0\sA_*)$\\\hline
\end{tabular}
\end{center}

\subsection{The differentials in \mybm{$H^*(E^0\sA)$}{H*(E0A)}}
We will use the filtration in the previous section and we will therefore use the notations in the left-hand side of Table \ref{tab:1}. 
We want to compute the $d_3$ differentials on $H^*(E^0\sA)$.

The following was already proven by May.
\begin{itemize}
	\item $d_3(b_{02})=h_1^3+h_0^2h_2,$
	\item $d_3(b_{ij})=h_{i+1}b_{i+1,j}+b_{i,j-1}h_{j+1}, ~j-i>2,$
    \item $d_3(h_i)=0$,
    \item $d_3(h_i(1))=h_ih_{i+2}^2$,
    \item $d_3(h_i(1, 3))=h_ih_{i+2}h_{i+2}(1)+h_i(1)h_{i+4}^2$,
    \item $d_3(h_i(1, 2))=h_{i+3}h_i(1, 3)$.
\end{itemize}

The main goal of this section is to determine the differentials on $h_{S, T}\in \sH$. 
Then all $d_3$ differentials in $H^*(E^0\sA)$ will be determined if Conjecture \ref{conj:2cc43131} is true.

\begin{definition}
	We say that $\alpha=\alpha_1\otimes\cdots\otimes \alpha_n\in C(\sA_*)$ is a monomial in $C(\sA_*)$ if each $\alpha_k$ is a monomial in $\sA_*$. 
	Note that all monomials form an additive basis of $C(\sA_*)$. 
	We say that the monomial $\alpha$ is simple if each $\alpha_k=\tilde R_{i_kj_k}$ for some $i_k, j_k$. 
	Note that $d_0(\alpha)=0$ in the $E_0$ page if $\alpha$ is a simple monomial.
\end{definition}

\begin{definition}
	We denote the span of simple monomials in $C(\sA_*)$ by $S(\sA_*)$ and the span of non-simple monomials by $S(\sA_*)^\perp$. 
	Note that we have $C(\sA_*)=S(\sA_*)\oplus S(\sA_*)^\perp$.
\end{definition}

\begin{proposition}\label{prop:1afacc52}
	The map $g: (E_0, d_0)\to E_1$ (with trivial differentials) given by
	$$g(\alpha)=\begin{cases}
	R_{i_1j_1}\cdots R_{i_nj_n} & \text{ if } \alpha=\tilde R_{i_1j_1}\otimes\cdots\otimes \tilde R_{i_nj_n}\in S(\sA_*)\\
	0 & \text{ if } \alpha \in S(\sA_*)^\perp
	\end{cases}$$
	is a homology isomorphism.
\end{proposition}
\begin{proof}
	It is clear that the homology classes $[\tilde R_{ij}]$ generate $E_1$ while $g$ is multiplicative. 
	Therefore $g$ induces an isomorphism $g_*:H(E_0, d_0)\to E_1$.
\end{proof}

\begin{remark}\label{rem:559b5258}
	We can project suitable chains in $C(\sA_*)$ into cycles in $E_r$ ($r\ge 1$) via $g$.
\end{remark}

\begin{lemma}\label{lem:77e89a29}
	If $\alpha\in C(\sA_*)$ is a non-simple monomial and $\beta$ is a simple monomial summand of $d(\alpha)$, then either $\beta$ is a summand of $d_0(\alpha)$ in $E_0$ or $w(\beta)\le w(\alpha)-2$.
\end{lemma}

\begin{proof}
	Write $\alpha=\alpha_1\otimes\cdots\otimes\alpha_n$. 
	If there is a simple summand $\beta$ of $d(\alpha)$, then there must be at most one factor $\alpha_\ell$ which is not equal to some $\tilde R_{ij}$ by (\ref{eq:ea84c17e}).
	Since $\alpha$ is not simple, there must be exactly one such $\alpha_\ell$.
	Assume that $\beta$ does not appear in $d_0(\alpha)$ in $E_0$.
	To obtain the simple summand $\beta_\ell\otimes \beta_{\ell+1}$ in $d(\alpha_\ell)$, we have to replace at least one factor $\tilde R_{ij}$ of $\alpha_\ell$ with $\tilde R_{kj}\otimes \tilde R_{ik}$ and either $\tilde R_{kj}$ or $\tilde R_{ik}$ will meet another copy of itself coming from another factor of $\alpha_\ell$ to become $\tilde R_{kj}^2=\tilde R_{k+1,j+1}$ or $\tilde R_{ik}^2=\tilde R_{i+1,k+1}$.
	Noting that $w(\tilde R_{kj}\otimes \tilde R_{ik})=w(\tilde R_{ij})-1$ and in general 
	$$w((\tilde R_{ij})^2)=w(\tilde R_{i+1,j+1})=2w(\tilde R_{ij})-(2(j-i)-1)\le 2w(\tilde R_{ij})-1,$$
	we see that $w(\beta)\le w(\alpha)-2$.
\end{proof}

\begin{lemma}\label{lem:b0fa534e}
	Assume that $$d(a_p+a_{p-1})=a_{p-2}+a_{p-3}+b_{p-3}\mod F_{p-4}C(\sA_*)$$ in $C(\sA_*)$, where $a_{p-i}$ consists of terms of weight $p-i$, $i=0,1,2,3$ and $b_{p-3}$ consists of terms of weight $p-3$. 
	Assume further that $a_p, b_{p-3}\in S(\sA_*)$ and $a_{p-1},a_{p-2}, a_{p-3}\in S(\sA_*)^\perp$. 
	Then $d_3(a_p)=b_{p-3}$ in the $E_3$ page of the spectral sequence determined by $F_pC(\sA_*)$.
\end{lemma}

\begin{proof}
	Note that $d(a_{p-2}+a_{p-3}+b_{p-3})=d^2(a_p+a_{p-1})=0$. 
	Hence we have $d_0(a_{p-2})=0$ in the $E_0$ page. 
	By Proposition \ref{prop:1afacc52}, $g(a_{p-2})=0$ in $E_1$ implies that $a_{p-2}$ is a boundary in $E_0$. 
	Therefore we can find $a_{p-2}^\prime\in F_{p-2}C(\sA_*)\cap S(\sA_*)^\perp$ such that $d_0(a_{p-2}^\prime)=a_{p-2}$ in $E_0$. 
	By Lemma \ref{lem:77e89a29}, we have $$d(a_{p-2}^\prime)=a_{p-2}+c_{p-3}\mod F_{p-4}C(\sA_*)$$ where $c_{p-3}\in F_{p-3}C(\sA_*)\cap S(\sA_*)^\perp$. 
	Now consider
	$$d(a_p+a_{p-1}+a_{p-2}^\prime)=b_{p-3}+c_{p-3}\mod F_{p-4}C(\sA_*).$$
	By Remark \ref{rem:559b5258} we have $d_3(a_p)=b_{p-3}$ in $E_3$.
\end{proof}

\newcommand{\otimesop}{\mathop{\bar\otimes}}
\newcommand{\bigotimesop}{\mathop{\overline\bigotimes}}

Now we are ready to prove the main theorem of this section. The formula for $d_3h_{S,T}$ below is actually inspired by some computations (with possible indeterminacies) using the Massey product decomposition of $h_{S,T}$ in Theorem \ref{thm:d55b92ee}.

\begin{theorem}\label{thm:9f9e1790}
	The differentials on $h_{S, T}\in \sH$ are given by the following
	$$d_3h_{S, T}=\sum_{s\in S,~ s+1\in T} h_{s+1, s+2}h_{S-\{s\}+\{s+1\}, T-\{s+1\}+\{s\}}.$$
\end{theorem}

\begin{proof}
	We are going to compute the differentials via the cobar complex $C(\sA_*)$. 
	Note that in $C(\sA_*)$, the differentials are given by $$d(\tilde R_{ij})=\sum_{k=i+1}^{j-1} \tilde R_{kj}\otimes \tilde R_{ik}.$$
	To make the right-hand side look more like matrix multiplications, in this proof we are going to write
	$$d(\tilde R_{ij})=\sum_{k=i+1}^{j-1} \tilde R_{ik}\otimesop \tilde R_{kj}$$
	where $x\otimesop y=y\otimes x$. 
	We also write $$\bigotimesop_{i=1}^n \alpha_i=\alpha_n\otimes \alpha_{n-1}\otimes\cdots \otimes \alpha_1.$$

	The homology class $h_{S, T}\in E_3=H^*(E^0\sA)$ can be represented in $E_0$ by $$\alpha=\sum_{\sigma\in \Sigma_n}\alpha_\sigma=\sum_{\sigma\in \Sigma_n} \tilde R_{s_1t_{\sigma(1)}}\otimesop\cdots\otimesop \tilde R_{s_nt_{\sigma(n)}}.$$
	
	Note that $d_1(\alpha)=0$ in $E_1$ but $d(\alpha)\neq 0$ in $C(\sA_*)$ because $C(\sA_*)$ is not commutative. 
	In fact, every monomial summand of $d(\alpha)$ can be paired with another summand the two being equal in the $E_1$ page. 
	Two typical examples are pairs $(d_{is_j}\alpha_\sigma, d_{is_j}\alpha_{\sigma^\prime})$ and $(d_{it_{\sigma(j)}}\alpha_\sigma,d_{jt_{\sigma(j)}}\alpha_{\sigma^\prime})$ where
	$$d_{is_j}\alpha_\sigma=\tilde R_{s_1t_{\sigma(1)}}\otimesop\cdots\otimesop \tilde R_{s_is_j}\otimesop\tilde  R_{s_jt_{\sigma(i)}}\otimesop\cdots\otimesop \tilde R_{s_jt_{\sigma(j)}}\otimesop\cdots\otimesop\tilde  R_{s_nt_{\sigma(n)}}$$
	$$d_{js_j}\alpha_{\sigma^\prime}=\tilde R_{s_1t_{\sigma(1)}}\otimesop\cdots\otimesop \tilde R_{s_is_j}\otimesop \tilde R_{s_jt_{\sigma(j)}}\otimesop\cdots\otimesop \tilde R_{s_jt_{\sigma(i)}}\otimesop\cdots\otimesop \tilde R_{s_nt_{\sigma(n)}}$$
	and
	$$d_{it_{\sigma(j)}}\alpha_\sigma=\tilde R_{s_1t_{\sigma(1)}}\otimesop\cdots\otimesop \tilde R_{s_it_{\sigma(j)}}\otimesop \tilde R_{t_{\sigma(j)}t_{\sigma(i)}}\otimesop\cdots\otimesop \tilde R_{s_jt_{\sigma(j)}}\otimesop\cdots\otimesop \tilde R_{s_nt_{\sigma(n)}}$$
	$$d_{jt_{\sigma(j)}}\alpha_{\sigma^\prime}=\tilde R_{s_1t_{\sigma(1)}}\otimesop\cdots\otimesop\tilde  R_{s_it_{\sigma(j)}}\otimesop\cdots\otimesop \tilde R_{s_jt_{\sigma(j)}}\otimesop \tilde R_{t_{\sigma(j)}t_{\sigma(i)}}\otimesop\cdots\otimesop \tilde R_{s_nt_{\sigma(n)}}.$$
	Here the permutation $\sigma^\prime$ is the same as $\sigma$ but with values $\sigma(i)$ and $\sigma(j)$ swapped and $d_{ik}\alpha_\sigma$ is the summand of $d(\alpha_\sigma)$ which replaces $\tilde R_{s_it_{\sigma(i)}}$ in $\alpha_\sigma$ with $\tilde R_{s_ik}\otimes \tilde R_{kt_{\sigma(i)}}$.
	
	Observe the typical example $$d_0(ab\otimesop c\otimesop d+b\otimesop ac\otimesop d+b\otimesop c\otimesop ad)=a\otimesop b\otimesop c\otimesop d+b\otimesop c\otimesop d\otimesop a$$
	where each $a,b,c,d$ is equal to some $\tilde R_{st}$. 
	We can find a chain in $C(\sA_*)$ whose $d_0$-boundary is the sum of either typical pair above. 
	In fact, we define
	$$\beta=\sum_\sigma\sum_{i<k<j}\gamma_{\sigma,ijk}+\sum_\sigma \sum_{\substack{i<j\\ \sigma(i)>\sigma(j)}}\gamma_{\sigma,ijj}+\sum_\sigma\sum_{\substack{i<k\le j\\\sigma(i)>\sigma(j)}}\delta_{\sigma,ijk}$$
	where
	$$\gamma_{\sigma,ijk}=\bigotimesop_{l=1}^n \gamma_{\sigma,ijkl},\hspace{10pt} \delta_{\sigma,ijk}=\bigotimesop_{l=1}^n \delta_{\sigma,ijkl}$$
	and
	$$\gamma_{\sigma,ijkl}=\begin{cases}
	\tilde R_{s_is_j} & \text{ if } l=i\\
	\tilde R_{s_jt_{\sigma(i)}} \tilde R_{s_lt_{\sigma(l)}} & \text{ if } l=k\\
	\tilde R_{s_lt_{\sigma(l)}} & \text{ otherwise }
	\end{cases}$$
	$$\delta_{\sigma,ijkl}=\begin{cases}
	\tilde R_{s_it_{\sigma(j)}} & \text{ if } l=i\\
	\tilde R_{t_{\sigma(j)}t_{\sigma(i)}} \tilde R_{s_lt_{\sigma(l)}} & \text{ if } l=k\\
	\tilde R_{s_lt_{\sigma(l)}} & \text{ otherwise }
	\end{cases}.$$
	The careful reader can check that for every $(\sigma, \sigma^\prime=\sigma\circ (ij), i, j)$ with $\sigma(i)>\sigma(j)$,
	$$d_0\left(\sum_{k=i+1}^{j-1}(\gamma_{\sigma,ijk}+\gamma_{\sigma^\prime,ijk})+\gamma_{\sigma,ijj} \right)=d_{is_j}\alpha_\sigma+ d_{is_j}\alpha_{\sigma^\prime}$$
	and
	$$d_0\left(\sum_{k=i+1}^j\delta_{\sigma,ijk} \right)=d_{it_{\sigma(j)}}\alpha_\sigma + d_{jt_{\sigma(j)}}\alpha_{\sigma^\prime}.$$
	Therefore $d_0(\beta)$ agrees with $d(\alpha)$. 
	Here if $\alpha$ is in weight $p$, $\beta$ and $d(\alpha)$ are all in weight $p-1$. 
	Noting that $\beta\in S(\sA_*)^\perp$, by Lemma \ref{lem:77e89a29}, all simple summands of $d(\alpha+\beta)$ live in weight $\le p-3$ since $d_0(\beta)$ is the same as $d(\alpha)$. 
	Therefore, by Lemma \ref{lem:b0fa534e}, in order to compute $d_3(h_i(S^\prime))$ we only have to compute all simple summands of $d(\beta)$ in weight $p-3=w(\beta)-2$. 
	By the proof of Lemma \ref{lem:b0fa534e} such summands can only occur in the $d$-boundary of 
	$$\sum_\sigma\sum_{\substack{i<j\\ \sigma(i)>\sigma(j)}}\gamma_{\sigma,ijj}$$
	because to get a simple summand of $d(\beta)$ in weight $\le w(\beta)-2$, we can only replace the tensor factor 
	$$\gamma_{\sigma,ijjj}=\tilde R_{s_jt_{\sigma(i)}}\tilde R_{s_jt_{\sigma(j)}}$$
	of $\gamma_{\sigma,ijj}$ with
	$$\tilde R_{s_jt_{\sigma(j)}}^2\otimesop \tilde R_{t_{\sigma(j)}t_{\sigma(i)}} = \tilde R_{s_j+1,t_{\sigma(j)}+1}\otimesop \tilde R_{t_{\sigma(j)}t_{\sigma(i)}}$$
	in $d(\gamma_{\sigma,ijjj})$ which has weight $\le w(\gamma_{\sigma,ijjj})-2$. 
	In this typical example,
	$$w(\tilde R_{s_j+1,t_{\sigma(j)}+1}\otimesop \tilde R_{t_{\sigma(j)}t_{\sigma(i)}}) = w(\gamma_{\sigma,ijjj})-1-(t_{\sigma(j)}-s_j).$$
	To reach the equality 
	$$w(\gamma_{\sigma,ijjj})-1-(t_{\sigma(j)}-s_j)=w(\gamma_{\sigma,ijjj})-2$$
	 we can further restrict our attention to the terms where $t_{\sigma(j)}-s_j=1$. 
	 Hence the simple part in $d(\beta)$ of weight $p-3$ is
	$$\gamma=\sum_{\sigma}\sum_{\substack{i<j\\\sigma(i)>\sigma(j)\\t_{\sigma(j)}-s_j=1}} \gamma_{\sigma,ijj}^\prime$$
	where
	$$\gamma_{\sigma,ijj}^\prime=\bigotimesop_{l=1}^n \gamma_{\sigma,ijjl}^\prime$$
	and
	$$\gamma_{\sigma,ijjl}^\prime=\begin{cases}
	\tilde R_{s_is_j} & \text{ if } l=i\\
	\tilde R_{s_j+1,t_{\sigma(j)}+1}\otimesop \tilde R_{t_{\sigma(j)}t_{\sigma(i)}}=\tilde R_{s_j+1,s_j+2}\otimesop \tilde R_{t_{\sigma(j)},t_{\sigma(i)}} & \text{ if } l=j\\
	\tilde R_{s_lt_{\sigma(l)}} & \text{ otherwise }
	\end{cases}$$
	If we pass $\gamma$ to the $E_3$ page, we get
	$$\gamma=\sum_{j=n \text{ or } s_j<s_{j+1}-1} R_{s_j+1,s_j+2}R_{S-\{s_j\}+\{s_j+1\}, T-\{s_j+1\}+\{s_j\}}$$
	which is exactly
	$$\sum_{s\in S,~ s+1\in T} h_{s+1, s+2}h_{S-\{s\}+\{s+1\}, T-\{s+1\}+\{s\}}.$$
	By Lemma \ref{lem:b0fa534e} this is $d_3(h_{S, T})$.
\end{proof}

\begin{remark}
	If we use the notation $h_i(S^\prime)$ instead of $h_{S, T}$, the differential can be written in the following form
	$$d_3h_i(s_1,\dots,s_{n-1})=\sum_{\substack{j=n-1 \text{ or }\\ s_j + 1 < s_{j+1}}}h_{i+s_j+1}h_i(s_1,\dots,s_{j-1}, s_j+1, s_{j+1},\dots, s_{n-1}).$$
	Keep in mind that this is $d_2$ in May's grading.
\end{remark}

\section{\Groebner{} Bases and Computations}\label{sec:7c98bbbc}
In order to do computations in $HX$, we need the help of \Groebner{} bases, to which we will give a brief introduction. \Groebner{} bases are usually used in computer algebra and computational algebraic geometry, where the algebras are usually ungraded. 
But in algebraic topology most algebras are graded. 
Therefore we will introduce \Groebner{} bases in this context. 
We only consider algebras over $\bF_2$.

We also prove a general result on polynomial differential graded algebras. 
We will use this result to compute the algebra $HX_7$ via an inductive method. 
The computational results show that Conjectures \ref{conj:5f7d758} and \ref{conj:2a224128} are both true in $HX_7\subset HX$. In particular, the relations (3C), (5) and (6) of Conjecture \ref{conj:5f7d758} hold in $HX_7$ and no powers of nonzero elements are zero in this range.

\subsection{\Groebner{} basis}
In this section we always assume that $P=\bF_2[x_1,\dots,x_n]$ is a connected graded polynomial algebra over $\bF_2$.

\begin{definition}
	All operations related to \Groebner{} bases require the choice of a total order on the monomials in each degree, with the following property of compatibility with multiplication. 
	For all monomials $M,N,P$ where $M,N$ are in the same degree,
	$$M\le N\Longleftrightarrow MP\le NP.$$
	A total order (in each degree) satisfying this condition is called an \emph{admissible ordering}.	
\end{definition}

\begin{example}
	Lexicographical ordering is an obvious example of admissible ordering. 
	In this article we are primarily interested in the reversed lexicographical ordering, where if $M=x_1^{e_1}\cdots x_n^{e_n}$ and $N=x_1^{e_1^\prime}\cdots x_n^{e_n^\prime}$ are in the same degree, then $M<N$ if and only if
	$$e_1=e_1^\prime, \dots, e_{k-1}=e_{k-1}^\prime, e_k>e_k^\prime$$
	for some $k$.
\end{example}

\newcommand{\LM}{\mathrm{LM}}
\begin{definition}
	Once a total ordering is fixed, we let $\LM(f)$ denote the largest monomial in $f\in P$. 
	It is called the leading monomial of $f$.
\end{definition}

\begin{remark}
	If we use the reversed lexicographical ordering, then the leading monomial of $f\in P$ is the least monomial of $f$ in the lexicographical ordering.
\end{remark}

From now on we assume $P$ is alway equipped with an admissible ordering.

\newcommand{\red}{\mathrm{red}}
\begin{definition}
	Given two polynomials $f$ and $g$ in $P$, one says that $f$ is \emph{reducible} by $g$ if some monomial $M$ in $f$ is divisible by $\LM(g)$. 
	In this case we define the \emph{one-step} reduction of $f$ by $g$ by
	$$\red_1(f, g)=f+\frac{M}{\LM(g)}g.$$ 
	Note that compared with $f$, $\red_1(f, g)$ replaces $M$ in $f$ with other monomials less than $M$.
\end{definition}

\begin{definition}
	For $f\in P$ and a finite subset $S\subset P$, we say that $f$ is \emph{reducible} by $S$ if $f$ is reducible by some $g\in S$. 
	In order to define $\red(f, S)$, if $f$ is reducible by some $g\in S$, we replace $f$ by $\red_1(f, g)$, and we iterate this until $f$ is not reducible by any $g\in S$. 
	The iteration always terminates because there are only finitely many monomials in each degree since $P$ is a connected algebra. 
	The final result depends on the ordering of choices of $g$, and we define $\red(f, S)$ to be the set of all possible outcomes.
\end{definition}

\begin{definition}
	A \Groebner{} basis $G$ of an ideal $I$ in $P$ is a generating set of $I$ such that the set of images of all monomials \emph{not} divisible by $\LM(g)$ for any $g\in G$ under the canonical map $P\to P/I$ form an additive basis for $P/I$.
\end{definition}

\begin{remark}
	If $G$ is a \Groebner{} basis, then $\red(f, G)$ is exactly the standard representation of $f$ in $P/I$ as a linear combination of the additive basis mentioned above. 
	Hence $\red(f, G)$ consists of a single element of $P$.
\end{remark}

\begin{Algorithm}[Buchberger]
	Given a finite generating set $G$ of an ideal $I$ in $P$, we can change $G$ into a \Groebner{} basis of $I$ by doing the following
	\begin{enumerate}
		\item For $f, g\in G$, let $$L=\lcm(\LM(f), \LM(g)).$$
		Find monomials $m, n$ such that $\LM(mf)=\LM(ng)=L$. 
		If $\red(mf+ng, G)$ contains a nonzero polynomial, then add it to $G$.
		\item Repeat (1) until $\red(mf+ng, G)$ is zero for every pair $f, g$ in $G$.
	\end{enumerate}
\end{Algorithm}

\begin{remark}
	In Step (1), each time we add a new element to $G$ the ideal generated by all leading monomials of $G$ will strictly increase. 
	Therefore the algorithm always terminates in finitely many steps, because $P$ is a Noetherian ring.
\end{remark}

\begin{definition}
	Let $R=P/I$ for an ideal $I$ of $P$. 
	For $(a_1,a_2,\dots, a_n)\in R^n$ we define
	$$\Ann(a_1,\cdots, a_n)=\{(b_1,\dots, b_n)\in R^n~|~a_1b_1+\cdots+a_nb_n=0\}.$$
	This is an $R$-submodule of $R^n$. 
	Note that for $1\le i<j\le n$,
	$$(0,\dots,0,\overset{i}{a_j},0,\dots,0,\overset{j}{a_i},0,\dots,0)\in \Ann(a_1,\cdots, a_n).$$
	These are called the \emph{commutators} of $a_1,a_2,\dots, a_n$.
\end{definition}

\begin{lemma}\label{lem:738cb7cb}
	Assume $I$ is trivial and $R=P$. 
	Then $\Ann(x_1,\dots,x_n)$ is generated by commutators of $x_1,\dots,x_n$.
\end{lemma}
\begin{proof}
	This is a consequence of the fact that $\Tor_P(\bF_2, \bF_2)\iso E[\sigma x_1,\dots,\sigma x_n]$, so that $\sigma x_i\sma \sigma x_j$ is an additive basis of $\Tor^2_P(\bF_2, \bF_2)$. 
	In the Koszul complex this means that all $P$-linear relations among $x_k$ are generated by $x_ix_j+x_jx_i=0$
\end{proof}

\begin{definition}
	For $f\in P$, $\bar f$ denotes the image of $f$ in $P/I$.
\end{definition}

\begin{theorem}
	Assume $P$ is equipped with the reversed lexicographical ordering and $G$ is the \Groebner{} basis of an ideal $I$ in $P$. 
	For the images $\bar x_1,\dots,\bar x_k$ of the first $k$ generators $x_1,\dots,x_k$ of $P$ in $R=P/I$, $\Ann(\bar x_1,\dots, \bar x_k)$ is generated as a $R$-submodule of $R^k$ by commutators of $\bar x_1,\dots, \bar x_k$ and all $(\bar f_1,\dots,\bar f_k)\in R^k$ such that $f_i\in P$ and $x_1f_1+\cdots+ x_kf_k\in G$.
\end{theorem}
\begin{proof}
	Assume that $x_1g_1+\cdots+ x_kg_k\in I$. 
	By the definition of a \Groebner{} basis, we can always choose representatives $g_i$ of $\bar g_i$ such that no $g_i$ is reducible by $G$. 
	In order to show that $(\bar g_1,\dots,\bar g_k)$ is an $R$-linear combination of commutators of $\bar x_1,\dots, \bar x_k$ and $(\bar f_1,\dots,\bar f_k)$ described in the theorem, by Lemma \ref{lem:738cb7cb} it suffices to show that $x_1g_1+\cdots+x_kg_k$ is a $P$-linear combination of elements of $G$ of the form $x_1f_1+\cdots+ x_kf_k$, i.e. 
	elements of $G$ in which all monomials contain at least one of $x_1,\dots,x_k$.

	In fact, since $\red(x_1g_1+\cdots+x_kg_k, G)=0$, for some $1\le i\le k$, $x_ig_i$ is reducible by some $g\in G$. 
	Since $g_i$ is not reducible by $G$ but $x_ig_i$ is reducible, $\LM(g)$ must contain $x_i$. 
	Since $\LM(g)$ is the least monomial in $g$ ordered lexicographically, other monomials of $g$ must contain at at least one of $x_1,\dots,x_i$. 
	Therefore if we replace $x_ig_i$ with $\red_1(x_ig_i, g)$, then $x_1g_1+\cdots+x_kg_k$ becomes another polynomial of the form $x_1g_1^\prime+\cdots+x_kg_k^\prime$. 
	We can iterate this until $x_1g_1+\cdots+x_kg_k$ becomes zero. 
	Hence $x_1g_1+\cdots+x_kg_k$ is a $P$-linear combination of $g\in G$ in which all monomials contain at least one of $x_1,\dots,x_k$.
\end{proof}

By the theorem for $a_1,\dots,a_k\in R$ we can make an algorithm for finding a generating set of $\Ann(a_1,\cdots, a_n)\in R=P/I$.

\begin{Algorithm}\label{algorithm:6e571137}
	Given an ideal $I$ in $P$, $R=P/I$ and $f_1,\dots,f_k\in P$, a generating set of $\Ann(\bar f_1,\dots,\bar f_k)$ as an $R$-submodule of $R^k$ can be obtained by doing the following
	\begin{enumerate}
		\item Equip $Q=\bF_2[y_1,\dots,y_k,x_1,\dots,x_n]$ with the reversed lexicographical ordering.
		\item Compute the \Groebner{} basis $G$ of $I+(y_1-f_1,\dots,y_k-f_k)$.
		\item Find all elements $g$ of $G$ such that $\LM(g)$ contains at least one of $y_1,\dots,y_k$ and write $g$ in the form $g=y_1h_1+\cdots+y_kh_k$ where $h_i\in Q$. 
		We can do this because we are using the reversed lexicographical ordering.
		\item Replace $h_i$ with a polynomial in $x_1,\dots,x_n$ using the relations $y_1=f_1$, $\dots$, $y_k=f_k$.
		\item All images of $(h_1,\dots,h_k)$ in $R=P/I$ together with commutators of $\bar f_1$, $\dots$, $\bar f_k$ form a generating set of $\Ann(\bar f_1,\dots,\bar f_k)$ as an $R$-submodule of $R^k$.
	\end{enumerate}
\end{Algorithm}

\begin{theorem}\label{thm:c329a57a}
	If the \Groebner{} basis $G$ of $I\subset P$ with respect to some monomial ordering has the property that all the leading monomials of $g\in G$ are square free, then $R=P/I$ is nilpotent free.
\end{theorem}

\begin{proof}
	By the properties of \Groebner{} bases, the set of all monomials not reducible by $G$ forms a basis for $P/I$. 
	If all the leading monomials are square free, we show that this basis is closed under the squaring map.

	In fact, given a square free monomial $\alpha=x_{i_1}\cdots x_{i_k}$ ($i_1<\cdots<i_k$) in $P$, another monomial $\beta=x_1^{e_1}x_2^{e_2}\cdots$ is not divisible by $\alpha$ if and only if $\beta^2$ is not divisible by $\alpha$. 
	This is because
	$$\alpha|\beta\iff e_{i_j}>0 ~(1\le j\le k) \iff 2e_{i_j}>0 ~(1\le j\le k) \iff \alpha|\beta^2.$$

	Therefore $R$ is nilpotent free since we have a basis closed under the squaring map.
\end{proof}

\subsection{Polynomial differential graded algebras}
Note that the differential graded algebra $X$ is also a polynomial algebra. 
The following proposition will help us calculate the homology of these kinds of algebras.

\begin{proposition}\label{prop:07de5358}
	Assume that $A$ is a commutative differential graded algebra over $\bF_2$ and $c\in A$ is a cycle. 
	Consider $B=A[x]$ as a differential graded algebra which extends $A$ with $dx=c$.
	
	If $[c]=0$ in $HA$, then $HB\iso HA\otimes \bF_2[\tilde x]$ where $\tilde x$ corresponds to $x+a$ where $da=c$ in $A$. 
	
	If $[c]\neq 0$ in $HA$, assume that the ideal 
	$$\Ann_{HA}([c])=\{y\in HA:y[c]=0\}$$
	of $HA$ is generated by $y_1,\dots, y_n$ ($n=0$ if the ideal is zero).
	If we filter $B$ by
	$$F_pB=\{ax^i: a\in A,~i\le p\},$$
	then the associated graded algebra $E^0HB$ can be represented by
	$$HA\otimes \bF_2[b, g_1,\dots, g_n]/\sim$$
	where the relations are given by $[c]=0$ and
	\begin{enumerate}
		\myitem{(i)}\label{item:df8723cf} if $a_1y_1+\cdots+a_ny_n=0$ in $HA$ for $a_i\in HA$ then 
		$$a_1g_1+\cdots+a_ng_n=0.$$
		\myitem{(ii)}\label{item:62c9c1cc} $g_ig_j=by_iy_j$.
	\end{enumerate}
\end{proposition}
\begin{proof}
	Note that $x$ is in filtration $1$ and $dx=c$ is in filtration $0$.
	Hence 
	$$E_1\iso HA\otimes \bF_2[x]$$
	with $dx=[c]$. 
	
	If $[c]=0$, then $E_1=E_\infty$ because $x$ is a permanent cycle represented by $x+a$ for some $a\in A$ such that $da=c$. 
	There are no extensions since there are no relations involving $x$. 
	Hence $HB\iso HA\otimes \bF_2[\tilde x]$.
	
	If $[c]\neq 0$, noting that $b=[x^2]$ is a permanent cycle, the set of elements in $E_2=HE_1$ in even filtrations is isomorphic to
	$$\bigoplus x^{2i}HA/([c])$$
	while the set of elements in odd filtrations is isomorphic to
	$$\bigoplus x^{2i-1}\Ann_{HA}([c]).$$
	The multiplication by $b=[x^2]$ will map elements in filtration $p$ isomorphically onto elements in filtration $p+2$. 
	Both filtrations give modules over $HA$ and the module structure of $x\Ann_{HA}([c])$ (elements in filtration $1$) is precisely given by \ref{item:df8723cf} with $g_i=[xy_i]$. 
	Relations in \ref{item:62c9c1cc} are direct consequences of $xy_i\cdot xy_j=x^2\cdot y_i\cdot y_j$ in $E_1$.
	The spectral sequence collapses in $E_2$ because the $g_i=[xy_i]$ are represented by cycles $xy_i+a_i\in B$ where $da_i=cy_i$ in $A$. 
	Therefore the $g_i$ are all permanent cycles.
\end{proof}

\begin{remark}
	The proposition does not solve the extension problem for computing $HB$. 
	However, it constrains the number of relations we have to deal with, which is very important for our computation of $HX_7$ in the next section.
\end{remark}

\begin{remark}
	A generating set of $(r_1,\dots,r_n)\in (HA)^n$ in (1) in the proposition can be obtained by Algorithm \ref{algorithm:6e571137}.
\end{remark}

\subsection{The computation of \mybm{$HX_7$}{HX7}}\label{sec:11e59387}
In this section, we are going to compute $HX_7$ by an inductive method using Proposition \ref{prop:07de5358}. 
We will see that Conjectures \ref{conj:5f7d758} and \ref{conj:2a224128} hold in $HX_7\subset HX$.

It is helpful to see that $X$ has a lot of symmetries. 
These will be useful in our induction.

\begin{definition}
	For $0\le m<n$, let $X[m, n]$ denote the sub-DGA of $X$
	$$X[m, n]=\bF_2[R_{ij}: m\le i<j\le n].$$
	Note that $X_n=X[0,n]$. 
	Let $X_{n, k}=\bF_2[R_{0i}: i\le k]\otimes X[1,n]$. This is also a sub-DGA of $X$.
\end{definition}

\begin{proposition}
	The map
	$$r: X\to X[m, n]$$
	given by
	$$r(R_{ij})=\begin{cases}
		R_{ij}, & \text{ if } m\le i<j\le n\\
		0, & \text{otherwise}
	\end{cases}$$
	is a retraction of DGAs. 
	Therefore the homomorphism in homology $HX[m, n]\to HX$ is injective.
\end{proposition}

In addition to Proposition \ref{prop:8ac1afe}, we have another property of symmetries in $X$.
\begin{proposition}\label{prop:c143b281}
	The translation map
	$$f_k: X[m, n]\to X[m+k, n+k]$$
	$$R_{ij}\mapsto R_{i+k, j+k}$$
	is an isomorphism between differential algebras.
	Therefore $$HX[m, n]\iso HX[m+k,n+k]$$ as algebras.
\end{proposition}

\begin{remark}\label{rmk:41c96856}
	The map $f_k$ is actually the same as the squaring operation $(Sq^0)^k$. 
	Here $Sq^0$ is a power operation in the May spectral sequence (See \cite{Nakamura72}).
\end{remark}

Our strategy to compute $HX_7$ is to show that Conjecture \ref{conj:5f7d758} holds in $HX_n$ for $n=1,2,\dots,7$ inductively. 
For $m<n$, if we can prove that Conjecture \ref{conj:5f7d758} on $HX[1,n]$ implies Conjecture \ref{conj:5f7d758} on $HX[0,n]=HX_n$, then by ignoring all $R_{ij}$ with $j>m$ in the proof, we can obtain a proof of the fact that Conjecture \ref{conj:5f7d758} on $HX[1,m]$ implies Conjecture \ref{conj:5f7d758} on $HX[0,m]=HX_m$. 
Moreover, by Proposition \ref{prop:c143b281}, we have
$$HX[1,n]\iso HX[0,n-1]=HX_{n-1}.$$

Therefore the statement
\begin{equation}\label{eq:dcca769e}
	\text{Conjecture \ref{conj:5f7d758} holds on $HX_6$ $\Longrightarrow$ Conjecture \ref{conj:5f7d758} holds on $HX_7$}
\end{equation}
implies the statement
\begin{center}
	Conjecture \ref{conj:5f7d758} holds on $HX_{k-1}$ $\Longrightarrow$ Conjecture \ref{conj:5f7d758} holds on $HX_k$
\end{center}
for $2\le k\le 5$. 
Since Conjecture \ref{conj:5f7d758} holds in $HX_1=\bF_2[h_0]$, it suffices to prove the statement (\ref{eq:dcca769e}).

Now we have our assumption on $HX_6\iso HX[1, 7]$. 
See Appendix \ref{app:3e79348e} for a list of generators and relations we generate for $HX[1,7]$ according to Conjecture \ref{conj:5f7d758}.

We are going to compute $HX[1,7]=HX_{7,0}$, $HX_{7,1}$, $\dots$, $HX_{7,7}=HX_7$ one by one. 
Note that $X_{7,i}=X_{7,i-1}\otimes \bF_2[R_{0i}]$. 
We apply Proposition \ref{prop:07de5358} to the case where $A=X_{7,i-1}$, $B=X_{7,i}$, $x=R_{0i}$ and $c=\sum_{j=1}^{i-1} R_{0j}R_{ji}$ to obtain the homology $HX_{7,i}$ from $HX_{7,i-1}$.

Recall that Proposition \ref{prop:07de5358} does not solve the extension problems for us.
I managed to solve all of the extensions via many different approaches, including pure guesses, and to check them with the aid of a computer by realizing the relations as boundaries of chains. 

Appendix \ref{app:cf9b8884}-\ref{app:97651fc5} lists the generators and relations of $HX_{7,1},\dots,HX_{7,7}$ computed by the author. 
In these charts, the relations are grouped into two parts. 
Part \ref{item:df8723cf} corresponds to relations \ref{item:df8723cf} in Proposition \ref{prop:07de5358} and Part \ref{item:62c9c1cc} corresponds to relations \ref{item:62c9c1cc} in Proposition \ref{prop:07de5358}. 
For Part \ref{item:df8723cf}, the author put the extension part of the relations on the right-hand side of the equations.

Appendix \ref{app:7b101ee1} reorganizes the relations of $HX_7=HX_{7,7}$ in the form of \Groebner{} bases. 
We can see that all of the leading monomials are square free. 
Hence Conjecture \ref{conj:2a224128} holds in $HX_7$ by Theorem \ref{thm:c329a57a}.

Appendix \ref{app:35b02630} lists the relations of $HX_7$ according to Conjecture \ref{conj:5f7d758}. 
It has been checked by the computer that these relations indeed generate the same \Groebner{} basis as that in Appendix \ref{app:7b101ee1}. 
Hence we see that Conjecture \ref{conj:5f7d758} indeed holds in $HX_7$.

Combining the results above proves Theorem \ref{thm:a26989b6}.

\subsection{A localization of the May spectral sequence}\label{sec:a82fc999}
One of the useful tools to compute the May spectral sequence is the Adams vanishing theorem.
\begin{theorem}[Adams \cite{Adams58}]\label{thm:35b90bcf}
	$\Ext_{\sA}^{s,t}(\bF_2,\bF_2)=0$ if $t-s<q(s)$ where the function $q$ is given by
	$$\begin{array}{lll}
		q(4k) &\hspace{-0.2cm}=\hspace{-0.2cm}& 8k-1;\\
		q(4k+1) &\hspace{-0.2cm}=\hspace{-0.2cm}& 8k+1;\\
		q(4k+2) &\hspace{-0.2cm}=\hspace{-0.2cm}& 8k+2;\\
		q(4k+3) &\hspace{-0.2cm}=\hspace{-0.2cm}& 8k+3.
	\end{array}$$
\end{theorem}

Note that May \cite{May64} and Tangora \cite{Tangora} both used this theorem to compute some differentials in the May spectral sequence. 
This is based on the fact that all the infinite $h_0$-structure lines in the May spectral sequence have to be truncated by some differentials in order for the vanishing line to appear in the $E_\infty$ page. 
One of the examples is the first nontrivial $d_6$ differential 
$$d_6(x)=h_0^5y$$
where
$$x=h_0b_{02}^3b_{03}h_0(1),~y=h_4b_{02}^2h_0(1)+h_0^3b_{02}b_{13}$$
in $E_6$. 
Here $h_0^iy\neq 0$ for all $i\ge 0$ and $x$ is the only thing that can truncate this infinite $h_0$-structure line supported by $y$. 
By computing the filtration degrees this differential is $d_6$.

These infinite $h_0$-structure lines inherit structures from the May spectral sequence and form another spectral sequence which converges to zero in positive stems because of Theorem \ref{thm:35b90bcf}. 
A better way to process this information is to invert $h_0$ in the May spectral sequence and study the localized spectral sequence which converges to $\bF_2[h_0^{\pm 1}]$. 
The following theorem shows the structure of the $E_2$ page of the localized May spectral sequence. 
What is surprising is that it contains a subalgebra $HX[2,\infty]$ which is isomorphic to the original $E_2\iso HX$ with a shift in degree $t$.
\begin{theorem}
	$$h_0^{-1}HX\iso\bF_2[h_0^{\pm 1}, b_{0j}: j\ge 2]\otimes HX[2, \infty]$$
\end{theorem}

\begin{proof}
	Note that as a differential algebra,
	\begin{equation*}
		h_0^{-1}HX\iso \bF_2[h_0^{\pm 1}]\otimes H(X/(R_{01}-1)),
	\end{equation*}
	since $h_0$ is represented by $R_{01}$.
	It suffices to show that
	\begin{equation}\label{eq:7abf37f4}
		H(X/(R_{01}-1))\iso \bF_2[b_{0j}:j\ge 2]\otimes HX[2,\infty].
	\end{equation}
	
	Let $$Y_m=X[2,\infty]\otimes \bF_2[R_{0j}, R_{1j}:j\le m]/(R_{01}-1).$$
	Observe that $$X\iso\colim\limits_{m} Y_m \hspace{10pt}\text{and}\hspace{10pt}  Y_m\iso Y_{m-1}\otimes \bF_2[R_{0m}, R_{1m}].$$
	Now it suffices to show by induction that for all $m$
	$$HY_m\iso \bF_2[b_{0j}:2\le j\le m]\otimes HX[2,\infty].$$
	
	The claim is trivial when $m=0,1$.

	Assume it is true for $Y_{m-1}$. 
	First we consider $Y_{m-1}\otimes \bF_2[R_{1m}]$. 
	Note that $dR_{1m}$ is a boundary in $HY_{m-1}$ since
	$$d(e_{0m})=d(R_{01}R_{1m}+R_{02}R_{2m}+\cdots+R_{0,m-1}R_{m-1,m})=0$$
	which implies
	$$d(R_{1m})=d(R_{02}R_{2m}+\cdots+R_{0,m-1}R_{m-1,m}).$$
	By Proposition \ref{prop:07de5358} we have
	$$H(Y_{m-1}\otimes \bF_2[R_{1m}])\iso HY_{m-1}\otimes \bF_2[e_{0m}].$$

	Now we consider $Y_m=Y_{m-1}\otimes \bF_2[R_{1m}]\otimes \bF_2[R_{0m}]$. 
	Note that $d R_{0m}=e_{0m}$ and $\Ann(e_{0m})$ is trivial in $HX[2,\infty]\otimes \bF_2[e_{0m}]$. 
	Therefore by Proposition \ref{prop:07de5358},
	$$HY_m\iso HY_{m-1}\otimes \bF_2[b_{0m}]\iso HX[2,\infty]\otimes \bF_2[b_{0j},2\le j\le m].$$
	Hence the induction is complete.
\end{proof}

By the Adams vanishing theorem on the $E_2$ page of the Adams spectral sequence we know that
$$h_0^{-1}\Ext_{\sA}^{*,*}(\bF_2,\bF_2)=\bF_2[h_0^{\pm 1}].$$
Hence after inverting $h_0$ in the May spectral sequence we get a spectral sequence with
$$E_2=h_0^{-1}HX\Longrightarrow \bF_2[h_0^{\pm 1}].$$
By the theorem above this is the same as
$$\bF_2[h_0^{\pm 1}, b_{0j}: j\ge 2]\otimes HX[2, \infty]\Longrightarrow \bF_2[h_0^{\pm 1}]$$.

Note that $HX[2,\infty]$ is isomorphic to $HX$ with a shift of degrees. 
Therefore the following composition is an embedding
$$\varphi: HX\fto{~~(Sq^0)^2~~}HX\fto{\hspace{1cm}}h_0^{-1}HX$$
where the second map is the localization. 
Since the operation $Sq^0$ (see Remark \ref{rmk:41c96856}) commutes with all $d_r$ differentials in the May spectral sequence we have a comparison map
\begin{equation}\label{eq:19bb24a5}
	\xymatrix{
		HX\ar@{=>}[r]\ar[d]_{\varphi} & \Ext_{\sA}^{*,*}(\bF_2,\bF_2)\ar[d]\\
		h_0^{-1}HX\ar@{=>}[r] & \bF_2[h_0^{\pm 1}]
	}
\end{equation}
The bottom spectral sequence has an advantage in calculation since all elements in positive stems have to be killed by differentials. 
We intend to use the bottom spectral sequence to aid in computing the top.
Interestingly, computations in low degrees lead us to the following conjecture.

\begin{conjecture}\label{conj:f911bd1c}
	The localized spectral sequence
	$$E_2=h_0^{-1}HX\Longrightarrow \bF_2[h_0^{\pm 1}]$$
	is isomorphic to a sub-spectral sequence 
	\begin{equation}\label{eq:6c4d5358}
		E_2=\bF_2[\frac{b_{0j}}{h_0^2}: j\ge 2]\otimes HX[2,\infty]\Longrightarrow \bF_2
	\end{equation}
	 tensored with $\bF_2[h_0^{\pm 1}]$.
\end{conjecture}

This conjecture will rule out many possible choices of differentials in the spectral sequence starting with $E_2=h_0^{-1}HX$. Although the author cannot yet prove this, there is another spectral sequence with the same $E_2$ and $E_\infty$ as those in (\ref{eq:6c4d5358}).

\begin{theorem}
	Consider the cobar resolution $\tilde C(\sA_*)$ of $\bF_2$ over $\sA_*$ where $\tilde C_s(\sA_*)$ consists of elements $[a_1|\cdots|a_s]a$ and
	$$d[a_1|\cdots|a_s]a=\sum_i\sum [a_1|\cdots|a_i^\prime|a_i^\pprime|\cdots|a_s]a + \sum [a_1|\cdots|a_i^\prime|a_i^\pprime|\cdots|a_s|\epsilon(a^\prime)]a^\pprime.$$

	There is a filtration on $\tilde C(\sA_*)$ such the resulting spectral sequence has an $E_2$ page isomorphic to the $E_2$ in (\ref{eq:6c4d5358})
	with a degree shift in $t$, and it also converges to $\bF_2$.
\end{theorem}
\begin{proof}
	We continue the use of Ravenel's filtration in Section \ref{sec:b0f0886a}. 
	Consider the weight function $w$ on $\sA_*$ and $C(\sA_*)$ in Definition \ref{def:826d0f73}. 
	We define another linear function $w^\prime$ on $\sA_*$ given by
	$$w^\prime(\xi_1^{r_1}\cdots \xi_k^{r_k})=\sum_{i=1}^k 2ir_i.$$
	We can define a weight function $w$ on $\tilde C_s(\sA_*)=C_s(\sA_*)\otimes \sA_*$ by
	$$w[a_1|\cdots|a_s]a=w(a_1)+\cdots +w(a_s)+w^\prime(a).$$

	Note that
	$$d([]\xi_n)=\sum_{k=0}^{n-1} [\xi_{n-k}^{2^k}]\xi_k$$
	and
	$$w([]\xi_n)=2n > w([\xi_{n-k}^{2^k}]\xi_k)=2k+2(n-k)-1=2n-1.$$
	Therefore on the $E_0$ page $d_0([]\xi_n)=0$. 
	Hence by Proposition \ref{prop:619ea206}, the $E_1$-term is isomorphic to $X\otimes \sA_*$. 
	The $d_1$ differentials are given by
	$$d_1(R_{ij})=\sum_k R_{ik}R_{kj},$$
	$$d_1(\xi_j)=R_{0j}+\sum_k \xi_kR_{kj}.$$
	
	By (\ref{eq:7abf37f4}) it suffices to show that
	$$X\otimes \sA_*\iso X/(R_{01}-1)$$
	as differential algebras. 
	In fact, it is not hard to check that the following map gives the isomorphism.
	$$X\otimes \sA_*\longrightarrow X/(R_{01}-1)$$
	$$R_{ij}\otimes 1\longmapsto R_{i+1,j+1}$$
	$$1\otimes\xi_j\longmapsto R_{0, j+1}.$$
\end{proof}

In contrast to the comparison map in (\ref{eq:19bb24a5}) we now build another comparison
$$\xymatrix{
	HX\ar@{=>}[r]\ar[d]_{\varphi} & \Ext_{\sA}^{*,*}(\bF_2,\bF_2)\ar[d]\\
	\bF_2[\frac{b_{0j}}{h_0^2}: j\ge 2]\otimes HX[2,\infty]\ar@{=>}[r] & \bF_2
}$$
using the composition of the map of complexes $C(\sA_*)\to \tilde C(\sA_*)$ and the operation $Sq^0$. 
The map $\varphi$ is again an embedding. 
A stronger version of Conjecture \ref{conj:f911bd1c} is that (\ref{eq:6c4d5358}) is isomorphic to the bottom spectral sequence above.

The localization map and other comparison maps with compositions of $(Sq^0)^i$ yield different indeterminacies for computing the May spectral sequence. 
The author has been collaborating with the computer and feeding these data into the program to obtain higher differentials in the May spectral sequence.

\appendix

\section{Charts}\label{app:667432e1}
There is a new symbol in the following charts. 
Note that for each indecomposable $h_i(S^\prime)=h_{S, T}\in \sH$,
$$\sum_{j=0}^{i-1} R_{0j}R_{S-\{i\}+\{j\}, T}$$ 
is a cycle in $HX_{7, i}$. 
We let $r_i(S^\prime)=r_{S,T}$ denote the homology class of this cycle in $HX_{7, i}$.
\subsection{\mybm{$HX[1,7]$}{HX[1,7]}}\label{app:3e79348e}
\subsubsection*{Generators}\hspace{1pt}

$h_i$, $1\le i\le 6$

$h_i(1)$, $1\le i\le 4$

$h_i(1, 3)$, $h_i(1, 2)$, $1\le i\le 2$

$b_{ij}$, $1\le i<i+2\le j\le 7$.

\subsubsection*{Relations}\hspace{1pt}

$h_1h_2 = 0$

$h_2h_3 = 0$

$h_3b_{13} = h_1h_1(1)$

$h_3h_4 = 0$

$h_3h_1(1) = h_1b_{24}$

$h_4h_1(1) = 0$

$b_{13}b_{24} = h_2^2b_{14} + h_1(1)^2$

$h_1h_2(1) = 0$

$h_4b_{24} = h_2h_2(1)$

$h_4h_5 = 0$

$h_2(1)b_{13} = h_2h_4b_{14}$

$h_4h_2(1) = h_2b_{35}$

$h_1(1)h_2(1) = 0$

$b_{13}b_{35} = h_1^2b_{25} + h_4^2b_{14}$

$h_1(1)b_{35} = h_1h_3b_{25}$

$h_5h_2(1) = 0$

$b_{24}b_{35} = h_3^2b_{25} + h_2(1)^2$

$h_2h_3(1) = 0$

$h_5b_{35} = h_3h_3(1)$

$h_3(1)b_{13} = h_1h_1(1, 3)$

$h_1(1)h_3(1) = h_3h_1(1, 3)$

$h_3h_1(1, 3) = h_1h_5b_{25}$

$h_3(1)b_{24} = h_3h_5b_{25}$

$h_5h_6 = 0$

$h_1(1)h_1(1, 3) = h_2^2h_5b_{15} + h_5b_{13}b_{25}$

$h_3(1)b_{14} = h_1h_1(1, 2) + h_3h_5b_{15}$

$h_5h_3(1) = h_3b_{46}$

$h_1(1, 3)b_{24} = h_2^2h_1(1, 2) + h_5h_1(1)b_{25}$

$h_2(1)h_3(1) = 0$

$h_1(1, 2)b_{13} = h_5h_1(1)b_{15} + h_1(1, 3)b_{14}$

$h_1(1)b_{46} = h_5h_1(1, 3)$

$h_1(1)h_1(1, 2) = h_5b_{24}b_{15} + h_5b_{14}b_{25}$

$h_2(1)h_1(1, 3) = h_2h_4h_1(1, 2)$

$b_{24}b_{46} = h_2^2b_{36} + h_5^2b_{25}$


$b_{46}b_{14} = h_1^2b_{26} + h_5^2b_{15} + b_{13}b_{36}$

$h_1(1, 3)b_{35} = h_1h_3(1)b_{25} + h_4^2h_1(1, 2)$

$h_1(1)b_{36} = h_1h_3b_{26} + h_5h_1(1, 2)$

$h_2(1)b_{46} = h_2h_4b_{36}$

$h_6h_3(1) = 0$


$b_{35}b_{46} = h_4^2b_{36} + h_3(1)^2$

$h_6h_1(1, 3) = 0$

$h_3h_4(1) = 0$

$h_3(1)h_1(1, 3) = h_1h_4^2b_{26} + h_1b_{46}b_{25}$

$h_6b_{46} = h_4h_4(1)$

$h_1(1)h_4(1) = 0$

$h_6h_1(1, 2) = 0$

$b_{13}b_{46}b_{25} = h_2^2h_4^2b_{16} + h_2^2b_{46}b_{15} + h_4^2b_{13}b_{26} + h_1(1, 3)^2$

$h_1h_2(1, 3) = 0$

$h_4(1)b_{24} = h_2h_2(1, 3)$

$h_3(1)h_1(1, 2) = h_1b_{35}b_{26} + h_1b_{25}b_{36}$

$h_2(1, 3)b_{13} = h_2h_4(1)b_{14}$

$b_{13}b_{25}b_{36} = h_1^2b_{25}b_{26} + h_2^2b_{35}b_{16} + h_2^2b_{36}b_{15} + h_4^2b_{14}b_{26} + h_1(1, 3)h_1(1, 2)$

$h_2(1)h_4(1) = h_4h_2(1, 3)$

$h_4h_2(1, 3) = h_2h_6b_{36}$

$h_1(1)h_2(1, 3) = 0$


$h_4(1)b_{35} = h_4h_6b_{36}$

$h_1(1, 2)b_{46} = h_1h_3(1)b_{26} + h_1(1, 3)b_{36}$

$b_{14}b_{25}b_{36} = h_3^2b_{15}b_{26} + h_2(1)^2b_{16} + h_1(1, 2)^2 + b_{24}b_{36}b_{15} + b_{35}b_{14}b_{26}$

$h_1h_2(1, 2) = 0$

$h_4(1)b_{25} = h_2h_2(1, 2) + h_4h_6b_{26}$

$h_2(1)h_2(1, 3) = h_3^2h_6b_{26} + h_6b_{24}b_{36}$


$h_6h_4(1) = h_4b_{57}$

$h_2(1, 2)b_{13} = h_2h_4h_6b_{16} + h_2h_4(1)b_{15}$

$h_2(1, 3)b_{35} = h_3^2h_2(1, 2) + h_6h_2(1)b_{36}$

$h_1(1)h_2(1, 2) = 0$

$h_3(1)h_4(1) = 0$

$h_2(1, 2)b_{24} = h_6h_2(1)b_{26} + h_2(1, 3)b_{25}$

$h_2(1, 2)b_{14} = h_6h_2(1)b_{16} + h_2(1, 3)b_{15}$

$h_4(1)h_1(1, 3) = 0$

$h_2(1)b_{57} = h_6h_2(1, 3)$

$h_2(1)h_2(1, 2) = h_6b_{35}b_{26} + h_6b_{25}b_{36}$


$h_3(1)h_2(1, 3) = h_3h_5h_2(1, 2)$

$b_{35}b_{57} = h_3^2b_{47} + h_6^2b_{36}$

$h_4(1)h_1(1, 2) = 0$

$h_1(1, 3)h_2(1, 3) = 0$



$b_{57}b_{25} = h_2^2b_{37} + h_6^2b_{26} + b_{24}b_{47}$

$b_{14}b_{47} = h_1^2b_{27} + h_6^2b_{16} + b_{13}b_{37} + b_{57}b_{15}$

$h_2(1)b_{47} = h_2h_4b_{37} + h_6h_2(1, 2)$

$h_2(1, 3)b_{46} = h_2h_4(1)b_{36} + h_5^2h_2(1, 2)$

$h_2(1, 3)h_1(1, 2) = 0$

$h_3(1)b_{57} = h_3h_5b_{47}$

$h_1(1, 3)b_{57} = h_5h_1(1)b_{47}$


$h_1(1, 3)h_2(1, 2) = 0$

$b_{46}b_{57} = h_5^2b_{47} + h_4(1)^2$

$h_1(1, 2)b_{57} = h_1h_3h_5b_{27} + h_5h_1(1)b_{37}$

$h_1(1, 2)h_2(1, 2) = 0$

$h_4(1)h_2(1, 3) = h_2h_5^2b_{37} + h_2b_{57}b_{36}$


$h_1(1, 2)b_{47} = h_1h_3(1)b_{27} + h_1(1, 3)b_{37}$

$b_{24}b_{57}b_{36} = h_3^2h_5^2b_{27} + h_3^2b_{57}b_{26} + h_5^2b_{24}b_{37} + h_2(1, 3)^2$

$h_4(1)h_2(1, 2) = h_2b_{46}b_{37} + h_2b_{36}b_{47}$


$b_{24}b_{36}b_{47} = h_2^2b_{36}b_{37} + h_3^2b_{46}b_{27} + h_3^2b_{47}b_{26} + h_5^2b_{25}b_{37} + h_2(1, 3)h_2(1, 2)$

$h_2(1, 2)b_{57} = h_2h_4(1)b_{37} + h_2(1, 3)b_{47}$

$b_{25}b_{36}b_{47} = h_4^2b_{26}b_{37} + h_3(1)^2b_{27} + h_2(1, 2)^2 + b_{35}b_{47}b_{26} + b_{46}b_{25}b_{37}$


\subsection{\mybm{$HX_{7, 1}$}{HX71}}\label{app:cf9b8884}
It is obvious that $HX_{7, 1}=HX[1, 7]\otimes \bF_2[h_0]$.

\subsection{\mybm{$HX_{7, 2}$}{HX72}}
Consider $d R_{02}=R_{01}R_{12}$ whose homology class is $r_1=h_0h_1$ in $HX_{7, 1}$. 
We have
$$\Ann_{HX_{7, 1}}(r_1)=(h_2, h_2(1), h_2(1, 3), h_2(1, 2))$$
obtained by Algorithm \ref{algorithm:6e571137}. 
Apply Proposition \ref{prop:07de5358} on $X_{7, 2}=X_{7, 1}\otimes\bF_2[R_{02}]$. [Explain Part (i) and Part (ii)]
The $E_2=E_\infty$ page is generated by $R_{02}h_2$, $R_{02}h_2(1)$, $R_{02}h_2(1, 3)$, $R_{02}h_2(1, 2)$ and $R_{02}^2$ which are represented by $r_2$, $r_2(1)$, $r_2(1, 3)$, $r_2(1, 2)$ and $b_{02}$ in $HX_{7, 2}$ respectively. 
In addition to relations in $HX_{7, 1}$, the new relations in $HX_{7, 2}$ are $r_1=0$ and
\subsubsection*{Part \ref{item:df8723cf}}
\footnote{Part \ref{item:df8723cf} corresponds to relations \ref{item:df8723cf} in Proposition \ref{prop:07de5358} and Part \ref{item:62c9c1cc} corresponds to relations \ref{item:62c9c1cc} in Proposition \ref{prop:07de5358}. The hidden extensions for Part \ref{item:df8723cf} are placed on the right-hand side of the equations. Note that there are no nontrivial extensions in this case.}

$r_2h_1=0$,

$r_2h_3=0$,

$r_2(1)h_1=0$,

$r_2(1)h_1(1)=0$,

$r_2(1)h_5=0$,

$r_2h_3(1)=0$,

$r_2(1)h_3(1)=0$,

$r_2(1, 3)h_1=0$,

$r_2(1, 3)h_1(1)=0$,

$r_2(1, 2)h_1=0$,

$r_2(1, 2)h_1(1)=0$,

$r_2(1, 3)h_1(1, 3)=0$,

$r_2(1, 3)h_1(1, 2)=0$,

$r_2(1, 2)h_1(1, 3)=0$,

$r_2(1, 2)h_1(1, 2)=0$,

$r_2(1)h_2+r_2h_2(1)=0$,

$r_2(1)b_{13}+r_2h_4b_{14}=0$,

$r_2(1)h_4+r_2b_{35}=0$,

$r_2(1)h_1(1, 3)+r_2h_4h_1(1, 2)=0$,

$r_2(1)b_{46}+r_2h_4b_{36}=0$,

$r_2(1, 3)h_2+r_2h_2(1, 3)=0$,

$r_2(1, 3)b_{13}+r_2h_4(1)b_{14}=0$,

$r_2(1, 3)h_4+r_2(1)h_4(1)=0$,

$r_2(1)h_4(1)+r_2h_6b_{36}=0$,

$r_2(1, 2)h_2+r_2h_2(1, 2)=0$,

$r_2(1, 3)h_2(1)+r_2(1)h_2(1, 3)=0$,

$r_2(1, 3)h_6+r_2(1)b_{57}=0$,

$r_2(1, 2)h_2(1)+r_2(1)h_2(1, 2)=0$,

$r_2(1, 2)h_3h_5+r_2(1, 3)h_3(1)=0$,

$r_2(1, 2)h_2(1, 3)+r_2(1, 3)h_2(1, 2)=0$,

$r_2(1, 2)b_{13}+r_2h_4h_6b_{16}+r_2h_4(1)b_{15}=0$,

$r_2(1, 2)h_3^2+r_2(1)h_6b_{36}+r_2(1, 3)b_{35}=0$,

$r_2(1, 2)b_{24}+r_2(1, 3)b_{25}+r_2(1)h_6b_{26}=0$,

$r_2(1, 2)b_{14}+r_2(1, 3)b_{15}+r_2(1)h_6b_{16}=0$,

$r_2(1, 2)h_5^2+r_2h_4(1)b_{36}+r_2(1, 3)b_{46}=0$,

$r_2(1, 2)h_6+r_2(1)b_{47}+r_2h_4b_{37}=0$,

$r_2(1, 3)h_4(1)+r_2h_5^2b_{37}+r_2b_{57}b_{36}=0$,

$r_2(1, 2)h_4(1)+r_2b_{36}b_{47}+r_2b_{46}b_{37}=0$,

$r_2(1, 2)b_{57}+r_2h_4(1)b_{37}+r_2(1, 3)b_{47}=0$,

\subsubsection*{Part \ref{item:62c9c1cc}}\hspace{1pt}

$r_{2, 3}r_{23, 45}=b_{01}b_{1, 4}h_{4, 5}+b_{02}b_{2, 4}h_{4, 5}$,

$r_{2, 3}r_{235, 467}=b_{01}b_{1, 4}h_{45, 67}+b_{02}b_{2, 4}h_{45, 67}$,

$r_{2, 3}r_{234, 567}=b_{01}b_{1, 5}h_{45, 67}+b_{02}b_{2, 5}h_{45, 67}+b_{01}b_{1, 6}h_{46, 57}+b_{02}b_{2, 6}h_{46, 57}$,

$r_{23, 45}r_{235, 467}=b_{01}b_{13, 46}h_{6, 7}+b_{02}b_{23, 46}h_{6, 7}$,

$r_{23, 45}r_{234, 567}=b_{01}b_{13, 56}h_{6, 7}+b_{02}b_{23, 56}h_{6, 7}$,

$r_{235, 467}r_{234, 567}=b_{01}b_{135, 567}+b_{02}b_{235, 567}$,

$r_{2, 3}r_{2, 3}=b_{01}b_{1, 3}+b_{02}b_{2, 3}$,

$r_{23, 45}r_{23, 45}=b_{01}b_{13, 45}+b_{02}b_{23, 45}$,

$r_{235, 467}r_{235, 467}=b_{01}b_{135, 467}+b_{02}b_{235, 467}$,

$r_{234, 567}r_{234, 567}=b_{01}b_{134, 567}+b_{02}b_{234, 567}$.

\vspace{6pt}

\subsection{\mybm{$HX_{7, 3}$}{HX73}}
Consider $d R_{03}=R_{01}R_{13}+R_{02}R_{23}$ whose homology class is $r_2$ in $HX_{7, 2}$. 
We have
$$\Ann_{HX_{7, 2}}(r_2)=(h_1, h_3, h_3(1)).$$
Apply Proposition \ref{prop:07de5358} on $X_{7, 3}=X_{7, 2}\otimes\bF_2[R_{03}]$. 
The $E_2=E_\infty$ page is generated by $R_{03}h_1$, $R_{03}h_3$, $R_{03}h_3(1)$ and $R_{03}^2$ which are represented by $h_0(1)$, $r_3$, $r_3(1)$ and $b_{03}$ in $HX_{7, 3}$ respectively. 
In addition to relations in $HX_{7, 1}$, the new relations in $HX_{7, 3}$ are $r_2=0$ and
\subsubsection*{Part \ref{item:df8723cf}}\hspace{1pt}

$h_0(1)h_0=b_{02}h_2$,

$h_0(1)h_2=h_0b_{13}$,

$r_3h_2=h_0h_1(1)$,

$r_3h_4=0$,

$h_0(1)h_2(1)=h_0b_{14}h_4$,

$h_0(1)r_2(1)=0$,

$r_3(1)h_2=h_0h_1(1,3)$,

$r_3(1)h_2(1)=h_0h_4h_1(1,2)$,

$r_3(1)r_2(1)=0$,

$r_3(1)h_6=0$,

$r_3h_4(1)=0$,

$h_0(1)h_2(1, 3)=h_0b_{14}h_4(1)$,

$h_0(1)r_2(1, 3)=0$,

$h_0(1)h_2(1, 2)=h_0b_{15}h_4(1) + h_0b_{16}h_4h_6$,

$h_0(1)r_2(1, 2)=0$,

$r_3(1)h_4(1)=0$,

$r_3h_1+h_0(1)h_3=0$,

$r_3b_{13}+h_0(1)h_1(1)=0$,

$r_3h_1(1)+h_0(1)b_{24}=h_0h_2b_{14}$,

$r_3(1)h_1+h_0(1)h_3(1)=0$,

$r_3(1)h_3+r_3h_3(1)=0$,

$r_3(1)b_{13}+h_0(1)h_1(1, 3)=0$,

$r_3(1)h_1(1)+h_0(1)h_5b_{25}=h_0h_2h_5b_{15}$,

$r_3h_1(1, 3)+h_0(1)h_5b_{25}=h_0h_2h_5b_{15}$,

$r_3(1)b_{24}+r_3h_5b_{25}=h_0h_2h_1(1,2)$,

$r_3(1)h_5+r_3b_{46}=0$,

$r_3(1)h_2(1, 3)+r_3h_5h_2(1, 2)=0$,

$r_3(1)r_2(1, 3)+r_3h_5r_2(1, 2)=0$,

$r_3(1)b_{57}+r_3h_5b_{47}=0$,

$r_3(1)b_{14}+r_3h_5b_{15}+h_0(1)h_1(1, 2)=0$,

$r_3(1)h_1(1, 3)+h_0(1)b_{46}b_{25}+h_0(1)h_4^2b_{26}=h_0h_2b_{14,56}$,

$r_3(1)h_1(1, 2)+h_0(1)b_{25}b_{36}+h_0(1)b_{35}b_{26}=h_0h_2b_{13,56}$,

\subsubsection*{Part \ref{item:62c9c1cc}}\hspace{1pt}

$h_{01, 23}r_{3, 4}=b_{0, 2}h_{12, 34}+b_{0, 3}h_{13, 24}$,

$h_{01, 23}r_{34, 56}=b_{0, 2}h_{124, 356}+b_{0, 3}h_{134, 256}$,

$r_{3, 4}r_{34, 56}=b_{01}b_{1, 5}h_{5, 6}+b_{02}b_{2, 5}h_{5, 6}+b_{03}b_{3, 5}h_{5, 6}$,

$h_{01, 23}h_{01, 23}=b_{01, 23}$,

$r_{3, 4}r_{3, 4}=b_{01}b_{1, 4}+b_{02}b_{2, 4}+b_{03}b_{3, 4}$,

$r_{34, 56}r_{34, 56}=b_{01}b_{14, 56}+b_{02}b_{24, 56}+b_{03}b_{34, 56}$.

\vspace{6pt}

\subsection{\mybm{$HX_{7, 4}$}{HX74}}
Consider $d R_{04}=R_{01}R_{14}+R_{02}R_{24}+R_{03}R_{34}$ whose homology class is $r_3$ in $HX_{7, 3}$. 
We have
$$\Ann_{HX_{7, 3}}(r_3)=(h_4, h_4(1)).$$
Apply Proposition \ref{prop:07de5358} on $X_{7, 4}=X_{7, 3}\otimes\bF_2[R_{04}]$. 
The $E_2=E_\infty$ page is generated by $R_{04}h_4$, $R_{04}h_4(1)$ and $R_{04}^2$ which are represented by $r_4$, $r_4(1)$ and $b_{04}$ in $HX_{7, 4}$ respectively. 
In addition to relations in $HX_{7, 3}$, the new relations in $HX_{7, 4}$ are $r_3=0$ and
\subsubsection*{Part \ref{item:df8723cf}}\hspace{1pt}

$r_4h_3=r_2(1)$,

$r_4h_1(1)=0$,

$r_4h_5=0$,

$r_4(1)h_3=r_2(1,3)$,

$r_4(1)h_1(1)=0$,

$r_4(1)h_3(1)=r_2(1, 2)h_5$,

$r_4(1)r_3(1)=0$,

$r_4(1)h_1(1, 3)=0$,

$r_4(1)h_1(1, 2)=0$,

$r_4(1)r_2(1, 2)=h_3(b_{01}b_{14,67}+b_{02}b_{24,67}+b_{03}b_{34,67})$,

$r_4(1)h_4+r_4h_4(1)=0$,

$r_4(1)h_2(1)+r_4h_2(1, 3)=0$,

$r_4(1)b_{35}+r_4h_6b_{36}=r_2(1, 2)h_3$,

$r_4(1)h_1b_{25}+r_4h_1h_6b_{26}=0$,

$r_4(1)h_6+r_4b_{57}=0$,

$r_4(1)b_{02}b_{25}+r_4h_6b_{02}b_{26}+r_4(1)h_0^2b_{15}+r_4h_0^2h_6b_{16}=r_2(1, 2)b_{03}h_3$,

$r_4(1)h_0(1)b_{25}+r_4(1)h_0h_2b_{15}+r_4h_6h_0(1)b_{26}+r_4h_0h_2h_6b_{16}=0$,

$r_4(1)b_{13}b_{25}+r_4h_6b_{13}b_{26}+r_4h_2^2h_6b_{16}+r_4(1)h_2^2b_{15}=0$,

$r_4(1)b_{14}b_{25}+r_4h_6b_{14}b_{26}+r_4h_6b_{24}b_{16}+r_4(1)b_{24}b_{15}=0$,

\subsubsection*{Part \ref{item:62c9c1cc}}\hspace{1pt}

$r_{4, 5}r_{45, 67}=b_{01}b_{1, 6}h_{6, 7}+b_{02}b_{2, 6}h_{6, 7}+b_{03}b_{3, 6}h_{6, 7}+b_{04}b_{4, 6}h_{6, 7}$,

$r_{4, 5}r_{4, 5}=b_{01}b_{1, 5}+b_{02}b_{2, 5}+b_{03}b_{3, 5}+b_{04}b_{4, 5}$,

$r_{45, 67}r_{45, 67}=b_{01}b_{15, 67}+b_{02}b_{25, 67}+b_{03}b_{35, 67}+b_{04}b_{45, 67}$.

\vspace{6pt}

\subsection{\mybm{$HX_{7, 5}$}{HX75}}
Consider $d R_{05}=\sum_{i=1}^r R_{0i}R_{i5}$ whose homology class is $r_4$ in $HX_{7, 4}$. 
We have
$$\Ann_{HX_{7, 4}}(r_4)=(h_1h_3, h_1(1), h_5).$$
Apply Proposition \ref{prop:07de5358} on $X_{7, 5}=X_{7, 4}\otimes\bF_2[R_{05}]$. 
The $E_2=E_\infty$ page is generated by $R_{05}h_1h_3$, $R_{05}h_1(1)$, $R_{05}h_5$ and $R_{05}^2$ which are represented by $h_0(1, 3)$, $h_0(1, 2)$, $r_5$ and $b_{05}$ in $HX_{7, 5}$ respectively. 
In addition to relations in $HX_{7, 4}$, the new relations in $HX_{7, 5}$ are $r_4=0$ and
\subsubsection*{Part \ref{item:df8723cf}}\hspace{1pt}

$h_0(1, 3)h_0=b_{02}h_2(1)$,

$h_0(1, 3)h_2=h_0b_{14}h_4$,

$h_0(1, 3)h_0(1)=b_{01,24}h_4$,

$h_0(1, 2)h_0=b_{03}h_2(1)+b_{04}h_2h_4$,

$h_0(1, 3)h_4=h_0(1)b_{35}$,

$h_0(1, 2)h_0(1)=b_{01,34}h_4$,

$h_0(1, 2)h_4=h_0h_2b_{15}+h_0(1)b_{25}$,

$h_0(1, 3)h_2(1)=h_0b_{13, 45}$,

$r_5h_4=r_3(1)$,

$h_0(1, 2)h_2(1)=h_0b_{12,45}$,

$r_5h_2(1)=h_0h_1(1, 2)$,

$r_5h_6=0$,

$h_0(1, 3)h_4(1)=h_0(1)b_{36}h_6$,

$h_0(1, 2)h_4(1)=h_0h_2b_{16}h_6+h_0(1)b_{26}h_6$,

$h_0(1, 3)r_4(1)=0$,

$h_0(1, 3)h_2(1, 3)=h_0b_{13,46}h_6$,

$h_0(1, 2)r_4(1)=0$,

$h_0(1, 2)h_2(1, 3)=h_0b_{12,46}h_6$,

$h_0(1, 3)h_2(1, 2)=h_0b_{13,56}h_6$,

$h_0(1, 3)r_2(1, 2)=0$,

$h_0(1, 2)h_2(1, 2)=h_0b_{12,56}h_6$,

$h_0(1, 2)r_2(1, 2)=0$,

$h_0(1, 2)h_1^2+h_0(1, 3)b_{13}=h_0(1)h_4b_{14}$,

$h_0(1, 2)b_{02}+h_0(1, 3)b_{03}=h_0(1)h_4b_{04}$,

$h_0(1, 2)h_1h_3+h_0(1, 3)h_1(1)=0$,

$h_0(1, 2)h_3^2+h_0(1, 3)b_{24}=h_0h_2(1)b_{14}$,

$r_5h_1h_3+h_0(1, 3)h_5=0$,

$r_5h_1(1)+h_0(1, 2)h_5=0$,

$h_0(1, 2)b_{35}+h_0(1, 3)b_{25}=h_0h_2(1)b_{15}$,

$r_5h_1b_{35}+h_0(1, 3)h_3(1)=0$,

$r_5h_1b_{25}+h_0(1, 2)h_3(1)=0$,

$h_0(1, 2)h_1h_3(1)+h_0(1, 3)h_1(1, 3)=h_0(1)h_4h_1(1,2)$,

$r_5h_1h_3(1)+h_0(1, 3)b_{46}=h_0(1)h_4b_{36}$,

$r_5h_1(1, 3)+h_0(1, 2)b_{46}=h_0(1)h_4b_{26}+h_0h_2h_4b_{16}$,

$r_5b_{13}b_{25}+r_5h_2^2b_{15}+h_0(1, 2)h_1(1, 3)=0$,

$r_5b_{35}b_{14}+h_0(1, 3)h_1(1, 2)+r_5h_3^2b_{15}=0$,

$r_5b_{14}b_{25}+h_0(1, 2)h_1(1, 2)+r_5b_{24}b_{15}=0$,

$r_5h_1(1, 2)+h_0(1, 2)b_{36}+h_0(1, 3)b_{26}=h_0h_2(1)b_{16}$,

\subsubsection*{Part \ref{item:62c9c1cc}}\hspace{1pt}

$h_{013, 245}h_{012, 345}=b_{013, 345}$,

$h_{013, 245}r_{5, 6}=b_{0, 2}h_{123, 456}+b_{0, 4}h_{134, 256}+b_{0, 5}h_{135, 246}$,

$h_{012, 345}r_{5, 6}=b_{0, 3}h_{123, 456}+b_{0, 4}h_{124, 356}+b_{0, 5}h_{125, 346}$,

$h_{013, 245}h_{013, 245}=b_{013, 245}$,

$h_{012, 345}h_{012, 345}=b_{012, 345}$,

$r_{5, 6}r_{5, 6}=b_{01}b_{1, 6}+b_{02}b_{2, 6}+b_{03}b_{3, 6}+b_{04}b_{4, 6}+b_{05}b_{5, 6}$.

\vspace{6pt}

\subsection{\mybm{$HX_{7, 6}$}{HX76}}
Consider $d R_{06}=\sum_{i=1}^r R_{0i}R_{i6}$ whose homology class is $r_5$ in $HX_{7, 5}$. 
We have
$$\Ann_{HX_{7, 5}}(r_5)=(h_6).$$
Apply Proposition \ref{prop:07de5358} on $X_{7, 6}=X_{7, 5}\otimes\bF_2[R_{06}]$. 
The $E_2=E_\infty$ page is generated by $R_{06}h_6$ and $R_{06}^2$ which are represented by $r_6$ and $b_{02}$ in $HX_{7, 6}$ respectively. 
In addition to relations in $HX_{7, 5}$, the new relations in $HX_{7, 6}$ are $r_5=0$ and
\subsubsection*{Part \ref{item:df8723cf}}\hspace{1pt}

$r_6h_5=r_4(1)$,

$r_6h_3(1)=r_2(1, 2)$,

$r_6h_1(1, 3)=0$,

$r_6h_1(1, 2)=0$,

\subsubsection*{Part \ref{item:62c9c1cc}}\hspace{1pt}

$r_{6, 7}r_{6, 7}=b_{01}b_{1, 7}+b_{02}b_{2, 7}+b_{03}b_{3, 7}+b_{04}b_{4, 7}+b_{05}b_{5, 7}+b_{06}b_{6, 7}$.

\vspace{6pt}

\subsection{\mybm{$HX_{7, 7}$}{HX77}}\label{app:97651fc5}
Consider $d R_{07}=\sum_{i=1}^r R_{0i}R_{i7}$ whose homology class is $r_6$ in $HX_{7, 6}$. 
We have
$$\Ann_{HX_{7, 6}}(r_6)=(h_1h_3h_5, h_1(1)h_5, h_1h_3(1), h_1(1, 3), h_1(1, 2)).$$
Apply Proposition \ref{prop:07de5358} on $X_{7, 7}=X_{7, 6}\otimes\bF_2[R_{07}]$. 
The $E_2=E_\infty$ page is generated by $R_{07}h_1h_3h_5$, $R_{07}h_1(1)h_5$, $R_{07}h_1h_3(1)$, $R_{07}h_1(1, 3)$, $R_{07}h_1(1, 2)$ and $R_{07}^2$ which are represented by $h_0(1, 3, 5)$, $h_0(1, 2, 5)$, $h_0(1, 3, 4)$, $h_0(1, 2, 4)$, $h_0(1, 2, 3)$ and $b_{07}$ in $HX_{7, 7}$ respectively. 
In addition to relations in $HX_{7, 6}$, the new relations in $HX_{7, 7}$ are $r_6=0$ and
\subsubsection*{Part \ref{item:df8723cf}}\hspace{1pt}

$h_0(1, 3, 5)h_0=b_{02}h_2(1, 3)$,

$h_0(1, 3, 5)h_2=h_0(1)h_2(1, 3)$,

$h_0(1, 3, 5)h_0(1)=b_{01, 24}h_4(1)$,

$h_0(1, 2, 5)h_0=b_{03}h_2(1, 3) + b_{04}h_2h_4(1)$,

$h_0(1, 3, 5)h_4=h_0(1, 3)h_4(1)$,

$h_0(1, 2, 5)h_0(1)=b_{01, 34}h_4(1)$,

$h_0(1, 2, 5)h_4=h_0(1, 2)h_4(1)$,

$h_0(1, 3, 4)h_0=b_{02}h_2(1, 2)$,

$h_0(1, 3, 5)h_2(1)=h_0(1, 3)h_2(1, 3)$,

$h_0(1, 3, 4)h_2=h_0(1)h_2(1, 2)$,

$h_0(1, 3, 5)h_0(1, 3)=b_{013,246}h_6$,

$h_0(1, 3, 4)h_0(1)=b_{01,25}h_4(1) + b_{01,26}h_4h_6$,

$h_0(1, 2, 4)h_0=b_{03}h_2(1, 2) + b_{05}h_2h_4(1) + b_{06}h_2h_4h_6$,

$h_0(1, 2, 5)h_2(1)=h_0(1, 2)h_2(1, 3)$,

$h_0(1, 3, 5)h_0(1, 2)=b_{012,246}h_6$,

$h_0(1, 2, 5)h_0(1, 3)=b_{013,346}h_6$,

$h_0(1, 2, 4)h_0(1)=b_{01,35}h_4(1) + b_{01,36}h_4h_6$,

$h_0(1, 2, 5)h_0(1, 2)=b_{012,346}h_6$,

$h_0(1, 2, 3)h_0=b_{0, 4}h_2(1, 2)+b_{0, 5}h_2(1, 3)+b_{0, 6}h_2(1)h_6$,

$h_0(1, 3, 5)h_6=h_0(1, 3)b_{57}$,

$h_0(1, 2, 3)h_0(1)=b_{01,45}h_4(1)+b_{01,46}h_4h_6$,

$h_0(1, 3, 4)h_2(1)=h_0(1, 3)h_2(1, 2)$,

$h_0(1, 2, 5)h_6=h_0(1, 2)b_{57}$,

$h_0(1, 3, 4)h_0(1, 3)=b_{013,256}h_6$,

$h_0(1, 3, 4)h_0(1, 2)=b_{012,256}h_6$,

$h_0(1, 2, 4)h_0(1, 3)=b_{013,356}h_6$,

$h_0(1, 2, 4)h_0(1, 2)=b_{012,356}h_6$,

$h_0(1, 3, 4)h_6=h_0(1)h_4b_{3, 7}+h_0(1, 3)b_{4, 7}$,

$h_0(1, 2, 3)h_0(1, 3)=b_{013,456}h_6$,

$h_0(1, 2, 4)h_6=h_0h_2h_4b_{17} + h_0(1)h_4b_{27} + h_0(1, 2)b_{47}$,

$h_0(1, 2, 3)h_0(1, 2)=b_{012,456}h_6$,

$h_0(1, 2, 3)h_6=h_0h_2(1)b_{17} + h_0(1, 3)b_{27} + h_0(1, 2)b_{37}$,

$h_0(1, 3, 5)h_4(1)=h_0(1)b_{35,67}$,

$h_0(1, 2, 5)h_4(1)=h_0h_2b_{15,67} + h_0(1)b_{25,67}$,

$h_0(1, 3, 5)h_2(1, 3)=h_0b_{135,467}$,

$h_0(1, 2, 5)h_2(1, 3)=h_0b_{125,467}$,

$h_0(1, 3, 4)h_4(1)=h_0(1)b_{34,67}$,

$h_0(1, 2, 4)h_4(1)=h_0h_2b_{14,67} + h_0(1)b_{24,67}$,

$h_0(1, 3, 4)h_2(1, 3)=h_0b_{134,467}$,

$h_0(1, 3, 5)h_2(1, 2)=h_0b_{135,567}$,

$h_0(1, 2, 3)h_4(1)=h_0h_2b_{13,67} + h_0(1)b_{23,67}$,

$h_0(1, 2, 5)h_2(1, 2)=h_0b_{125,567}$,

$h_0(1, 2, 4)h_2(1, 3)=h_0b_{124,467}$,

$h_0(1, 2, 3)h_2(1, 3)=h_0b_{123,467}$,

$h_0(1, 3, 4)h_2(1, 2)=h_0b_{134,567}$,

$h_0(1, 2, 4)h_2(1, 2)=h_0b_{124,567}$,

$h_0(1, 2, 3)h_2(1, 2)=h_0b_{123,567}$,

$h_0(1, 2, 5)h_1^2+h_0(1, 3, 5)b_{13}=h_0(1)h_4(1)b_{14}$,

$h_0(1, 2, 5)b_{02}+h_0(1, 3, 5)b_{03}=h_0(1)h_4(1)b_{04}$,

$h_0(1, 2, 5)h_1h_3+h_0(1, 3, 5)h_1(1)=0$,

$h_0(1, 2, 5)h_3^2+h_0(1, 3, 5)b_{24}=h_0h_2(1, 3)b_{14}$,

$h_0(1, 2, 4)h_1^2+h_0(1, 3, 4)b_{13}=h_0(1)h_4(1)b_{15} + h_0(1)h_4h_6b_{16}$,

$h_0(1, 2, 4)b_{02}+h_0(1, 3, 4)b_{03}=h_0(1)h_4(1)b_{05} + h_0(1)h_4h_6b_{06}$,

$h_0(1, 3, 4)h_3^2+h_0(1, 3, 5)b_{35}=h_0(1, 3)h_6b_{36}$,

$h_0(1, 2, 4)h_1h_3+h_0(1, 3, 4)h_1(1)=0$,

$h_0(1, 3, 4)b_{24}+h_0(1, 3, 5)b_{25}=h_0(1, 3)h_6b_{26}$,

$h_0(1, 2, 5)b_{35}+h_0(1, 3, 5)b_{25}=h_0h_2(1, 3)b_{15}$,

$h_0(1, 2, 4)h_3^2+h_0(1, 2, 5)b_{35}=h_0(1, 2)h_6b_{36}$,

$h_0(1, 3, 4)h_3h_5+h_0(1, 3, 5)h_3(1)=0$,

$h_0(1, 2, 3)h_2h_4+h_0(1, 2, 4)h_2(1)=h_0(1, 2)h_2(1, 2)$,

$h_0(1, 2, 4)h_3h_5+h_0(1, 2, 5)h_3(1)=0$,

$h_0(1, 2, 4)h_1h_3h_5+h_0(1, 3, 5)h_1(1, 3)=0$,

$h_0(1, 2, 4)h_5h_1(1)+h_0(1, 2, 5)h_1(1, 3)=0$,

$h_0(1, 3, 4)h_5^2+h_0(1, 3, 5)b_{46}=h_0(1)h_4(1)b_{36}$,

$h_0(1, 2, 3)h_1h_3h_5+h_0(1, 3, 5)h_1(1, 2)=0$,

$h_0(1, 2, 4)h_5^2+h_0(1, 2, 5)b_{46}=h_0(1)h_4(1)b_{26} + h_0h_2h_4(1)b_{16}$,

$h_0(1, 2, 3)h_5h_1(1)+h_0(1, 2, 5)h_1(1, 2)=0$,

$h_0(1, 2, 4)h_1h_3(1)+h_0(1, 3, 4)h_1(1, 3)=0$,

$h_0(1, 2, 3)h_1h_3(1)+h_0(1, 3, 4)h_1(1, 2)=0$,

$h_0(1, 2, 3)h_1(1, 3)+h_0(1, 2, 4)h_1(1, 2)=0$,

$h_0(1, 3, 4)b_{57}+h_0(1, 3, 5)b_{47}=h_0(1)h_4(1)b_{37}$,

$h_0(1, 2, 4)b_{57}+h_0(1, 2, 5)b_{47}=h_0(1)h_2(1)b_{27} + h_0h_2h_4(1)b_{17}$,

$h_0(1, 2, 3)h_1^2+h_0(1, 3, 4)b_{14}+h_0(1, 3, 5)b_{15}=h_0(1, 3)h_6b_{16}$,

$h_0(1, 2, 3)b_{02}+h_0(1, 3, 4)b_{04}+h_0(1, 3, 5)b_{05}=h_0(1, 3)h_6b_{06}$,

$h_0(1, 2, 3)h_2^2+h_0(1, 2, 4)b_{24}+h_0(1, 2, 5)b_{25}=h_0(1, 2)h_6b_{26}$,

$h_0(1, 2, 3)b_{13}+h_0(1, 2, 4)b_{14}+h_0(1, 2, 5)b_{15}=h_0(1, 2)h_6b_{16}$,

$h_0(1, 2, 3)b_{03}+h_0(1, 2, 5)b_{05}+h_0(1, 2, 4)b_{04}=h_0(1, 2)h_6b_{06}$,

$h_0(1, 2, 3)h_4^2+h_0(1, 2, 4)b_{35}+h_0(1, 3, 4)b_{25}=h_0h_2(1, 2)b_{15}$,

$h_0(1, 2, 3)h_5^2+h_0(1, 2, 5)b_{36}+h_0(1, 3, 5)b_{26}=h_0h_2(1, 3)b_{16}$,

$h_0(1, 2, 3)b_{46}+h_0(1, 2, 4)b_{36}+h_0(1, 3, 4)b_{26}=h_0h_2(1, 2)b_{16}$,

$h_0(1, 2, 3)b_{57}+h_0(1, 3, 5)b_{27}+h_0(1, 2, 5)b_{37}=h_0h_2(1, 3)b_{17}$,

$h_0(1, 2, 3)b_{47}+h_0(1, 2, 4)b_{37}+h_0(1, 3, 4)b_{27}=h_0h_2(1, 2)b_{17}$,

\subsubsection*{Part \ref{item:62c9c1cc}}\hspace{1pt}

$h_{0135, 2467}h_{0134, 2567}=b_{0135, 2567}$,

$h_{0135, 2467}h_{0125, 3467}=b_{0135, 3467}$,

$h_{0135, 2467}h_{0124, 3567}=b_{0135, 3567}$,

$h_{0135, 2467}h_{0123, 4567}=b_{0135, 4567}$,

$h_{0134, 2567}h_{0125, 3467}=b_{0134, 3467}$,

$h_{0134, 2567}h_{0124, 3567}=b_{0134, 3567}$,

$h_{0134, 2567}h_{0123, 4567}=b_{0134, 4567}$,

$h_{0125, 3467}h_{0124, 3567}=b_{0125, 3567}$,

$h_{0125, 3467}h_{0123, 4567}=b_{0125, 4567}$,

$h_{0124, 3567}h_{0123, 4567}=b_{0124, 4567}$,

$h_{0135, 2467}h_{0135, 2467}=b_{0135, 2467}$,

$h_{0134, 2567}h_{0134, 2567}=b_{0134, 2567}$,

$h_{0125, 3467}h_{0125, 3467}=b_{0125, 3467}$,

$h_{0124, 3567}h_{0124, 3567}=b_{0124, 3567}$,

$h_{0123, 4567}h_{0123, 4567}=b_{0123, 4567}$.

\vspace{6pt}

\setlength{\parskip}{2pt plus 2pt minus 1pt}
\subsection{A \Groebner{} basis of \mybm{$HX_7$}{HX7}}\label{app:7b101ee1}

\subsubsection*{Monomial ordering}\hspace{1pt}

The monomial ordering we use here is the reversed lexicographical ordering by the sequence of the following generators

\begin{tabular}{|l|l|l|}\hline
    name & degree $(s, t, v)$ & range of $i$\\\hline
    $h_i$ & $(1, 2^i, 1)$ & $0\le i\le 6$\\\hline
    $h_i(1)$ & $(2, 9\cdot 2^i, 4)$ & $0\le i\le 4$\\\hline
    $h_i(1, 3)$ & $(3, 41\cdot 2^i, 7)$ & $0\le i\le 2$\\\hline
    $h_i(1, 2)$ & $(3, 49\cdot 2^i, 9)$ & $0\le i\le 2$\\\hline
    $h_0(1, 3, 5)$ & $(4, 169, 10)$ & \\\hline
    $h_0(1, 2, 5)$ & $(4, 177, 12)$ & \\\hline
    $h_0(1, 3, 4)$ & $(4, 201, 12)$ & \\\hline
    $h_0(1, 2, 4)$ & $(4, 209, 14)$ & \\\hline
    $h_0(1, 2, 3)$ & $(4, 225, 16)$ & \\\hline
    $b_{ij}$ & $(2, 2(2^j-2^i), 2(j-i))$ & $0\le i\le j-2<j\le 7$\\\hline
\end{tabular}

Here $b_{ij}$ is ordered first by $j-i$ and then by $i$.\vspace{6pt}

\subsubsection*{\Groebner{} basis}\hspace{1pt}\footnote{An element $g$ of the \Groebner{} basis here is presented in the form $\LM(g)=g-\LM(g)$}

$h_0h_1 = 0$

$h_1h_2 = 0$

$h_2b_{02} = h_0h_0(1)$

$h_2h_3 = 0$

$h_2h_0(1) = h_0b_{13}$

$h_3h_0(1) = 0$

$b_{02}b_{13} = h_1^2b_{03} + h_0(1)^2$

$h_0h_1(1) = 0$

$h_3b_{13} = h_1h_1(1)$

$h_3h_4 = 0$

$h_1(1)b_{02} = h_1h_3b_{03}$

$h_3h_1(1) = h_1b_{24}$

$h_0(1)h_1(1) = 0$

$b_{02}b_{24} = h_0^2b_{14} + h_3^2b_{03}$

$h_0(1)b_{24} = h_0h_2b_{14}$

$h_4h_1(1) = 0$

$b_{13}b_{24} = h_2^2b_{14} + h_1(1)^2$

$h_1h_2(1) = 0$

$h_4b_{24} = h_2h_2(1)$

$h_2(1)b_{02} = h_0h_0(1, 3)$

$h_0(1)h_2(1) = h_0h_4b_{14}$

$h_2h_0(1, 3) = h_0h_4b_{14}$

$h_4h_5 = 0$

$h_2(1)b_{13} = h_2h_4b_{14}$

$h_2(1)b_{03} = h_0h_0(1, 2) + h_2h_4b_{04}$

$h_0(1)h_0(1, 3) = h_1^2h_4b_{04} + h_4b_{02}b_{14}$

$h_4h_2(1) = h_2b_{35}$

$h_0(1, 3)b_{13} = h_1^2h_0(1, 2) + h_4h_0(1)b_{14}$

$h_1(1)h_2(1) = 0$

$h_0(1, 2)b_{02} = h_4h_0(1)b_{04} + h_0(1, 3)b_{03}$

$h_0(1)b_{35} = h_4h_0(1, 3)$

$h_0(1)h_0(1, 2) = h_4b_{13}b_{04} + h_4b_{03}b_{14}$

$h_1(1)h_0(1, 3) = h_1h_3h_0(1, 2)$

$b_{13}b_{35} = h_1^2b_{25} + h_4^2b_{14}$

$h_2h_4b_{03}b_{14} = h_0h_0(1, 2)b_{13} + h_2h_4b_{13}b_{04}$

$b_{35}b_{03} = h_0^2b_{15} + h_4^2b_{04} + b_{02}b_{25}$

$h_0(1, 3)b_{24} = h_0h_2(1)b_{14} + h_3^2h_0(1, 2)$

$h_0(1)b_{25} = h_0h_2b_{15} + h_4h_0(1, 2)$

$h_1(1)b_{35} = h_1h_3b_{25}$

$h_5h_2(1) = 0$

$h_2h_4h_0(1, 2) = h_0h_2^2b_{15} + h_0b_{13}b_{25}$

$b_{24}b_{35} = h_3^2b_{25} + h_2(1)^2$

$h_5h_0(1, 3) = 0$

$h_2h_3(1) = 0$

$h_2(1)h_0(1, 3) = h_0h_3^2b_{15} + h_0b_{35}b_{14}$

$h_5b_{35} = h_3h_3(1)$

$h_0(1)h_3(1) = 0$

$h_5h_0(1, 2) = 0$

$b_{02}b_{35}b_{14} = h_1^2h_3^2b_{05} + h_1^2b_{35}b_{04} + h_3^2b_{02}b_{15} + h_0(1, 3)^2$

$h_0h_1(1, 3) = 0$

$h_3(1)b_{13} = h_1h_1(1, 3)$

$h_2(1)h_0(1, 2) = h_0b_{24}b_{15} + h_0b_{14}b_{25}$

$h_1(1, 3)b_{02} = h_1h_3(1)b_{03}$

$h_1(1)h_3(1) = h_1h_5b_{25}$

$b_{02}b_{14}b_{25} = h_0^2b_{14}b_{15} + h_1^2b_{24}b_{05} + h_1^2b_{25}b_{04} + h_3^2b_{03}b_{15} + h_0(1, 3)h_0(1, 2)$

$h_3h_1(1, 3) = h_1h_5b_{25}$

$h_0(1)h_1(1, 3) = 0$

$h_5b_{02}b_{25} = h_0^2h_5b_{15} + h_3h_3(1)b_{03}$

$h_5h_6 = 0$

$h_3(1)b_{24} = h_3h_5b_{25}$

$h_0(1, 2)b_{35} = h_0h_2(1)b_{15} + h_0(1, 3)b_{25}$

$b_{03}b_{14}b_{25} = h_2^2b_{04}b_{15} + h_1(1)^2b_{05} + h_0(1, 2)^2 + b_{13}b_{25}b_{04} + b_{24}b_{03}b_{15}$

$h_0h_1(1, 2) = 0$

$h_3(1)b_{14} = h_1h_1(1, 2) + h_3h_5b_{15}$

$h_1(1)h_1(1, 3) = h_2^2h_5b_{15} + h_5b_{13}b_{25}$

$h_0h_5b_{13}b_{25} = h_0h_2^2h_5b_{15}$

$h_1(1, 2)b_{02} = h_1h_3h_5b_{05} + h_1h_3(1)b_{04}$

$h_5h_3(1) = h_3b_{46}$

$h_1(1, 3)b_{24} = h_2^2h_1(1, 2) + h_5h_1(1)b_{25}$

$h_0(1)h_1(1, 2) = 0$

$h_2(1)h_3(1) = 0$

$h_1(1, 2)b_{13} = h_5h_1(1)b_{15} + h_1(1, 3)b_{14}$

$h_1(1, 2)b_{03} = h_5h_1(1)b_{05} + h_1(1, 3)b_{04}$

$h_3(1)h_0(1, 3) = 0$

$h_1(1)b_{46} = h_5h_1(1, 3)$

$h_1(1)h_1(1, 2) = h_5b_{24}b_{15} + h_5b_{14}b_{25}$

$h_0h_5b_{14}b_{25} = h_0h_5b_{24}b_{15}$

$h_2(1)h_1(1, 3) = h_2h_4h_1(1, 2)$

$b_{24}b_{46} = h_2^2b_{36} + h_5^2b_{25}$

$h_3(1)h_0(1, 2) = 0$

$h_0(1, 3)h_1(1, 3) = 0$

$h_1(1, 3)b_{03}b_{14} = h_5h_1(1)b_{13}b_{05} + h_5h_1(1)b_{03}b_{15} + h_1(1, 3)b_{13}b_{04}$

$h_3h_5b_{14}b_{25} = h_1h_1(1, 2)b_{24} + h_3h_5b_{24}b_{15}$

$b_{46}b_{14} = h_1^2b_{26} + h_5^2b_{15} + b_{13}b_{36}$

$b_{03}b_{36} = h_0^2b_{16} + h_5^2b_{05} + b_{02}b_{26} + b_{46}b_{04}$

$h_1(1, 3)b_{35} = h_1h_3(1)b_{25} + h_4^2h_1(1, 2)$

$h_1(1)b_{36} = h_1h_3b_{26} + h_5h_1(1, 2)$

$h_1(1, 3)h_0(1, 2) = 0$

$h_2(1)b_{46} = h_2h_4b_{36}$

$h_6h_3(1) = 0$

$h_0(1, 3)b_{46} = h_4h_0(1)b_{36}$

$h_3h_5h_1(1, 2) = h_1h_3^2b_{26} + h_1b_{24}b_{36}$

$h_0(1, 3)h_1(1, 2) = 0$

$b_{35}b_{46} = h_4^2b_{36} + h_3(1)^2$

$h_0(1, 2)b_{46} = h_0h_2h_4b_{16} + h_4h_0(1)b_{26}$

$h_6h_1(1, 3) = 0$

$h_0(1, 2)h_1(1, 2) = 0$

$h_3h_4(1) = 0$

$h_3(1)h_1(1, 3) = h_1h_4^2b_{26} + h_1b_{46}b_{25}$

$b_{02}b_{46}b_{25} = h_0^2h_4^2b_{16} + h_0^2b_{46}b_{15} + h_4^2b_{02}b_{26} + h_3(1)^2b_{03}$

$h_6b_{46} = h_4h_4(1)$

$h_0(1, 2)b_{36} = h_0h_2(1)b_{16} + h_0(1, 3)b_{26}$

$h_1(1)h_4(1) = 0$

$h_6h_1(1, 2) = 0$

$b_{13}b_{46}b_{25} = h_2^2h_4^2b_{16} + h_2^2b_{46}b_{15} + h_4^2b_{13}b_{26} + h_1(1, 3)^2$

$h_1h_2(1, 3) = 0$

$h_4(1)b_{24} = h_2h_2(1, 3)$

$h_2(1, 3)b_{02} = h_0h_0(1, 3, 5)$

$h_3(1)h_1(1, 2) = h_1b_{35}b_{26} + h_1b_{25}b_{36}$

$h_0(1)h_2(1, 3) = h_0h_4(1)b_{14}$

$h_2h_0(1, 3, 5) = h_0h_4(1)b_{14}$

$b_{02}b_{25}b_{36} = h_0^2b_{35}b_{16} + h_0^2b_{36}b_{15} + h_3^2b_{46}b_{05} + h_3(1)^2b_{04} + b_{02}b_{35}b_{26}$

$h_2(1, 3)b_{13} = h_2h_4(1)b_{14}$

$h_2(1, 3)b_{03} = h_0h_0(1, 2, 5) + h_2h_4(1)b_{04}$

$h_4(1)b_{02}b_{14} = h_1^2h_4(1)b_{04} + h_0(1)h_0(1, 3, 5)$

$h_2(1)h_4(1) = h_2h_6b_{36}$

$b_{13}b_{25}b_{36} = h_1^2b_{25}b_{26} + h_2^2b_{35}b_{16} + h_2^2b_{36}b_{15} + h_4^2b_{14}b_{26} + h_1(1, 3)h_1(1, 2)$

$h_4h_2(1, 3) = h_2h_6b_{36}$

$h_0(1, 3, 5)b_{13} = h_1^2h_0(1, 2, 5) + h_0(1)h_4(1)b_{14}$

$h_1(1)h_2(1, 3) = 0$

$h_0(1, 2, 5)b_{02} = h_0(1)h_4(1)b_{04} + h_0(1, 3, 5)b_{03}$

$h_4(1)h_0(1, 3) = h_4h_0(1, 3, 5)$

$h_6h_0(1)b_{36} = h_4h_0(1, 3, 5)$

$h_4(1)b_{03}b_{14} = h_0(1)h_0(1, 2, 5) + h_4(1)b_{13}b_{04}$

$h_1(1)h_0(1, 3, 5) = h_1h_3h_0(1, 2, 5)$

$h_6b_{13}b_{36} = h_1^2h_6b_{26} + h_4h_4(1)b_{14}$

$h_4(1)b_{35} = h_4h_6b_{36}$

$h_0(1, 3, 5)b_{24} = h_0h_2(1, 3)b_{14} + h_3^2h_0(1, 2, 5)$

$h_4(1)h_0(1, 2) = h_4h_0(1, 2, 5)$

$h_6h_0(1)b_{26} = h_0h_2h_6b_{16} + h_4h_0(1, 2, 5)$

$h_1(1, 2)b_{46} = h_1h_3(1)b_{26} + h_1(1, 3)b_{36}$

$b_{14}b_{25}b_{36} = h_3^2b_{15}b_{26} + h_2(1)^2b_{16} + h_1(1, 2)^2 + b_{24}b_{36}b_{15} + b_{35}b_{14}b_{26}$

$h_2h_4h_0(1, 2, 5) = h_0h_2^2h_6b_{16} + h_0h_6b_{13}b_{26}$

$h_1h_2(1, 2) = 0$

$h_4(1)b_{25} = h_2h_2(1, 2) + h_4h_6b_{26}$

$h_2(1)h_2(1, 3) = h_3^2h_6b_{26} + h_6b_{24}b_{36}$

$h_1h_6b_{24}b_{36} = h_1h_3^2h_6b_{26}$

$h_2(1, 2)b_{02} = h_0h_0(1, 3, 4)$

$h_0(1)h_2(1, 2) = h_0h_4h_6b_{16} + h_0h_4(1)b_{15}$

$h_0(1, 3)h_2(1, 3) = h_0h_3^2h_6b_{16} + h_0h_6b_{14}b_{36}$

$h_2h_0(1, 3, 4) = h_0h_4h_6b_{16} + h_0h_4(1)b_{15}$

$h_2(1)h_0(1, 3, 5) = h_0h_3^2h_6b_{16} + h_0h_6b_{14}b_{36}$

$h_2(1, 2)b_{13} = h_2h_4h_6b_{16} + h_2h_4(1)b_{15}$

$h_6h_4(1) = h_4b_{57}$

$h_0(1, 3)b_{25}b_{36} = h_0h_2(1)b_{35}b_{16} + h_0h_2(1)b_{36}b_{15} + h_0(1, 3)b_{35}b_{26}$

$h_2(1, 2)b_{03} = h_0h_0(1, 2, 4) + h_2h_4h_6b_{06} + h_2h_4(1)b_{05}$

$h_4(1)b_{02}b_{15} = h_1^2h_4h_6b_{06} + h_1^2h_4(1)b_{05} + h_4h_6b_{02}b_{16} + h_0(1)h_0(1, 3, 4)$

$h_0(1, 3)h_0(1, 3, 5) = h_1^2h_3^2h_6b_{06} + h_1^2h_6b_{36}b_{04} + h_3^2h_6b_{02}b_{16} + h_6b_{02}b_{14}b_{36}$

$h_2(1, 3)b_{35} = h_3^2h_2(1, 2) + h_6h_2(1)b_{36}$

$h_0(1, 3, 4)b_{13} = h_1^2h_0(1, 2, 4) + h_4h_6h_0(1)b_{16} + h_0(1)h_4(1)b_{15}$

$h_2(1, 3)h_0(1, 2) = h_0h_6b_{24}b_{16} + h_0h_6b_{14}b_{26}$

$h_2(1)h_0(1, 2, 5) = h_0h_6b_{24}b_{16} + h_0h_6b_{14}b_{26}$

$h_1(1)h_2(1, 2) = 0$

$h_0(1, 2, 4)b_{02} = h_4h_6h_0(1)b_{06} + h_0(1)h_4(1)b_{05} + h_0(1, 3, 4)b_{03}$

$h_3(1)h_4(1) = 0$

$h_0(1, 3, 5)b_{35} = h_3^2h_0(1, 3, 4) + h_6h_0(1, 3)b_{36}$

$h_0(1, 2)h_0(1, 3, 5) = h_0^2h_6b_{14}b_{16} + h_1^2h_6b_{24}b_{06} + h_1^2h_6b_{04}b_{26} + h_3^2h_6b_{03}b_{16} + h_6b_{02}b_{14}b_{26}$

$h_4(1)b_{03}b_{15} = h_4h_6b_{13}b_{06} + h_4h_6b_{03}b_{16} + h_0(1)h_0(1, 2, 4) + h_4(1)b_{13}b_{05}$

$h_0(1, 3)h_0(1, 2, 5) = h_0^2h_6b_{14}b_{16} + h_1^2h_6b_{24}b_{06} + h_1^2h_6b_{04}b_{26} + h_3^2h_6b_{03}b_{16} + h_6b_{02}b_{14}b_{26}$

$h_1(1)h_0(1, 3, 4) = h_1h_3h_0(1, 2, 4)$

$h_2(1, 2)b_{24} = h_6h_2(1)b_{26} + h_2(1, 3)b_{25}$

$h_2(1, 2)b_{14} = h_6h_2(1)b_{16} + h_2(1, 3)b_{15}$

$h_0(1, 3, 4)b_{24} = h_0h_6h_2(1)b_{16} + h_0h_2(1, 3)b_{15} + h_3^2h_0(1, 2, 4)$

$h_0(1, 2, 5)b_{35} = h_0h_6h_2(1)b_{16} + h_3^2h_0(1, 2, 4) + h_6h_0(1, 3)b_{26}$

$h_0(1, 3, 5)b_{25} = h_0h_6h_2(1)b_{16} + h_0h_2(1, 3)b_{15} + h_3^2h_0(1, 2, 4) + h_6h_0(1, 3)b_{26}$

$h_4(1)h_1(1, 3) = 0$

$h_2(1, 2)b_{04} = h_0h_0(1, 2, 3) + h_6h_2(1)b_{06} + h_2(1, 3)b_{05}$

$h_0(1, 2)h_0(1, 2, 5) = h_2^2h_6b_{04}b_{16} + h_6h_1(1)^2b_{06} + h_6b_{13}b_{04}b_{26} + h_6b_{24}b_{03}b_{16} + h_6b_{03}b_{14}b_{26}$

$h_2(1)b_{57} = h_6h_2(1, 3)$

$h_0(1, 3, 4)b_{14} = h_1^2h_0(1, 2, 3) + h_6h_0(1, 3)b_{16} + h_0(1, 3, 5)b_{15}$

$h_0(1, 2, 3)b_{02} = h_6h_0(1, 3)b_{06} + h_0(1, 3, 5)b_{05} + h_0(1, 3, 4)b_{04}$

$h_2(1)h_2(1, 2) = h_6b_{35}b_{26} + h_6b_{25}b_{36}$

$h_0(1, 2, 4)b_{24} = h_2^2h_0(1, 2, 3) + h_6h_0(1, 2)b_{26} + h_0(1, 2, 5)b_{25}$

$h_0(1, 3)b_{57} = h_6h_0(1, 3, 5)$

$h_1h_6b_{25}b_{36} = h_1h_6b_{35}b_{26}$

$h_4(1)b_{04}b_{15} = h_4h_6b_{14}b_{06} + h_4h_6b_{04}b_{16} + h_0(1)h_0(1, 2, 3) + h_4(1)b_{14}b_{05}$

$h_3(1)h_2(1, 3) = h_3h_5h_2(1, 2)$

$h_0(1, 2, 3)b_{13} = h_6h_0(1, 2)b_{16} + h_0(1, 2, 5)b_{15} + h_0(1, 2, 4)b_{14}$

$h_0(1, 3)h_2(1, 2) = h_0h_6b_{35}b_{16} + h_0h_6b_{36}b_{15}$

$h_2(1)h_0(1, 3, 4) = h_0h_6b_{35}b_{16} + h_0h_6b_{36}b_{15}$

$h_0(1, 2, 3)b_{03} = h_6h_0(1, 2)b_{06} + h_0(1, 2, 5)b_{05} + h_0(1, 2, 4)b_{04}$

$b_{35}b_{57} = h_3^2b_{47} + h_6^2b_{36}$

$h_0(1, 2)b_{57} = h_6h_0(1, 2, 5)$

$h_3(1)h_0(1, 3, 5) = h_3h_5h_0(1, 3, 4)$

$h_6h_0(1)h_0(1, 3, 5) = h_1^2h_4b_{57}b_{04} + h_4b_{02}b_{57}b_{14}$

$h_4(1)h_1(1, 2) = 0$

$h_0(1, 3)h_0(1, 3, 4) = h_1^2h_6b_{35}b_{06} + h_1^2h_6b_{36}b_{05} + h_6b_{02}b_{35}b_{16} + h_6b_{02}b_{36}b_{15}$

$h_3h_0(1, 2, 4)b_{14} = h_1h_1(1)h_0(1, 2, 3) + h_3h_6h_0(1, 2)b_{16} + h_3h_0(1, 2, 5)b_{15}$

$h_0(1, 2, 5)b_{13}b_{25} = h_2^2h_6h_0(1, 2)b_{16} + h_2^2h_0(1, 2, 5)b_{15} + h_6h_0(1, 2)b_{13}b_{26} + h_1(1)^2h_0(1, 2, 4)$

$h_0(1, 2)h_2(1, 2) = h_0h_6b_{25}b_{16} + h_0h_6b_{15}b_{26}$

$h_2(1)h_0(1, 2, 4) = h_0h_6b_{25}b_{16} + h_0h_6b_{15}b_{26} + h_2h_4h_0(1, 2, 3)$

$h_1(1, 3)h_2(1, 3) = 0$

$h_2(1, 3)b_{14}b_{25} = h_6h_2(1)b_{24}b_{16} + h_6h_2(1)b_{14}b_{26} + h_2(1, 3)b_{24}b_{15}$

$h_4h_6b_{25}b_{36} = h_2h_2(1, 2)b_{35} + h_4h_6b_{35}b_{26}$

$b_{57}b_{25} = h_2^2b_{37} + h_6^2b_{26} + b_{24}b_{47}$

$h_4h_0(1, 2, 5)b_{25} = h_0h_2h_6b_{25}b_{16} + h_0h_2h_6b_{15}b_{26} + h_4h_6h_0(1, 2)b_{26}$

$h_3(1)h_0(1, 2, 5) = h_3h_5h_0(1, 2, 4)$

$h_6h_0(1)h_0(1, 2, 5) = h_4b_{13}b_{57}b_{04} + h_4b_{57}b_{03}b_{14}$

$h_2(1, 3)b_{25}b_{04} = h_0h_0(1, 2, 3)b_{24} + h_6h_2(1)b_{24}b_{06} + h_6h_2(1)b_{04}b_{26} + h_2(1, 3)b_{24}b_{05}$

$h_0(1, 2)h_0(1, 3, 4) = h_1^2h_6b_{25}b_{06} + h_1^2h_6b_{26}b_{05} + h_6b_{02}b_{25}b_{16} + h_6b_{02}b_{15}b_{26}$

$h_0(1, 3)h_0(1, 2, 4) = h_1^2h_6b_{25}b_{06} + h_1^2h_6b_{26}b_{05} + h_4h_0(1)h_0(1, 2, 3) + h_6b_{02}b_{25}b_{16} + h_6b_{02}b_{15}b_{26}$

$h_0(1, 2, 4)b_{03}b_{14} = h_6h_0(1, 2)b_{13}b_{06} + h_6h_0(1, 2)b_{03}b_{16} + h_0(1, 2, 5)b_{13}b_{05} + h_0(1, 2, 5)b_{03}b_{15} + h_0(1, 2, 4)b_{13}b_{04}$

$h_1(1, 3)h_0(1, 3, 5) = h_1h_3h_5h_0(1, 2, 4)$

$b_{14}b_{47} = h_1^2b_{27} + h_6^2b_{16} + b_{13}b_{37} + b_{57}b_{15}$

$h_2(1, 3)b_{04}b_{15} = h_0h_0(1, 2, 3)b_{14} + h_6h_2(1)b_{14}b_{06} + h_6h_2(1)b_{04}b_{16} + h_2(1, 3)b_{14}b_{05}$

$b_{47}b_{04} = h_0^2b_{17} + h_6^2b_{06} + b_{02}b_{27} + b_{57}b_{05} + b_{03}b_{37}$

$h_0(1, 2, 4)b_{35} = h_0h_2(1, 2)b_{15} + h_4^2h_0(1, 2, 3) + h_0(1, 3, 4)b_{25}$

$h_0(1, 2)h_0(1, 2, 4) = h_2^2h_6b_{15}b_{06} + h_2^2h_6b_{05}b_{16} + h_6b_{13}b_{25}b_{06} + h_6b_{13}b_{26}b_{05} + h_6b_{03}b_{25}b_{16} + h_6b_{03}b_{15}b_{26}$

$h_1(1, 3)h_0(1, 2, 5) = h_5h_1(1)h_0(1, 2, 4)$

$h_2(1, 3)b_{46} = h_2h_4(1)b_{36} + h_5^2h_2(1, 2)$

$h_2(1)b_{47} = h_2h_4b_{37} + h_6h_2(1, 2)$

$h_0(1, 2, 5)b_{14}b_{25} = h_6h_0(1, 2)b_{24}b_{16} + h_6h_0(1, 2)b_{14}b_{26} + h_1(1)^2h_0(1, 2, 3) + h_0(1, 2, 5)b_{24}b_{15}$

$h_6b_{25}b_{36}b_{04} = h_0h_2(1)h_0(1, 2, 3) + h_3^2h_6b_{26}b_{05} + h_6h_2(1)^2b_{06} + h_6b_{24}b_{36}b_{05} + h_6b_{35}b_{04}b_{26}$

$h_2(1, 3)h_1(1, 2) = 0$

$h_3(1)b_{57} = h_3h_5b_{47}$

$h_0(1, 3, 5)b_{46} = h_5^2h_0(1, 3, 4) + h_0(1)h_4(1)b_{36}$

$h_0(1, 3)b_{47} = h_4h_0(1)b_{37} + h_6h_0(1, 3, 4)$

$h_0(1, 3)h_0(1, 2, 3) = h_3^2h_6b_{15}b_{06} + h_3^2h_6b_{05}b_{16} + h_6b_{35}b_{14}b_{06} + h_6b_{35}b_{04}b_{16} + h_6b_{14}b_{36}b_{05} + h_6b_{36}b_{04}b_{15}$

$h_1(1, 2)h_0(1, 3, 5) = h_1h_3h_5h_0(1, 2, 3)$

$h_2h_6b_{36}b_{04}b_{15} = h_0h_4h_0(1, 2, 3)b_{14} + h_2h_6b_{35}b_{14}b_{06} + h_2h_6b_{35}b_{04}b_{16} + h_2h_6b_{14}b_{36}b_{05}$

$h_0(1, 2, 5)b_{46} = h_0h_2h_4(1)b_{16} + h_5^2h_0(1, 2, 4) + h_0(1)h_4(1)b_{26}$

$h_0(1, 2)b_{47} = h_0h_2h_4b_{17} + h_4h_0(1)b_{27} + h_6h_0(1, 2, 4)$

$h_6h_0(1)h_0(1, 3, 4) = h_1^2h_4h_6^2b_{06} + h_1^2h_4b_{57}b_{05} + h_4h_6^2b_{02}b_{16} + h_4b_{02}b_{57}b_{15}$

$h_1(1, 3)b_{57} = h_5h_1(1)b_{47}$

$h_0(1, 2)h_0(1, 2, 3) = h_6b_{24}b_{15}b_{06} + h_6b_{24}b_{05}b_{16} + h_6b_{14}b_{25}b_{06} + h_6b_{14}b_{26}b_{05} + h_6b_{25}b_{04}b_{16} + h_6b_{04}b_{15}b_{26}$

$h_1(1, 2)h_0(1, 2, 5) = h_5h_1(1)h_0(1, 2, 3)$

$h_4h_6h_2(1, 2) = h_2h_4^2b_{37} + h_2b_{35}b_{47}$

$h_2h_4h_6b_{03}b_{15}b_{26} = h_0h_2^2h_0(1, 2, 4)b_{15} + h_0h_0(1, 2, 4)b_{13}b_{25} + h_2^3h_4h_6b_{15}b_{06} + h_2^3h_4h_6b_{05}b_{16} + h_2h_4h_6b_{13}b_{25}b_{06} + h_2h_4h_6b_{13}b_{26}b_{05} + h_2h_4h_6b_{03}b_{25}b_{16}$

$h_1(1, 3)h_2(1, 2) = 0$

$b_{03}b_{14}b_{37} = h_0^2b_{14}b_{17} + h_1^2b_{04}b_{27} + h_6^2b_{14}b_{06} + h_6^2b_{04}b_{16} + b_{02}b_{14}b_{27} + b_{13}b_{04}b_{37} + b_{57}b_{14}b_{05} + b_{57}b_{04}b_{15}$

$h_6h_0(1)h_0(1, 2, 4) = h_4h_6^2b_{13}b_{06} + h_4h_6^2b_{03}b_{16} + h_4b_{13}b_{57}b_{05} + h_4b_{57}b_{03}b_{15}$

$h_1(1, 3)h_0(1, 3, 4) = h_1h_3(1)h_0(1, 2, 4)$

$h_2h_4b_{57}b_{03}b_{15} = h_0h_6h_0(1, 2, 4)b_{13} + h_2h_4h_6^2b_{13}b_{06} + h_2h_4h_6^2b_{03}b_{16} + h_2h_4b_{13}b_{57}b_{05}$

$b_{46}b_{57} = h_5^2b_{47} + h_4(1)^2$

$h_5h_0(1, 3, 4)b_{25} = h_0h_5h_2(1, 2)b_{15} + h_3h_3(1)h_0(1, 2, 4)$

$h_0(1, 2, 5)b_{36} = h_0h_2(1, 3)b_{16} + h_5^2h_0(1, 2, 3) + h_0(1, 3, 5)b_{26}$

$h_0(1, 2)b_{37} = h_0h_2(1)b_{17} + h_6h_0(1, 2, 3) + h_0(1, 3)b_{27}$

$h_6b_{24}b_{36}b_{04}b_{15} = h_0h_2(1)h_0(1, 2, 3)b_{14} + h_3^2h_6b_{14}b_{26}b_{05} + h_3^2h_6b_{04}b_{15}b_{26} + h_6h_2(1)^2b_{14}b_{06} + h_6h_2(1)^2b_{04}b_{16} + h_6b_{24}b_{14}b_{36}b_{05}$

$h_1(1, 2)b_{57} = h_1h_3h_5b_{27} + h_5h_1(1)b_{37}$

$h_2h_4h_6h_0(1, 2, 4) = h_0h_2^2h_4^2b_{17} + h_0h_2^2b_{47}b_{15} + h_0h_4^2b_{13}b_{27} + h_0b_{13}b_{25}b_{47}$

$h_2h_4h_6b_{04}b_{15}b_{26} = h_0h_2^2h_0(1, 2, 3)b_{15} + h_0h_6h_0(1, 2)b_{25}b_{16} + h_0h_0(1, 2, 5)b_{25}b_{15} + h_0h_0(1, 2, 4)b_{14}b_{25} + h_2^2h_6h_2(1)b_{15}b_{06} + h_2^2h_6h_2(1)b_{05}b_{16} + h_2h_4h_6b_{14}b_{25}b_{06} + h_2h_4h_6b_{14}b_{26}b_{05} + h_2h_4h_6b_{25}b_{04}b_{16}$

$h_1(1, 2)h_2(1, 2) = 0$

$h_0(1, 3)b_{03}b_{37} = h_0^2h_0(1, 3)b_{17} + h_4h_0(1)b_{04}b_{37} + h_6^2h_0(1, 3)b_{06} + h_6h_0(1, 3, 5)b_{05} + h_6h_0(1, 3, 4)b_{04} + h_0(1, 3)b_{02}b_{27}$

$h_6h_0(1)h_0(1, 2, 3) = h_4h_6^2b_{14}b_{06} + h_4h_6^2b_{04}b_{16} + h_4b_{57}b_{14}b_{05} + h_4b_{57}b_{04}b_{15}$

$h_1(1, 2)h_0(1, 3, 4) = h_1h_3(1)h_0(1, 2, 3)$

$h_6h_2(1)b_{36}b_{04}b_{15} = h_0h_3^2h_0(1, 2, 3)b_{15} + h_0h_0(1, 2, 3)b_{35}b_{14} + h_3^2h_6h_2(1)b_{15}b_{06} + h_3^2h_6h_2(1)b_{05}b_{16} + h_6h_2(1)b_{35}b_{14}b_{06} + h_6h_2(1)b_{35}b_{04}b_{16} + h_6h_2(1)b_{14}b_{36}b_{05}$

$h_2h_4b_{57}b_{04}b_{15} = h_0h_6^2h_0(1, 2)b_{16} + h_0h_6h_0(1, 2, 5)b_{15} + h_0h_6h_0(1, 2, 4)b_{14} + h_2h_4h_6^2b_{14}b_{06} + h_2h_4h_6^2b_{04}b_{16} + h_2h_4b_{57}b_{14}b_{05}$

$h_1(1, 2)h_0(1, 2, 4) = h_1(1, 3)h_0(1, 2, 3)$

$h_4(1)h_2(1, 3) = h_2h_5^2b_{37} + h_2b_{57}b_{36}$

$h_2h_4h_6h_0(1, 2, 3) = h_0h_2^2b_{35}b_{17} + h_0h_2^2b_{15}b_{37} + h_0h_4^2b_{14}b_{27} + h_0b_{13}b_{25}b_{37}$

$h_6h_2(1)b_{04}b_{15}b_{26} = h_0h_0(1, 2, 3)b_{24}b_{15} + h_0h_0(1, 2, 3)b_{14}b_{25} + h_6h_2(1)b_{24}b_{15}b_{06} + h_6h_2(1)b_{24}b_{05}b_{16} + h_6h_2(1)b_{14}b_{25}b_{06} + h_6h_2(1)b_{14}b_{26}b_{05} + h_6h_2(1)b_{25}b_{04}b_{16}$

$h_4(1)h_0(1, 3, 5) = h_5^2h_0(1)b_{37} + h_0(1)b_{57}b_{36}$

$b_{13}b_{57}b_{36} = h_1^2h_5^2b_{27} + h_1^2b_{57}b_{26} + h_5^2b_{13}b_{37} + h_4(1)^2b_{14}$

$h_0(1, 2, 3)b_{46} = h_0h_2(1, 2)b_{16} + h_0(1, 3, 4)b_{26} + h_0(1, 2, 4)b_{36}$

$h_4(1)h_0(1, 2, 5) = h_0h_2h_5^2b_{17} + h_0h_2b_{57}b_{16} + h_5^2h_0(1)b_{27} + h_0(1)b_{57}b_{26}$

$h_1(1, 2)b_{47} = h_1h_3(1)b_{27} + h_1(1, 3)b_{37}$

$h_6h_2(1)h_0(1, 2, 3) = h_0h_3^2b_{15}b_{27} + h_0h_2(1)^2b_{17} + h_0b_{24}b_{15}b_{37} + h_0b_{35}b_{14}b_{27} + h_0b_{14}b_{25}b_{37}$

$h_2h_6b_{35}b_{04}b_{15}b_{26} = h_0h_2h_2(1)h_0(1, 2, 3)b_{15} + h_0h_4h_0(1, 2, 3)b_{14}b_{25} + h_2h_6h_2(1)^2b_{15}b_{06} + h_2h_6h_2(1)^2b_{05}b_{16} + h_2h_6b_{35}b_{14}b_{25}b_{06} + h_2h_6b_{35}b_{14}b_{26}b_{05} + h_2h_6b_{35}b_{25}b_{04}b_{16}$

$b_{24}b_{57}b_{36} = h_3^2h_5^2b_{27} + h_3^2b_{57}b_{26} + h_5^2b_{24}b_{37} + h_2(1, 3)^2$

$h_2(1, 3)h_0(1, 3, 5) = h_0h_3^2h_5^2b_{17} + h_0h_3^2b_{57}b_{16} + h_0h_5^2b_{14}b_{37} + h_0b_{57}b_{14}b_{36}$

$h_0(1, 3)b_{25}b_{37} = h_0h_2(1)b_{35}b_{17} + h_0h_2(1)b_{15}b_{37} + h_6h_0(1, 2, 3)b_{35} + h_0(1, 3)b_{35}b_{27}$

$b_{02}b_{57}b_{14}b_{36} = h_1^2h_3^2h_5^2b_{07} + h_1^2h_3^2b_{57}b_{06} + h_1^2h_5^2b_{04}b_{37} + h_1^2b_{57}b_{36}b_{04} + h_3^2h_5^2b_{02}b_{17} + h_3^2b_{02}b_{57}b_{16} + h_5^2b_{02}b_{14}b_{37} + h_0(1, 3, 5)^2$

$h_4(1)h_2(1, 2) = h_2b_{46}b_{37} + h_2b_{36}b_{47}$

$h_2h_6h_0(1, 2, 3)b_{35} = h_0h_2h_2(1)b_{35}b_{17} + h_0h_2h_2(1)b_{15}b_{37} + h_0h_4b_{35}b_{14}b_{27} + h_0h_4b_{14}b_{25}b_{37}$

$h_2(1, 3)h_0(1, 2, 5) = h_0h_5^2b_{24}b_{17} + h_0h_5^2b_{14}b_{27} + h_0b_{24}b_{57}b_{16} + h_0b_{57}b_{14}b_{26}$

$h_4(1)h_0(1, 3, 4) = h_0(1)b_{46}b_{37} + h_0(1)b_{36}b_{47}$

$b_{02}b_{57}b_{14}b_{26} = h_0^2h_5^2b_{14}b_{17} + h_0^2b_{57}b_{14}b_{16} + h_1^2h_5^2b_{24}b_{07} + h_1^2h_5^2b_{04}b_{27} + h_1^2b_{24}b_{57}b_{06} + h_1^2b_{57}b_{04}b_{26} + h_3^2h_5^2b_{03}b_{17} + h_3^2b_{57}b_{03}b_{16} + h_5^2b_{02}b_{14}b_{27} + h_0(1, 3, 5)h_0(1, 2, 5)$

$b_{13}b_{36}b_{47} = h_1^2b_{46}b_{27} + h_1^2b_{47}b_{26} + h_4^2b_{57}b_{16} + h_4(1)^2b_{15} + b_{13}b_{46}b_{37}$

$h_0(1, 3, 4)b_{46}b_{25} = h_0h_4^2h_2(1, 2)b_{16} + h_0h_2(1, 2)b_{46}b_{15} + h_4^2h_0(1, 3, 4)b_{26} + h_3(1)^2h_0(1, 2, 4)$

$h_4(1)h_0(1, 2, 4) = h_0h_2b_{46}b_{17} + h_0h_2b_{47}b_{16} + h_0(1)b_{46}b_{27} + h_0(1)b_{47}b_{26}$

$b_{57}b_{03}b_{14}b_{26} = h_2^2h_5^2b_{04}b_{17} + h_2^2b_{57}b_{04}b_{16} + h_5^2h_1(1)^2b_{07} + h_5^2b_{13}b_{04}b_{27} + h_5^2b_{24}b_{03}b_{17} + h_5^2b_{03}b_{14}b_{27} + h_1(1)^2b_{57}b_{06} + h_0(1, 2, 5)^2 + b_{13}b_{57}b_{04}b_{26} + b_{24}b_{57}b_{03}b_{16}$

$b_{24}b_{36}b_{47} = h_2^2b_{36}b_{37} + h_3^2b_{46}b_{27} + h_3^2b_{47}b_{26} + h_5^2b_{25}b_{37} + h_2(1, 3)h_2(1, 2)$

$h_2(1, 2)h_0(1, 3, 5) = h_0h_3^2b_{46}b_{17} + h_0h_3^2b_{47}b_{16} + h_0h_5^2b_{15}b_{37} + h_0h_6^2b_{36}b_{16} + h_0b_{57}b_{36}b_{15}$

$h_2(1, 3)h_0(1, 3, 4) = h_0h_3^2b_{46}b_{17} + h_0h_3^2b_{47}b_{16} + h_0h_5^2b_{15}b_{37} + h_0h_6^2b_{36}b_{16} + h_0b_{57}b_{36}b_{15}$

$h_0(1, 3, 4)b_{25}b_{36} = h_0h_2(1, 2)b_{35}b_{16} + h_0h_2(1, 2)b_{36}b_{15} + h_3(1)^2h_0(1, 2, 3) + h_0(1, 3, 4)b_{35}b_{26}$

$h_4(1)h_0(1, 2, 3) = h_0h_2b_{36}b_{17} + h_0h_2b_{37}b_{16} + h_0(1)b_{36}b_{27} + h_0(1)b_{26}b_{37}$

$b_{02}b_{57}b_{36}b_{15} = h_1^2h_3^2b_{46}b_{07} + h_1^2h_3^2b_{47}b_{06} + h_1^2h_5^2b_{37}b_{05} + h_1^2h_6^2b_{36}b_{06} + h_1^2b_{57}b_{36}b_{05} + h_3^2b_{02}b_{46}b_{17} + h_3^2b_{02}b_{47}b_{16} + h_5^2b_{02}b_{15}b_{37} + h_6^2b_{02}b_{36}b_{16} + h_0(1, 3, 5)h_0(1, 3, 4)$

$h_2(1, 2)h_0(1, 2, 5) = h_0h_2^2b_{37}b_{16} + h_0h_5^2b_{25}b_{17} + h_0h_5^2b_{15}b_{27} + h_0h_6^2b_{26}b_{16} + h_0b_{24}b_{47}b_{16} + h_0b_{57}b_{15}b_{26}$

$h_2(1, 3)h_0(1, 2, 4) = h_0h_2^2b_{36}b_{17} + h_0h_5^2b_{25}b_{17} + h_0h_5^2b_{15}b_{27} + h_0h_6^2b_{26}b_{16} + h_0b_{13}b_{36}b_{27} + h_0b_{13}b_{26}b_{37} + h_0b_{24}b_{47}b_{16} + h_0b_{57}b_{15}b_{26}$

$b_{02}b_{57}b_{15}b_{26} = h_0^2h_5^2b_{15}b_{17} + h_0^2h_6^2b_{16}^2 + h_0^2b_{13}b_{36}b_{17} + h_0^2b_{13}b_{37}b_{16} + h_0^2b_{57}b_{15}b_{16} + h_1^2h_5^2b_{25}b_{07} + h_1^2h_5^2b_{05}b_{27} + h_1^2h_6^2b_{26}b_{06} + h_1^2b_{24}b_{47}b_{06} + h_1^2b_{57}b_{26}b_{05} + h_3^2b_{46}b_{03}b_{17} + h_3^2b_{03}b_{47}b_{16} + h_5^2b_{02}b_{15}b_{27} + h_6^2b_{02}b_{26}b_{16} + h_0(1)^2b_{36}b_{27} + h_0(1)^2b_{26}b_{37} + h_0(1, 3, 5)h_0(1, 2, 4)$

$h_0(1, 2, 5)h_0(1, 3, 4) = h_0^2b_{13}b_{36}b_{17} + h_0^2b_{13}b_{37}b_{16} + h_0(1)^2b_{36}b_{27} + h_0(1)^2b_{26}b_{37} + h_0(1, 3, 5)h_0(1, 2, 4)$

$h_6h_0(1, 2, 4)b_{36} = h_0h_2h_4b_{36}b_{17} + h_0h_2h_4b_{37}b_{16} + h_0h_6h_2(1, 2)b_{16} + h_4h_0(1)b_{36}b_{27} + h_4h_0(1)b_{26}b_{37} + h_6h_0(1, 3, 4)b_{26}$

$b_{57}b_{03}b_{15}b_{26} = h_2^2h_5^2b_{15}b_{07} + h_2^2h_5^2b_{05}b_{17} + h_2^2h_6^2b_{16}b_{06} + h_2^2b_{57}b_{05}b_{16} + h_2^2b_{03}b_{37}b_{16} + h_5^2b_{13}b_{25}b_{07} + h_5^2b_{13}b_{05}b_{27} + h_5^2b_{03}b_{25}b_{17} + h_5^2b_{03}b_{15}b_{27} + h_6^2b_{13}b_{26}b_{06} + h_6^2b_{03}b_{26}b_{16} + h_1(1)^2b_{47}b_{06} + h_0(1, 2, 5)h_0(1, 2, 4) + b_{13}b_{57}b_{26}b_{05} + b_{24}b_{03}b_{47}b_{16}$

$h_2(1, 2)b_{57} = h_2h_4(1)b_{37} + h_2(1, 3)b_{47}$

$h_2(1, 3)h_0(1, 2, 3) = h_0h_3^2b_{26}b_{17} + h_0h_3^2b_{16}b_{27} + h_0b_{24}b_{36}b_{17} + h_0b_{24}b_{37}b_{16} + h_0b_{14}b_{36}b_{27} + h_0b_{14}b_{26}b_{37}$

$b_{25}b_{36}b_{47} = h_4^2b_{26}b_{37} + h_3(1)^2b_{27} + h_2(1, 2)^2 + b_{35}b_{47}b_{26} + b_{46}b_{25}b_{37}$

$h_0(1, 3, 4)b_{57} = h_0(1)h_4(1)b_{37} + h_0(1, 3, 5)b_{47}$

$b_{57}b_{36}b_{04}b_{15} = h_0^2h_3^2b_{16}b_{17} + h_1^2h_3^2b_{26}b_{07} + h_1^2h_3^2b_{27}b_{06} + h_1^2b_{24}b_{36}b_{07} + h_1^2b_{24}b_{37}b_{06} + h_3^2h_5^2b_{05}b_{17} + h_3^2b_{02}b_{16}b_{27} + h_3^2b_{46}b_{04}b_{17} + h_3^2b_{03}b_{37}b_{16} + h_5^2b_{14}b_{37}b_{05} + h_5^2b_{04}b_{15}b_{37} + h_6^2b_{14}b_{36}b_{06} + h_6^2b_{36}b_{04}b_{16} + h_0(1, 3, 5)h_0(1, 2, 3) + b_{57}b_{14}b_{36}b_{05}$

$b_{02}b_{14}b_{26}b_{37} = h_0^2h_3^2b_{16}b_{17} + h_0^2b_{14}b_{36}b_{17} + h_0^2b_{14}b_{37}b_{16} + h_1^2h_3^2b_{26}b_{07} + h_1^2h_3^2b_{27}b_{06} + h_1^2b_{24}b_{36}b_{07} + h_1^2b_{24}b_{37}b_{06} + h_1^2b_{36}b_{04}b_{27} + h_1^2b_{04}b_{26}b_{37} + h_3^2h_5^2b_{05}b_{17} + h_3^2b_{02}b_{16}b_{27} + h_3^2b_{46}b_{04}b_{17} + h_3^2b_{03}b_{37}b_{16} + h_0(1, 3, 5)h_0(1, 2, 3) + b_{02}b_{14}b_{36}b_{27}$

$h_2(1, 2)h_0(1, 3, 4) = h_0h_4^2b_{37}b_{16} + h_0h_3(1)^2b_{17} + h_0b_{35}b_{47}b_{16} + h_0b_{46}b_{15}b_{37} + h_0b_{36}b_{47}b_{15}$

$h_0(1, 3)b_{26}b_{37} = h_0h_2(1)b_{36}b_{17} + h_0h_2(1)b_{37}b_{16} + h_6h_0(1, 2, 3)b_{36} + h_0(1, 3)b_{36}b_{27}$

$h_0(1, 2, 4)b_{57} = h_0h_2h_4(1)b_{17} + h_0(1)h_4(1)b_{27} + h_0(1, 2, 5)b_{47}$

$b_{02}b_{36}b_{47}b_{15} = h_1^2h_4^2b_{37}b_{06} + h_1^2h_3(1)^2b_{07} + h_1^2b_{35}b_{47}b_{06} + h_1^2b_{46}b_{37}b_{05} + h_1^2b_{36}b_{47}b_{05} + h_4^2b_{02}b_{37}b_{16} + h_3(1)^2b_{02}b_{17} + h_0(1, 3, 4)^2 + b_{02}b_{35}b_{47}b_{16} + b_{02}b_{46}b_{15}b_{37}$

$b_{57}b_{04}b_{15}b_{26} = h_0^2b_{24}b_{16}b_{17} + h_0^2b_{14}b_{16}b_{27} + h_1^2b_{24}b_{27}b_{06} + h_2^2b_{04}b_{37}b_{16} + h_3^2b_{03}b_{16}b_{27} + h_5^2b_{24}b_{15}b_{07} + h_5^2b_{24}b_{05}b_{17} + h_5^2b_{14}b_{25}b_{07} + h_5^2b_{14}b_{05}b_{27} + h_5^2b_{25}b_{04}b_{17} + h_5^2b_{04}b_{15}b_{27} + h_6^2b_{14}b_{26}b_{06} + h_6^2b_{04}b_{26}b_{16} + h_1(1)^2b_{37}b_{06} + h_0(1, 2, 5)h_0(1, 2, 3) + b_{24}b_{03}b_{37}b_{16} + b_{57}b_{14}b_{26}b_{05}$

$h_2h_6h_0(1, 2, 3)b_{36} = h_0h_2h_2(1)b_{36}b_{17} + h_0h_2h_2(1)b_{37}b_{16} + h_0h_4b_{14}b_{36}b_{27} + h_0h_4b_{14}b_{26}b_{37}$

$h_2(1, 2)h_0(1, 2, 4) = h_0h_4^2b_{26}b_{17} + h_0h_4^2b_{16}b_{27} + h_0b_{46}b_{25}b_{17} + h_0b_{46}b_{15}b_{27} + h_0b_{25}b_{47}b_{16} + h_0b_{47}b_{15}b_{26}$

$h_0(1, 3, 5)h_0(1, 2, 4)b_{14} = h_0^2h_6^2b_{14}b_{16}^2 + h_0^2b_{13}b_{14}b_{36}b_{17} + h_0^2b_{13}b_{14}b_{37}b_{16} + h_1^2h_6^2b_{24}b_{16}b_{06} + h_1^2h_6^2b_{04}b_{26}b_{16} + h_1^2h_0(1, 2, 5)h_0(1, 2, 3) + h_3^2h_6^2b_{03}b_{16}^2 + h_6^2b_{02}b_{14}b_{26}b_{16} + h_0(1)^2b_{14}b_{36}b_{27} + h_0(1)^2b_{14}b_{26}b_{37} + h_0(1, 3, 5)h_0(1, 2, 5)b_{15}$

$b_{02}b_{47}b_{15}b_{26} = h_0^2h_4^2b_{16}b_{17} + h_0^2b_{46}b_{15}b_{17} + h_1^2h_4^2b_{26}b_{07} + h_1^2h_4^2b_{27}b_{06} + h_1^2b_{46}b_{25}b_{07} + h_1^2b_{46}b_{05}b_{27} + h_1^2b_{25}b_{47}b_{06} + h_1^2b_{47}b_{26}b_{05} + h_4^2b_{02}b_{16}b_{27} + h_3(1)^2b_{03}b_{17} + h_0(1, 3, 4)h_0(1, 2, 4) + b_{02}b_{46}b_{15}b_{27} + b_{02}b_{25}b_{47}b_{16}$

$h_0(1, 2, 3)b_{57} = h_0h_2(1, 3)b_{17} + h_0(1, 3, 5)b_{27} + h_0(1, 2, 5)b_{37}$

$h_1(1, 3)b_{36}b_{47} = h_1h_3(1)b_{46}b_{27} + h_1h_3(1)b_{47}b_{26} + h_1(1, 3)b_{46}b_{37}$

$b_{03}b_{47}b_{15}b_{26} = h_2^2h_4^2b_{06}b_{17} + h_2^2b_{46}b_{05}b_{17} + h_2^2b_{47}b_{15}b_{06} + h_2^2b_{47}b_{05}b_{16} + h_4^2b_{13}b_{27}b_{06} + h_4^2b_{03}b_{26}b_{17} + h_4^2b_{03}b_{16}b_{27} + h_1(1, 3)^2b_{07} + h_0(1, 2, 4)^2 + b_{13}b_{46}b_{05}b_{27} + b_{13}b_{25}b_{47}b_{06} + b_{13}b_{47}b_{26}b_{05} + b_{46}b_{03}b_{25}b_{17} + b_{46}b_{03}b_{15}b_{27} + b_{03}b_{25}b_{47}b_{16}$

$h_2(1, 2)h_0(1, 2, 3) = h_0b_{35}b_{26}b_{17} + h_0b_{35}b_{16}b_{27} + h_0b_{25}b_{36}b_{17} + h_0b_{25}b_{37}b_{16} + h_0b_{36}b_{15}b_{27} + h_0b_{15}b_{26}b_{37}$

$h_6h_0(1, 2, 3)b_{24}b_{36} = h_0h_3^2h_2(1)b_{26}b_{17} + h_0h_3^2h_2(1)b_{16}b_{27} + h_0h_2(1)b_{24}b_{36}b_{17} + h_0h_2(1)b_{24}b_{37}b_{16} + h_0h_2(1)b_{14}b_{36}b_{27} + h_0h_2(1)b_{14}b_{26}b_{37} + h_3^2h_6h_0(1, 2, 3)b_{26}$

$b_{02}b_{15}b_{26}b_{37} = h_0^2b_{35}b_{16}b_{17} + h_0^2b_{36}b_{15}b_{17} + h_1^2b_{35}b_{26}b_{07} + h_1^2b_{35}b_{27}b_{06} + h_1^2b_{25}b_{36}b_{07} + h_1^2b_{25}b_{37}b_{06} + h_1^2b_{36}b_{05}b_{27} + h_1^2b_{26}b_{37}b_{05} + h_3^2b_{46}b_{05}b_{17} + h_3(1)^2b_{04}b_{17} + h_0(1, 3, 4)h_0(1, 2, 3) + b_{02}b_{35}b_{16}b_{27} + b_{02}b_{25}b_{37}b_{16} + b_{02}b_{36}b_{15}b_{27}$

$b_{03}b_{15}b_{26}b_{37} = h_0^2b_{25}b_{16}b_{17} + h_0^2b_{15}b_{26}b_{17} + h_1^2b_{25}b_{27}b_{06} + h_2^2b_{35}b_{06}b_{17} + h_2^2b_{36}b_{05}b_{17} + h_2^2b_{15}b_{37}b_{06} + h_2^2b_{37}b_{05}b_{16} + h_4^2b_{14}b_{27}b_{06} + h_4^2b_{04}b_{26}b_{17} + h_4^2b_{04}b_{16}b_{27} + h_5^2b_{25}b_{05}b_{17} + h_5^2b_{15}b_{05}b_{27} + h_1(1, 3)h_1(1, 2)b_{07} + h_0(1, 2, 4)h_0(1, 2, 3) + b_{02}b_{25}b_{16}b_{27} + b_{02}b_{15}b_{26}b_{27} + b_{13}b_{25}b_{37}b_{06} + b_{13}b_{36}b_{05}b_{27} + b_{13}b_{26}b_{37}b_{05} + b_{46}b_{25}b_{04}b_{17} + b_{46}b_{04}b_{15}b_{27} + b_{03}b_{25}b_{37}b_{16}$

$h_0(1, 2, 3)b_{47} = h_0h_2(1, 2)b_{17} + h_0(1, 3, 4)b_{27} + h_0(1, 2, 4)b_{37}$

$b_{04}b_{15}b_{26}b_{37} = h_3^2b_{15}b_{27}b_{06} + h_3^2b_{26}b_{05}b_{17} + h_3^2b_{05}b_{16}b_{27} + h_2(1)^2b_{06}b_{17} + h_1(1, 2)^2b_{07} + h_0(1, 2, 3)^2 + b_{24}b_{36}b_{05}b_{17} + b_{24}b_{15}b_{37}b_{06} + b_{24}b_{37}b_{05}b_{16} + b_{35}b_{14}b_{27}b_{06} + b_{35}b_{04}b_{26}b_{17} + b_{35}b_{04}b_{16}b_{27} + b_{14}b_{25}b_{37}b_{06} + b_{14}b_{36}b_{05}b_{27} + b_{14}b_{26}b_{37}b_{05} + b_{25}b_{36}b_{04}b_{17} + b_{25}b_{04}b_{37}b_{16} + b_{36}b_{04}b_{15}b_{27}$

$h_6h_0(1, 2, 3)b_{25}b_{36} = h_0h_2(1)b_{35}b_{26}b_{17} + h_0h_2(1)b_{35}b_{16}b_{27} + h_0h_2(1)b_{25}b_{36}b_{17} + h_0h_2(1)b_{25}b_{37}b_{16} + h_0h_2(1)b_{36}b_{15}b_{27} + h_0h_2(1)b_{15}b_{26}b_{37} + h_6h_0(1, 2, 3)b_{35}b_{26}$

$h_0(1, 3, 5)h_0(1, 2, 4)b_{36} = h_0^2h_3^2b_{46}b_{16}b_{17} + h_0^2h_3^2b_{47}b_{16}^2 + h_0^2h_5^2b_{15}b_{37}b_{16} + h_0^2h_6^2b_{36}b_{16}^2 + h_0^2b_{13}b_{36}^2b_{17} + h_0^2b_{13}b_{36}b_{37}b_{16} + h_0^2b_{57}b_{36}b_{15}b_{16} + h_5^2h_0(1, 3, 4)h_0(1, 2, 3) + h_0(1)^2b_{36}^2b_{27} + h_0(1)^2b_{36}b_{26}b_{37} + h_0(1, 3, 5)h_0(1, 3, 4)b_{26}$

$h_0(1, 2, 4)b_{36}b_{47} = h_0h_2(1, 2)b_{46}b_{17} + h_0h_2(1, 2)b_{47}b_{16} + h_0(1, 3, 4)b_{46}b_{27} + h_0(1, 3, 4)b_{47}b_{26} + h_0(1, 2, 4)b_{46}b_{37}$

\subsection{Relations of \mybm{$HX_7$}{HX7} organized by patterns}\label{app:35b02630}
This section is coordinating with Conjecture \ref{conj:5f7d758} and Theorem \ref{thm:bc7c165}.
\subsubsection*{Relations \ref{rel:8d29c79f}}\hspace{6pt}

$h_0^2b_{14} + h_3^2b_{03} + b_{02}b_{24}=0$

$h_0^2b_{15} + h_4^2b_{04} + b_{02}b_{25} + b_{35}b_{03}=0$

$h_0^2b_{16} + h_5^2b_{05} + b_{02}b_{26} + b_{46}b_{04} + b_{03}b_{36}=0$

$h_0^2b_{17} + h_6^2b_{06} + b_{02}b_{27} + b_{57}b_{05} + b_{03}b_{37} + b_{47}b_{04}=0$

$h_0^2b_{18} + h_7^2b_{07} + b_{02}b_{28} + b_{68}b_{06} + b_{03}b_{38} + b_{58}b_{05} + b_{04}b_{48}=0$

\subsubsection*{Relations \ref{rel:2762c538}}\hspace{6pt}

$h_0h_1=0$

$h_3h_0(1)=0$

$h_0h_1(1)=0$

$h_0(1)h_1(1)=0$

$h_5h_0(1, 3)=0$

$h_0(1)h_3(1)=0$

$h_5h_0(1, 2)=0$

$h_0h_1(1, 3)=0$

$h_0(1)h_1(1, 3)=0$

$h_0h_1(1, 2)=0$

$h_0(1)h_1(1, 2)=0$

$h_3(1)h_0(1, 3)=0$

$h_3(1)h_0(1, 2)=0$

$h_0(1, 3)h_1(1, 3)=0$

$h_1(1, 3)h_0(1, 2)=0$

$h_0(1, 3)h_1(1, 2)=0$

$h_0(1, 2)h_1(1, 2)=0$

$h_7h_0(1, 3, 5)=0$

$h_7h_0(1, 2, 5)=0$

$h_5(1)h_0(1, 3)=0$

$h_7h_0(1, 3, 4)=0$

$h_0(1)h_3(1, 3)=0$

$h_5(1)h_0(1, 2)=0$

$h_7h_0(1, 2, 4)=0$

$h_0h_1(1, 3, 5)=0$

$h_0(1)h_1(1, 3, 5)=0$

$h_7h_0(1, 2, 3)=0$

$h_0h_1(1, 2, 5)=0$

$h_0(1)h_1(1, 2, 5)=0$

$h_0(1, 3)h_3(1, 3)=0$

$h_3(1, 3)h_0(1, 2)=0$

$h_0(1, 3)h_1(1, 3, 5)=0$

$h_0(1, 2)h_1(1, 3, 5)=0$

$h_0(1, 3)h_1(1, 2, 5)=0$

$h_0(1)h_3(1, 2)=0$

$h_0h_1(1, 3, 4)=0$

$h_0(1, 2)h_1(1, 2, 5)=0$

$h_0(1)h_1(1, 3, 4)=0$

$h_0h_1(1, 2, 4)=0$

$h_0(1)h_1(1, 2, 4)=0$

$h_0(1, 3)h_3(1, 2)=0$

$h_0(1, 2)h_3(1, 2)=0$

$h_0(1, 3)h_1(1, 3, 4)=0$

$h_0h_1(1, 2, 3)=0$

$h_0(1, 2)h_1(1, 3, 4)=0$

$h_5(1)h_0(1, 3, 5)=0$

$h_0(1)h_1(1, 2, 3)=0$

$h_0(1, 3)h_1(1, 2, 4)=0$

$h_5(1)h_0(1, 2, 5)=0$

$h_0(1, 2)h_1(1, 2, 4)=0$

$h_5(1)h_0(1, 3, 4)=0$

$h_0(1, 3)h_1(1, 2, 3)=0$

$h_3(1, 3)h_0(1, 3, 5)=0$

$h_5(1)h_0(1, 2, 4)=0$

$h_0(1, 2)h_1(1, 2, 3)=0$

$h_3(1, 3)h_0(1, 2, 5)=0$

$h_0(1, 3, 5)h_1(1, 3, 5)=0$

$h_5(1)h_0(1, 2, 3)=0$

$h_1(1, 3, 5)h_0(1, 2, 5)=0$

$h_0(1, 3, 5)h_1(1, 2, 5)=0$

$h_3(1, 3)h_0(1, 3, 4)=0$

$h_0(1, 2, 5)h_1(1, 2, 5)=0$

$h_3(1, 3)h_0(1, 2, 4)=0$

$h_1(1, 3, 5)h_0(1, 3, 4)=0$

$h_1(1, 3, 5)h_0(1, 2, 4)=0$

$h_3(1, 3)h_0(1, 2, 3)=0$

$h_1(1, 2, 5)h_0(1, 3, 4)=0$

$h_3(1, 2)h_0(1, 3, 5)=0$

$h_1(1, 2, 5)h_0(1, 2, 4)=0$

$h_1(1, 3, 5)h_0(1, 2, 3)=0$

$h_3(1, 2)h_0(1, 2, 5)=0$

$h_0(1, 3, 5)h_1(1, 3, 4)=0$

$h_0(1, 2, 5)h_1(1, 3, 4)=0$

$h_1(1, 2, 5)h_0(1, 2, 3)=0$

$h_0(1, 3, 5)h_1(1, 2, 4)=0$

$h_3(1, 2)h_0(1, 3, 4)=0$

$h_0(1, 2, 5)h_1(1, 2, 4)=0$

$h_3(1, 2)h_0(1, 2, 4)=0$

$h_0(1, 3, 4)h_1(1, 3, 4)=0$

$h_1(1, 3, 4)h_0(1, 2, 4)=0$

$h_3(1, 2)h_0(1, 2, 3)=0$

$h_0(1, 3, 5)h_1(1, 2, 3)=0$

$h_0(1, 3, 4)h_1(1, 2, 4)=0$

$h_0(1, 2, 5)h_1(1, 2, 3)=0$

$h_0(1, 2, 4)h_1(1, 2, 4)=0$

$h_1(1, 3, 4)h_0(1, 2, 3)=0$

$h_1(1, 2, 4)h_0(1, 2, 3)=0$

$h_0(1, 3, 4)h_1(1, 2, 3)=0$

$h_0(1, 2, 4)h_1(1, 2, 3)=0$

$h_0(1, 2, 3)h_1(1, 2, 3)=0$

\subsubsection*{Relations \ref{rel:b0a9481d}}\hspace{6pt}

$h_0h_2b_{14} + h_0(1)b_{24}=0$

$h_0h_2(1)b_{14} + h_3^2h_0(1, 2) + h_0(1, 3)b_{24}=0$

$h_0h_2(1)b_{15} + h_0(1, 3)b_{25} + h_0(1, 2)b_{35}=0$

$h_4h_0(1)b_{36} + h_0(1, 3)b_{46}=0$

$h_0h_2h_4b_{16} + h_4h_0(1)b_{26} + h_0(1, 2)b_{46}=0$

$h_0h_2(1)b_{16} + h_0(1, 3)b_{26} + h_0(1, 2)b_{36}=0$

$h_0h_2(1, 3)b_{14} + h_3^2h_0(1, 2, 5) + h_0(1, 3, 5)b_{24}=0$

$h_0h_2(1, 2)b_{14} + h_3^2h_0(1, 2, 4) + h_0(1, 3, 4)b_{24}=0$

$h_0h_2(1, 3)b_{15} + h_0(1, 3, 5)b_{25} + h_0(1, 2, 5)b_{35}=0$

$h_0h_2(1, 2)b_{15} + h_4^2h_0(1, 2, 3) + h_0(1, 3, 4)b_{25} + h_0(1, 2, 4)b_{35}=0$

$h_5^2h_0(1, 3, 4) + h_0(1)h_4(1)b_{36} + h_0(1, 3, 5)b_{46}=0$

$h_0h_2h_4(1)b_{16} + h_5^2h_0(1, 2, 4) + h_0(1)h_4(1)b_{26} + h_0(1, 2, 5)b_{46}=0$

$h_0h_2(1, 3)b_{16} + h_5^2h_0(1, 2, 3) + h_0(1, 3, 5)b_{26} + h_0(1, 2, 5)b_{36}=0$

$h_0h_2(1, 2)b_{16} + h_0(1, 3, 4)b_{26} + h_0(1, 2, 4)b_{36} + h_0(1, 2, 3)b_{46}=0$

$h_0(1)h_4(1)b_{37} + h_0(1, 3, 5)b_{47} + h_0(1, 3, 4)b_{57}=0$

$h_0h_2h_4(1)b_{17} + h_0(1)h_4(1)b_{27} + h_0(1, 2, 5)b_{47} + h_0(1, 2, 4)b_{57}=0$

$h_0h_2(1, 3)b_{17} + h_0(1, 3, 5)b_{27} + h_0(1, 2, 5)b_{37} + h_0(1, 2, 3)b_{57}=0$

$h_0h_2(1, 2)b_{17} + h_0(1, 3, 4)b_{27} + h_0(1, 2, 4)b_{37} + h_0(1, 2, 3)b_{47}=0$

$h_6h_0(1, 3)b_{58} + h_0(1, 3, 5)b_{68}=0$

$h_6h_0(1, 2)b_{58} + h_0(1, 2, 5)b_{68}=0$

$h_4h_6h_0(1)b_{38} + h_6h_0(1, 3)b_{48} + h_0(1, 3, 4)b_{68}=0$

$h_0h_2h_4h_6b_{18} + h_4h_6h_0(1)b_{28} + h_6h_0(1, 2)b_{48} + h_0(1, 2, 4)b_{68}=0$

$h_0h_6h_2(1)b_{18} + h_6h_0(1, 3)b_{28} + h_6h_0(1, 2)b_{38} + h_0(1, 2, 3)b_{68}=0$

$h_0(1)h_4(1)b_{38} + h_0(1, 3, 5)b_{48} + h_0(1, 3, 4)b_{58}=0$

$h_0h_2h_4(1)b_{18} + h_0(1)h_4(1)b_{28} + h_0(1, 2, 5)b_{48} + h_0(1, 2, 4)b_{58}=0$

$h_0h_2(1, 3)b_{18} + h_0(1, 3, 5)b_{28} + h_0(1, 2, 5)b_{38} + h_0(1, 2, 3)b_{58}=0$

$h_0h_2(1, 2)b_{18} + h_0(1, 3, 4)b_{28} + h_0(1, 2, 4)b_{38} + h_0(1, 2, 3)b_{48}=0$

\subsubsection*{Relations \ref{rel:c4698d4e}}\hspace{6pt}

$h_1h_3b_{03} + h_1(1)b_{02}=0$

$h_1^2h_0(1, 2) + h_4h_0(1)b_{14} + h_0(1, 3)b_{13}=0$

$h_4h_0(1)b_{04} + h_0(1, 3)b_{03} + h_0(1, 2)b_{02}=0$

$h_1h_3(1)b_{03} + h_1(1, 3)b_{02}=0$

$h_1h_3h_5b_{05} + h_1h_3(1)b_{04} + h_1(1, 2)b_{02}=0$

$h_5h_1(1)b_{05} + h_1(1, 3)b_{04} + h_1(1, 2)b_{03}=0$

$h_1^2h_0(1, 2, 5) + h_0(1)h_4(1)b_{14} + h_0(1, 3, 5)b_{13}=0$

$h_0(1)h_4(1)b_{04} + h_0(1, 3, 5)b_{03} + h_0(1, 2, 5)b_{02}=0$

$h_1^2h_0(1, 2, 4) + h_4h_6h_0(1)b_{16} + h_0(1)h_4(1)b_{15} + h_0(1, 3, 4)b_{13}=0$

$h_4h_6h_0(1)b_{06} + h_0(1)h_4(1)b_{05} + h_0(1, 3, 4)b_{03} + h_0(1, 2, 4)b_{02}=0$

$h_3^2h_0(1, 3, 4) + h_6h_0(1, 3)b_{36} + h_0(1, 3, 5)b_{35}=0$

$h_6h_0(1, 3)b_{26} + h_0(1, 3, 5)b_{25} + h_0(1, 3, 4)b_{24}=0$

$h_1^2h_0(1, 2, 3) + h_6h_0(1, 3)b_{16} + h_0(1, 3, 5)b_{15} + h_0(1, 3, 4)b_{14}=0$

$h_6h_0(1, 3)b_{06} + h_0(1, 3, 5)b_{05} + h_0(1, 3, 4)b_{04} + h_0(1, 2, 3)b_{02}=0$

$h_2^2h_0(1, 2, 3) + h_6h_0(1, 2)b_{26} + h_0(1, 2, 5)b_{25} + h_0(1, 2, 4)b_{24}=0$

$h_6h_0(1, 2)b_{16} + h_0(1, 2, 5)b_{15} + h_0(1, 2, 4)b_{14} + h_0(1, 2, 3)b_{13}=0$

$h_6h_0(1, 2)b_{06} + h_0(1, 2, 5)b_{05} + h_0(1, 2, 4)b_{04} + h_0(1, 2, 3)b_{03}=0$

$h_1h_3(1, 3)b_{03} + h_1(1, 3, 5)b_{02}=0$

$h_1h_3h_5(1)b_{05} + h_1h_3(1, 3)b_{04} + h_1(1, 2, 5)b_{02}=0$

$h_1(1)h_5(1)b_{05} + h_1(1, 3, 5)b_{04} + h_1(1, 2, 5)b_{03}=0$

$h_1h_3(1, 2)b_{03} + h_1(1, 3, 4)b_{02}=0$

$h_1h_3h_5h_7b_{07} + h_1h_3h_5(1)b_{06} + h_1h_3(1, 2)b_{04} + h_1(1, 2, 4)b_{02}=0$

$h_5h_7h_1(1)b_{07} + h_1(1)h_5(1)b_{06} + h_1(1, 3, 4)b_{04} + h_1(1, 2, 4)b_{03}=0$

$h_1h_7h_3(1)b_{07} + h_1h_3(1, 3)b_{06} + h_1h_3(1, 2)b_{05} + h_1(1, 2, 3)b_{02}=0$

$h_7h_1(1, 3)b_{07} + h_1(1, 3, 5)b_{06} + h_1(1, 3, 4)b_{05} + h_1(1, 2, 3)b_{03}=0$

$h_7h_1(1, 2)b_{07} + h_1(1, 2, 5)b_{06} + h_1(1, 2, 4)b_{05} + h_1(1, 2, 3)b_{04}=0$

\subsubsection*{Relations \ref{rel:a078c1d9}}\hspace{6pt}

$h_2h_0(1, 3) + h_0(1)h_2(1)=0$

$h_2h_0(1, 3, 5) + h_0(1)h_2(1, 3)=0$

$h_4h_0(1, 3, 5) + h_4(1)h_0(1, 3)=0$

$h_4h_0(1, 2, 5) + h_4(1)h_0(1, 2)=0$

$h_2h_0(1, 3, 4) + h_0(1)h_2(1, 2)=0$

$h_2(1)h_0(1, 3, 5) + h_0(1, 3)h_2(1, 3)=0$

$h_2(1)h_0(1, 2, 5) + h_2(1, 3)h_0(1, 2)=0$

$h_2(1)h_0(1, 3, 4) + h_0(1, 3)h_2(1, 2)=0$

$h_2h_4h_0(1, 2, 3) + h_2(1)h_0(1, 2, 4) + h_0(1, 2)h_2(1, 2)=0$

$h_2(1, 3)h_0(1, 3, 4) + h_2(1, 2)h_0(1, 3, 5)=0$

$h_2h_4(1)h_0(1, 2, 3) + h_2(1, 3)h_0(1, 2, 4) + h_2(1, 2)h_0(1, 2, 5)=0$

\subsubsection*{Relations \ref{rel:51b194a6}}\hspace{6pt}

$h_0(1, 3)h_0(1, 2, 5) + h_0(1, 2)h_0(1, 3, 5)=0$

$h_4h_0(1)h_0(1, 2, 3) + h_0(1, 3)h_0(1, 2, 4) + h_0(1, 2)h_0(1, 3, 4)=0$

\subsubsection*{Relations \ref{rel:e7460c84}}\hspace{6pt}

$h_0h_0(1) + h_2b_{02}=0$

$h_0b_{13} + h_2h_0(1)=0$

$h_1^2b_{03} + h_0(1)^2 + b_{02}b_{13}=0$

$h_0h_0(1, 3) + h_2(1)b_{02}=0$

$h_0h_4b_{14} + h_0(1)h_2(1)=0$

$h_0h_0(1, 2) + h_2h_4b_{04} + h_2(1)b_{03}=0$

$h_1^2h_4b_{04} + h_4b_{02}b_{14} + h_0(1)h_0(1, 3)=0$

$h_4h_0(1, 3) + h_0(1)b_{35}=0$

$h_4b_{13}b_{04} + h_4b_{03}b_{14} + h_0(1)h_0(1, 2)=0$

$h_0h_2b_{15} + h_4h_0(1, 2) + h_0(1)b_{25}=0$

$h_0h_3^2b_{15} + h_0b_{35}b_{14} + h_2(1)h_0(1, 3)=0$

$h_1^2h_3^2b_{05} + h_1^2b_{35}b_{04} + h_3^2b_{02}b_{15} + h_0(1, 3)^2 + b_{02}b_{35}b_{14}=0$

$h_0b_{24}b_{15} + h_0b_{14}b_{25} + h_2(1)h_0(1, 2)=0$

$h_3^2b_{13}b_{05} + h_3^2b_{03}b_{15} + h_0(1, 3)h_0(1, 2) + b_{13}b_{35}b_{04} + b_{35}b_{03}b_{14}=0$

$h_2^2b_{14}b_{05} + h_2^2b_{04}b_{15} + h_0(1, 2)^2 + b_{13}b_{24}b_{05} + b_{13}b_{25}b_{04} + b_{24}b_{03}b_{15} + b_{03}b_{14}b_{25}=0$

$h_0h_0(1, 3, 5) + h_2(1, 3)b_{02}=0$

$h_0h_4(1)b_{14} + h_0(1)h_2(1, 3)=0$

$h_0h_0(1, 2, 5) + h_2h_4(1)b_{04} + h_2(1, 3)b_{03}=0$

$h_1^2h_4(1)b_{04} + h_0(1)h_0(1, 3, 5) + h_4(1)b_{02}b_{14}=0$

$h_6h_0(1)b_{36} + h_4(1)h_0(1, 3)=0$

$h_0(1)h_0(1, 2, 5) + h_4(1)b_{13}b_{04} + h_4(1)b_{03}b_{14}=0$

$h_0h_2h_6b_{16} + h_6h_0(1)b_{26} + h_4(1)h_0(1, 2)=0$

$h_0h_0(1, 3, 4) + h_2(1, 2)b_{02}=0$

$h_0h_4h_6b_{16} + h_0h_4(1)b_{15} + h_0(1)h_2(1, 2)=0$

$h_0h_3^2h_6b_{16} + h_0h_6b_{14}b_{36} + h_0(1, 3)h_2(1, 3)=0$

$h_0h_0(1, 2, 4) + h_2h_4h_6b_{06} + h_2h_4(1)b_{05} + h_2(1, 2)b_{03}=0$

$h_1^2h_4h_6b_{06} + h_1^2h_4(1)b_{05} + h_4h_6b_{02}b_{16} + h_0(1)h_0(1, 3, 4) + h_4(1)b_{02}b_{15}=0$

$h_1^2h_3^2h_6b_{06} + h_1^2h_6b_{36}b_{04} + h_3^2h_6b_{02}b_{16} + h_6b_{02}b_{14}b_{36} + h_0(1, 3)h_0(1, 3, 5)=0$

$h_0h_6b_{24}b_{16} + h_0h_6b_{14}b_{26} + h_2(1, 3)h_0(1, 2)=0$

$h_4h_6b_{13}b_{06} + h_4h_6b_{03}b_{16} + h_0(1)h_0(1, 2, 4) + h_4(1)b_{13}b_{05} + h_4(1)b_{03}b_{15}=0$

$h_3^2h_6b_{13}b_{06} + h_3^2h_6b_{03}b_{16} + h_6b_{13}b_{36}b_{04} + h_6b_{03}b_{14}b_{36} + h_0(1, 3)h_0(1, 2, 5)=0$

$h_0h_0(1, 2, 3) + h_6h_2(1)b_{06} + h_2(1, 3)b_{05} + h_2(1, 2)b_{04}=0$

$h_2^2h_6b_{14}b_{06} + h_2^2h_6b_{04}b_{16} + h_6b_{13}b_{24}b_{06} + h_6b_{13}b_{04}b_{26} + h_6b_{24}b_{03}b_{16} + h_6b_{03}b_{14}b_{26} + h_0(1, 2)h_0(1, 2, 5)=0$

$h_6h_0(1, 3, 5) + h_0(1, 3)b_{57}=0$

$h_4h_6b_{14}b_{06} + h_4h_6b_{04}b_{16} + h_0(1)h_0(1, 2, 3) + h_4(1)b_{14}b_{05} + h_4(1)b_{04}b_{15}=0$

$h_0h_6b_{35}b_{16} + h_0h_6b_{36}b_{15} + h_0(1, 3)h_2(1, 2)=0$

$h_6h_0(1, 2, 5) + h_0(1, 2)b_{57}=0$

$h_1^2h_6b_{35}b_{06} + h_1^2h_6b_{36}b_{05} + h_6b_{02}b_{35}b_{16} + h_6b_{02}b_{36}b_{15} + h_0(1, 3)h_0(1, 3, 4)=0$

$h_0h_6b_{25}b_{16} + h_0h_6b_{15}b_{26} + h_0(1, 2)h_2(1, 2)=0$

$h_6b_{13}b_{35}b_{06} + h_6b_{13}b_{36}b_{05} + h_6b_{35}b_{03}b_{16} + h_6b_{03}b_{36}b_{15} + h_0(1, 3)h_0(1, 2, 4)=0$

$h_2^2h_6b_{15}b_{06} + h_2^2h_6b_{05}b_{16} + h_6b_{13}b_{25}b_{06} + h_6b_{13}b_{26}b_{05} + h_6b_{03}b_{25}b_{16} + h_6b_{03}b_{15}b_{26} + h_0(1, 2)h_0(1, 2, 4)=0$

$h_4h_0(1)b_{37} + h_6h_0(1, 3, 4) + h_0(1, 3)b_{47}=0$

$h_3^2h_6b_{15}b_{06} + h_3^2h_6b_{05}b_{16} + h_6b_{35}b_{14}b_{06} + h_6b_{35}b_{04}b_{16} + h_6b_{14}b_{36}b_{05} + h_6b_{36}b_{04}b_{15} + h_0(1, 3)h_0(1, 2, 3)=0$

$h_0h_2h_4b_{17} + h_4h_0(1)b_{27} + h_6h_0(1, 2, 4) + h_0(1, 2)b_{47}=0$

$h_6b_{24}b_{15}b_{06} + h_6b_{24}b_{05}b_{16} + h_6b_{14}b_{25}b_{06} + h_6b_{14}b_{26}b_{05} + h_6b_{25}b_{04}b_{16} + h_6b_{04}b_{15}b_{26} + h_0(1, 2)h_0(1, 2, 3)=0$

$h_0h_2(1)b_{17} + h_6h_0(1, 2, 3) + h_0(1, 3)b_{27} + h_0(1, 2)b_{37}=0$

$h_5^2h_0(1)b_{37} + h_0(1)b_{57}b_{36} + h_4(1)h_0(1, 3, 5)=0$

$h_0h_2h_5^2b_{17} + h_0h_2b_{57}b_{16} + h_5^2h_0(1)b_{27} + h_0(1)b_{57}b_{26} + h_4(1)h_0(1, 2, 5)=0$

$h_0h_3^2h_5^2b_{17} + h_0h_3^2b_{57}b_{16} + h_0h_5^2b_{14}b_{37} + h_0b_{57}b_{14}b_{36} + h_2(1, 3)h_0(1, 3, 5)=0$

$h_1^2h_3^2h_5^2b_{07} + h_1^2h_3^2b_{57}b_{06} + h_1^2h_5^2b_{04}b_{37} + h_1^2b_{57}b_{36}b_{04} + h_3^2h_5^2b_{02}b_{17} + h_3^2b_{02}b_{57}b_{16} + h_5^2b_{02}b_{14}b_{37} + h_0(1, 3, 5)^2 + b_{02}b_{57}b_{14}b_{36}=0$

$h_0h_5^2b_{24}b_{17} + h_0h_5^2b_{14}b_{27} + h_0b_{24}b_{57}b_{16} + h_0b_{57}b_{14}b_{26} + h_2(1, 3)h_0(1, 2, 5)=0$

$h_0(1)b_{46}b_{37} + h_0(1)b_{36}b_{47} + h_4(1)h_0(1, 3, 4)=0$

$h_3^2h_5^2b_{13}b_{07} + h_3^2h_5^2b_{03}b_{17} + h_3^2b_{13}b_{57}b_{06} + h_3^2b_{57}b_{03}b_{16} + h_5^2b_{13}b_{04}b_{37} + h_5^2b_{03}b_{14}b_{37} + h_0(1, 3, 5)h_0(1, 2, 5) + b_{13}b_{57}b_{36}b_{04} + b_{57}b_{03}b_{14}b_{36}=0$

$h_0h_2b_{46}b_{17} + h_0h_2b_{47}b_{16} + h_0(1)b_{46}b_{27} + h_0(1)b_{47}b_{26} + h_4(1)h_0(1, 2, 4)=0$

$h_2^2h_5^2b_{14}b_{07} + h_2^2h_5^2b_{04}b_{17} + h_2^2b_{57}b_{14}b_{06} + h_2^2b_{57}b_{04}b_{16} + h_5^2b_{13}b_{24}b_{07} + h_5^2b_{13}b_{04}b_{27} + h_5^2b_{24}b_{03}b_{17} + h_5^2b_{03}b_{14}b_{27} + h_0(1, 2, 5)^2 + b_{13}b_{24}b_{57}b_{06} + b_{13}b_{57}b_{04}b_{26} + b_{24}b_{57}b_{03}b_{16} + b_{57}b_{03}b_{14}b_{26}=0$

$h_0h_5^2b_{35}b_{17} + h_0h_5^2b_{15}b_{37} + h_0b_{35}b_{57}b_{16} + h_0b_{57}b_{36}b_{15} + h_2(1, 2)h_0(1, 3, 5)=0$

$h_0h_2b_{36}b_{17} + h_0h_2b_{37}b_{16} + h_0(1)b_{36}b_{27} + h_0(1)b_{26}b_{37} + h_4(1)h_0(1, 2, 3)=0$

$h_1^2h_5^2b_{35}b_{07} + h_1^2h_5^2b_{37}b_{05} + h_1^2b_{35}b_{57}b_{06} + h_1^2b_{57}b_{36}b_{05} + h_5^2b_{02}b_{35}b_{17} + h_5^2b_{02}b_{15}b_{37} + h_0(1, 3, 5)h_0(1, 3, 4) + b_{02}b_{35}b_{57}b_{16} + b_{02}b_{57}b_{36}b_{15}=0$

$h_0h_5^2b_{25}b_{17} + h_0h_5^2b_{15}b_{27} + h_0b_{57}b_{25}b_{16} + h_0b_{57}b_{15}b_{26} + h_2(1, 2)h_0(1, 2, 5)=0$

$h_5^2b_{13}b_{35}b_{07} + h_5^2b_{13}b_{37}b_{05} + h_5^2b_{35}b_{03}b_{17} + h_5^2b_{03}b_{15}b_{37} + h_0(1, 3, 5)h_0(1, 2, 4) + b_{13}b_{35}b_{57}b_{06} + b_{13}b_{57}b_{36}b_{05} + b_{35}b_{57}b_{03}b_{16} + b_{57}b_{03}b_{36}b_{15}=0$

$h_3^2b_{13}b_{46}b_{07} + h_3^2b_{13}b_{47}b_{06} + h_3^2b_{46}b_{03}b_{17} + h_3^2b_{03}b_{47}b_{16} + h_0(1, 2, 5)h_0(1, 3, 4) + b_{13}b_{46}b_{04}b_{37} + b_{13}b_{36}b_{47}b_{04} + b_{46}b_{03}b_{14}b_{37} + b_{03}b_{14}b_{36}b_{47}=0$

$h_2^2h_5^2b_{15}b_{07} + h_2^2h_5^2b_{05}b_{17} + h_2^2b_{57}b_{15}b_{06} + h_2^2b_{57}b_{05}b_{16} + h_5^2b_{13}b_{25}b_{07} + h_5^2b_{13}b_{05}b_{27} + h_5^2b_{03}b_{25}b_{17} + h_5^2b_{03}b_{15}b_{27} + h_0(1, 2, 5)h_0(1, 2, 4) + b_{13}b_{57}b_{25}b_{06} + b_{13}b_{57}b_{26}b_{05} + b_{57}b_{03}b_{25}b_{16} + b_{57}b_{03}b_{15}b_{26}=0$

$h_0h_3^2b_{26}b_{17} + h_0h_3^2b_{16}b_{27} + h_0b_{24}b_{36}b_{17} + h_0b_{24}b_{37}b_{16} + h_0b_{14}b_{36}b_{27} + h_0b_{14}b_{26}b_{37} + h_2(1, 3)h_0(1, 2, 3)=0$

$h_3^2h_5^2b_{15}b_{07} + h_3^2h_5^2b_{05}b_{17} + h_3^2b_{57}b_{15}b_{06} + h_3^2b_{57}b_{05}b_{16} + h_5^2b_{35}b_{14}b_{07} + h_5^2b_{35}b_{04}b_{17} + h_5^2b_{14}b_{37}b_{05} + h_5^2b_{04}b_{15}b_{37} + h_0(1, 3, 5)h_0(1, 2, 3) + b_{35}b_{57}b_{14}b_{06} + b_{35}b_{57}b_{04}b_{16} + b_{57}b_{14}b_{36}b_{05} + b_{57}b_{36}b_{04}b_{15}=0$

$h_0h_4^2b_{36}b_{17} + h_0h_4^2b_{37}b_{16} + h_0b_{35}b_{46}b_{17} + h_0b_{35}b_{47}b_{16} + h_0b_{46}b_{15}b_{37} + h_0b_{36}b_{47}b_{15} + h_2(1, 2)h_0(1, 3, 4)=0$

$h_1^2h_4^2b_{36}b_{07} + h_1^2h_4^2b_{37}b_{06} + h_1^2b_{35}b_{46}b_{07} + h_1^2b_{35}b_{47}b_{06} + h_1^2b_{46}b_{37}b_{05} + h_1^2b_{36}b_{47}b_{05} + h_4^2b_{02}b_{36}b_{17} + h_4^2b_{02}b_{37}b_{16} + h_0(1, 3, 4)^2 + b_{02}b_{35}b_{46}b_{17} + b_{02}b_{35}b_{47}b_{16} + b_{02}b_{46}b_{15}b_{37} + b_{02}b_{36}b_{47}b_{15}=0$

$h_5^2b_{24}b_{15}b_{07} + h_5^2b_{24}b_{05}b_{17} + h_5^2b_{14}b_{25}b_{07} + h_5^2b_{14}b_{05}b_{27} + h_5^2b_{25}b_{04}b_{17} + h_5^2b_{04}b_{15}b_{27} + h_0(1, 2, 5)h_0(1, 2, 3) + b_{24}b_{57}b_{15}b_{06} + b_{24}b_{57}b_{05}b_{16} + b_{57}b_{14}b_{25}b_{06} + b_{57}b_{14}b_{26}b_{05} + b_{57}b_{25}b_{04}b_{16} + b_{57}b_{04}b_{15}b_{26}=0$

$h_0h_4^2b_{26}b_{17} + h_0h_4^2b_{16}b_{27} + h_0b_{46}b_{25}b_{17} + h_0b_{46}b_{15}b_{27} + h_0b_{25}b_{47}b_{16} + h_0b_{47}b_{15}b_{26} + h_2(1, 2)h_0(1, 2, 4)=0$

$h_4^2b_{13}b_{36}b_{07} + h_4^2b_{13}b_{37}b_{06} + h_4^2b_{03}b_{36}b_{17} + h_4^2b_{03}b_{37}b_{16} + h_0(1, 3, 4)h_0(1, 2, 4) + b_{13}b_{35}b_{46}b_{07} + b_{13}b_{35}b_{47}b_{06} + b_{13}b_{46}b_{37}b_{05} + b_{13}b_{36}b_{47}b_{05} + b_{35}b_{46}b_{03}b_{17} + b_{35}b_{03}b_{47}b_{16} + b_{46}b_{03}b_{15}b_{37} + b_{03}b_{36}b_{47}b_{15}=0$

$h_2^2h_4^2b_{16}b_{07} + h_2^2h_4^2b_{06}b_{17} + h_2^2b_{46}b_{15}b_{07} + h_2^2b_{46}b_{05}b_{17} + h_2^2b_{47}b_{15}b_{06} + h_2^2b_{47}b_{05}b_{16} + h_4^2b_{13}b_{26}b_{07} + h_4^2b_{13}b_{27}b_{06} + h_4^2b_{03}b_{26}b_{17} + h_4^2b_{03}b_{16}b_{27} + h_0(1, 2, 4)^2 + b_{13}b_{46}b_{25}b_{07} + b_{13}b_{46}b_{05}b_{27} + b_{13}b_{25}b_{47}b_{06} + b_{13}b_{47}b_{26}b_{05} + b_{46}b_{03}b_{25}b_{17} + b_{46}b_{03}b_{15}b_{27} + b_{03}b_{25}b_{47}b_{16} + b_{03}b_{47}b_{15}b_{26}=0$

$h_0b_{35}b_{26}b_{17} + h_0b_{35}b_{16}b_{27} + h_0b_{25}b_{36}b_{17} + h_0b_{25}b_{37}b_{16} + h_0b_{36}b_{15}b_{27} + h_0b_{15}b_{26}b_{37} + h_2(1, 2)h_0(1, 2, 3)=0$

$h_3^2h_4^2b_{16}b_{07} + h_3^2h_4^2b_{06}b_{17} + h_3^2b_{46}b_{15}b_{07} + h_3^2b_{46}b_{05}b_{17} + h_3^2b_{47}b_{15}b_{06} + h_3^2b_{47}b_{05}b_{16} + h_4^2b_{14}b_{36}b_{07} + h_4^2b_{14}b_{37}b_{06} + h_4^2b_{36}b_{04}b_{17} + h_4^2b_{04}b_{37}b_{16} + h_0(1, 3, 4)h_0(1, 2, 3) + b_{35}b_{46}b_{14}b_{07} + b_{35}b_{46}b_{04}b_{17} + b_{35}b_{14}b_{47}b_{06} + b_{35}b_{47}b_{04}b_{16} + b_{46}b_{14}b_{37}b_{05} + b_{46}b_{04}b_{15}b_{37} + b_{14}b_{36}b_{47}b_{05} + b_{36}b_{47}b_{04}b_{15}=0$

$h_4^2b_{24}b_{16}b_{07} + h_4^2b_{24}b_{06}b_{17} + h_4^2b_{14}b_{26}b_{07} + h_4^2b_{14}b_{27}b_{06} + h_4^2b_{04}b_{26}b_{17} + h_4^2b_{04}b_{16}b_{27} + h_0(1, 2, 4)h_0(1, 2, 3) + b_{24}b_{46}b_{15}b_{07} + b_{24}b_{46}b_{05}b_{17} + b_{24}b_{47}b_{15}b_{06} + b_{24}b_{47}b_{05}b_{16} + b_{46}b_{14}b_{25}b_{07} + b_{46}b_{14}b_{05}b_{27} + b_{46}b_{25}b_{04}b_{17} + b_{46}b_{04}b_{15}b_{27} + b_{14}b_{25}b_{47}b_{06} + b_{14}b_{47}b_{26}b_{05} + b_{25}b_{47}b_{04}b_{16} + b_{47}b_{04}b_{15}b_{26}=0$

$h_3^2b_{25}b_{16}b_{07} + h_3^2b_{25}b_{06}b_{17} + h_3^2b_{15}b_{26}b_{07} + h_3^2b_{15}b_{27}b_{06} + h_3^2b_{26}b_{05}b_{17} + h_3^2b_{05}b_{16}b_{27} + h_0(1, 2, 3)^2 + b_{24}b_{35}b_{16}b_{07} + b_{24}b_{35}b_{06}b_{17} + b_{24}b_{36}b_{15}b_{07} + b_{24}b_{36}b_{05}b_{17} + b_{24}b_{15}b_{37}b_{06} + b_{24}b_{37}b_{05}b_{16} + b_{35}b_{14}b_{26}b_{07} + b_{35}b_{14}b_{27}b_{06} + b_{35}b_{04}b_{26}b_{17} + b_{35}b_{04}b_{16}b_{27} + b_{14}b_{25}b_{36}b_{07} + b_{14}b_{25}b_{37}b_{06} + b_{14}b_{36}b_{05}b_{27} + b_{14}b_{26}b_{37}b_{05} + b_{25}b_{36}b_{04}b_{17} + b_{25}b_{04}b_{37}b_{16} + b_{36}b_{04}b_{15}b_{27} + b_{04}b_{15}b_{26}b_{37}=0$

\subsubsection*{Relations \ref{rel:80bb77fd}}\hspace{6pt}

$h_1h_3h_0(1, 2) + h_1(1)h_0(1, 3)=0$

$h_1h_3h_0(1, 2, 5) + h_1(1)h_0(1, 3, 5)=0$

$h_1h_3h_0(1, 2, 4) + h_1(1)h_0(1, 3, 4)=0$

$h_3h_5h_0(1, 3, 4) + h_3(1)h_0(1, 3, 5)=0$

$h_3h_5h_0(1, 2, 4) + h_3(1)h_0(1, 2, 5)=0$

$h_1h_3h_5h_0(1, 2, 4) + h_1(1, 3)h_0(1, 3, 5)=0$

$h_5h_1(1)h_0(1, 2, 4) + h_1(1, 3)h_0(1, 2, 5)=0$

$h_1h_3h_5h_0(1, 2, 3) + h_1(1, 2)h_0(1, 3, 5)=0$

$h_5h_1(1)h_0(1, 2, 3) + h_1(1, 2)h_0(1, 2, 5)=0$

$h_1h_3(1)h_0(1, 2, 4) + h_1(1, 3)h_0(1, 3, 4)=0$

$h_1h_3(1)h_0(1, 2, 3) + h_1(1, 2)h_0(1, 3, 4)=0$

$h_1(1, 3)h_0(1, 2, 3) + h_1(1, 2)h_0(1, 2, 4)=0$
\vspace{6pt}

\makebibliography

\end{document}